\documentclass[reqno]{amsart}
\pdfoutput=1
\usepackage{amsmath}
\usepackage{amsfonts}
\usepackage{amssymb}
\usepackage{graphicx}
\usepackage{mathrsfs}%
\providecommand{\U}[1]{\protect\rule{.1in}{.1in}}
\newtheorem{theorem}{Theorem}

\newtheorem{conjecture}[theorem]{Conjecture}
\newtheorem{corollary}[theorem]{Corollary}

\newtheorem{definition}[theorem]{Definition}

\newtheorem{lemma}[theorem]{Lemma}
\newtheorem{notation}[theorem]{Notation}

\newtheorem{proposition}[theorem]{Proposition}
\newtheorem{remark}[theorem]{Remark}

\numberwithin{theorem}{section}
\numberwithin{equation}{section}
\begin{document}

\title{The Brown measure of the free multiplicative Brownian motion}
\author{Bruce K.\ Driver}
\address{Department of Mathematics, University of California San
Diego, La Jolla, CA 92093, USA}
\email{bdriver@ucsd.edu}
\author{ Brian C.\ Hall}
\address{Department of Mathematics, University of Notre Dame, Notre Dame,
IN 46556, USA}
\email{bhall@nd.edu}
\author{Todd Kemp}
\thanks{Kemp's research supported in part by NSF CAREER Award DMS-1254807 and NSF Award DMS-1800733}
\address{Department of Mathematics, University of California San Diego, La Jolla, CA 92093, USA}
\email{tkemp@math.ucsd.edu}

\begin{abstract}
The free multiplicative Brownian motion $b_{t}$ is the large-$N$ limit of the
Brownian motion on $\mathsf{GL}(N;\mathbb{C}),$ in the sense of $\ast
$-distributions. The natural candidate for the large-$N$ limit of the
empirical distribution of eigenvalues is thus the Brown measure of $b_{t}$. In
previous work, the second and third authors showed that this Brown measure is
supported in the closure of a region $\Sigma_{t}$ that appeared in work of Biane.
In the present paper, we compute the Brown measure completely. It has a
continuous density $W_{t}$ on $\overline{\Sigma}_{t},$ which is strictly
positive and real analytic on $\Sigma_{t}$. This density has a simple form in
polar coordinates:
\[
W_{t}(r,\theta)=\frac{1}{r^{2}}w_{t}(\theta),
\]
where $w_{t}$ is an analytic function determined by the geometry of the region
$\Sigma_{t}$.

We show also that the spectral measure of free unitary Brownian motion $u_{t}$
is a \textquotedblleft shadow\textquotedblright\ of the Brown measure of
$b_{t}$, precisely mirroring the relationship between Wigner's semicircle law
and Ginibre's circular law. We develop several new methods, based on
stochastic differential equations and PDE, to prove these results.

\end{abstract}

\maketitle

\tableofcontents

\section{Introduction}

\subsection{$\ast$-Distribution and Brown measure\label{starAndBrown.sec}}

Two of the most fundamental models in random matrix theory are the
\emph{Gaussian unitary ensemble} and the \emph{Ginibre ensemble}. Each is a
Gaussian random matrix; the Ginibre ensemble has i.i.d.\ complex Gaussian
entries, with variance $1/N$, while the Gaussian unitary ensemble is the
Hermitian part of the Ginibre ensemble. For our purposes, it is natural to
think of these two ensembles as endpoints of \emph{Brownian motion on Lie
algebras}. Indeed, the Lie algebra $\mathrm{gl}(N;\mathbb{C})$ of the general
linear group consists of all $N\times N$ complex matrices. For each fixed $t,$
the Brownian motion $Z_{t}$ on this space (with an appropriately scaled time
parameter) is distributed as the Ginibre ensemble, scaled by $\sqrt{t}.$
Similarly, the space of Hermitian matrices is equal to $i\mathrm{u}(N)$, where
$\mathrm{u}(N)$ is the Lie algebra of the unitary group, namely the space of
skew-Hermitian matrices. For each fixed $t,$ the Brownian motion $X_{t}$ on
this space is distributed as the Gaussian unitary ensemble, scaled by
$\sqrt{t}.$

Among the earliest results in random matrix theory are the discovery of the
large-$N$ limits of the empirical eigenvalue distributions of these ensembles.
That is to say, the (random) counting measure $\frac{1}{N}\sum_{j=1}^{N}%
\delta_{\lambda_{j}}$ of the eigenvalues $\{\lambda_{j}\}$ of each ensemble
has an almost-sure limit which is a deterministic measure. The eigenvalues of
$X_{t}$ are real, and their limit empirical eigenvalue distribution is the
semicircle law on the interval $[-2\sqrt{t},2\sqrt{t}]$ (cf.\ \cite{Wigner});
the eigenvalues of $Z_{t}$ are complex, and their limit empirical eigenvalue
distribution is the uniform probability measure on the disk of radius
$\sqrt{t}$ (cf.\ \cite{Gin}). We note for later reference a simple but
intriguing link between these two limiting distributions: the push-forward
under the \textquotedblleft real part\textquotedblright\ map of the uniform
measure on the disk is the semicircular distribution on an interval.

It is convenient to recast the empirical eigenvalue distribution in an
analytic form. For an $N\times N$ matrix $A=A_{N}$, the function%
\begin{align}
L_{A}(\lambda) &  =\frac{1}{2N}\log(\det[(A-\lambda)^{\ast}(A-\lambda
)])\nonumber\\
&  =\frac{1}{N}\sum_{j=1}^{N}\log(\left\vert \lambda-\lambda_{j}\right\vert
)\label{LaN}%
\end{align}
is subharmonic on $\mathbb{C}$. Actually, $L_{A}$ is harmonic on the
complement of the spectrum, and $-\infty$ at the eigenvalues. The
(distributional) Laplacian of $L$ is therefore a positive measure; in fact it
is (up to a factor of $2\pi$) equal to the empirical eigenvalue distribution
of the matrix $A$. Hence, if one can compute an appropriate \textquotedblleft
large-$N$ limit\textquotedblright\ of the functions $L_{A_{N}}$, the Laplacian
of the limiting function will provide a natural candidate for the limiting
empirical eigenvalue distribution.

Free probability theory affords a medium in which to identify abstract limits
of random matrix ensembles themselves. The limits are constructed as operators
in a tracial von Neumann algebra $(\mathcal{A},\tau)$, and the limit is with
respect to \emph{$\ast$-distribution}. If $A_{N}$ is a sequence of $N\times N$
random matrices, an operator $a\in\mathcal{A}$ is said to be a \textbf{limit
in $\ast$-distribution} of $A_{N}$ if, for each polynomial $p$ in two
noncommuting variables, we have
\[
\lim_{N\rightarrow\infty}\frac{1}{N}\mathrm{trace}[p(A_{N},A_{N}^{\ast}%
)]=\tau\lbrack p(a,a^{\ast})]
\]
almost surely. For $X_{t}$ and $Z_{t}$, Voiculescu \cite{Voiculescu} showed
that the large-$N$ limits can be identified as certain \emph{free stochastic
processes}, namely the \textbf{free additive Brownian motion} $x_{t}$ and the
\textbf{free circular Brownian motion} $c_{t}.$ These limiting process are no
longer random: they are one-parameter families of operators with freely
independent increments, and $\ast$-distributions that can be described
elegantly in the combinatorial framework of free probability.

In particular, $x_{t}$ is a self-adjoint operator, and it therefore has a
spectral resolution $E^{x_{t}}$: a projection-value measure so that
$x_{t}=\int_{\mathbb{R}}\lambda\,E^{x_{t}}(d\lambda)$. The composition
$\mu_{x_{t}}:=\tau\circ E^{x_{t}}$ is a probability measure on $\mathbb{R}$
called the \emph{spectral measure} of $x_{t}$ (in the state $\tau$). The
statement that $X_{t}$ has the semicircle law as its limit empirical
eigenvalue distribution is equivalent to the statement that $\mu_{x_{t}}$ is semicircular.

On the other hand, since the operator $c_{t}$ is not normal, there is no
spectral theorem to yield a spectral measure. There is, however, a substitute:
the Brown measure, introduced in \cite{Br}. For any operator $a\in
\mathcal{A},$ define a function $L_{a}$ on $\mathbb{C}$ by%
\begin{equation}
L_{a}(\lambda)=\tau\lbrack\log(\left\vert a-\lambda\right\vert
)],\label{e.FKdet}%
\end{equation}
where $|a-\lambda|$ is the self-adjoint operator $((a-\lambda)^{\ast
}(a-\lambda))^{1/2}$. The quantity $L_{a}(\lambda)$ is the logarithm of the
Fuglede--Kadison determinant \cite{FK1,FK2} of $a-\lambda.$ It is finite
outside the spectrum of $a$ but may become $-\infty$ at points in the
spectrum. If $\mathcal{A}$ is the space of all $N\times N$ matrices and
$\tau=\frac{1}{N}\mathrm{trace},$ then $L_{a}$ agrees with the function in
(\ref{LaN}). In a general tracial von Neumann algebra $(\mathcal{A},\tau)$,
the function $L_{a}$ is subharmonic.

The \textbf{Brown measure} of $a$ is then defined in terms of the
distributional Laplacian of $L_{a}$:
\[
\mu_{a}:=\frac{1}{2\pi}\Delta L_{a}.
\]
If $a$ is self-adjoint, the Brown measure of $a$ coincides with the spectral
measure. If $a=c_{t}$ is the free circular Brownian motion at time $t$, its
Brown measure $\mu_{c_{t}}$ is equal to the uniform probability measure on the
disk of radius $\sqrt{t}$.

By regularizing the right-hand side of (\ref{e.FKdet}), one can construct the
Brown measure $\mu_{a}$ as a weak limit,
\begin{equation}
d\mu_{a}(\lambda)=\frac{1}{4\pi}%
\Delta_{\lambda}\lim_{\varepsilon\rightarrow0^{+}}\tau\lbrack\log[(a-\lambda)^{\ast}(a-\lambda)+\varepsilon
)]~d\lambda, \label{weakLim}%
\end{equation}
where $\Delta_{\lambda}$ is the Laplacian with respect to $\lambda$ and
$d\lambda$ is the Lebesgue measure on the plane. (See \cite[Section 11.5]{MS}
and \cite[Eq.\ (2.11)]{HK}.) It is not hard to see that the Brown measure of
$a$ is determined by the $\ast$-moments of $a,$ but the dependence is
singular: if a sequence of operators $a_{n}$ converges in $\ast$-distribution
to $a$, the Brown measures $\mu_{a_{n}}$ need not converge to $\mu_{a}$.

\subsection{Brownian motions on $\mathsf{U}(N)$, $\mathsf{GL}(N;\mathbb{C})$,
and their large-$N$ limits}

As we have noted, the Gaussian unitary ensemble and the Ginibre ensemble can
be described in terms of Brownian motions on the Lie algebras $\mathrm{u}(N)$
and $\mathrm{gl}(N;\mathbb{C})$. Specifically, the Brownian motions are
induced by a choice of inner product on these finite-dimensional vector
spaces; in both cases, we use the inner product
\[
\left\langle X,Y\right\rangle _{N}=N\operatorname{Re}\mathrm{trace}(X^{\ast
}Y).
\]
(The factor of $N$ in the definition produces the scaling of $1/N$ in the
variances of the two ensembles.) It is natural to consider the counterpart
\emph{Brownian motions on the Lie groups} $\mathsf{U}(N)$ and $\mathsf{GL}%
(N;\mathbb{C})$. In general, if $G\subset\mathsf{M}_{N}(\mathbb{C})$ is a
matrix Lie group with Lie algebra $\mathfrak{g}\subset\mathsf{M}%
_{N}(\mathbb{C})$, there is a simple relationship between the Brownian motion
$B_{t}$ on $G$ and the Brownian motion $A_{t}$ on $\mathfrak{g}$ (the latter
being the standard Brownian motion determined by an inner product on
$\mathfrak{g}$). It is known as the \emph{rolling map}, and it can be written
as a Stratonovitch stochastic differential equation (SDE):
\[
dB_{t}=B_{t}\circ dA_{t},\qquad B_{0}=I.
\]

The solution of this SDE is a diffusion process on $G$ whose increments
(computed in the left multiplicative sense) are independent and whose
generator is half the Laplacian on $G$ determined by the left-invariant
Riemannian metric induced by the given inner product on $\mathfrak{g}$. Thus,
its distribution at each time is the heat kernel on the group. For
computational purposes, it is useful to write the SDE in It\^{o} form; the
result depends on the structure of the group. Letting $U_{t}=U_{t}^{N}$ denote
the Brownian motion on $\mathsf{U}(N)$ and $B_{t}=B_{t}^{N}$ the Brownian
motion on $\mathsf{GL}(N;\mathbb{C})$, the corresponding It\^{o} SDEs are
\begin{align*}
dU_{t}  &  =iU_{t}\,dX_{t}-\textstyle{\frac{1}{2}}U_{t}\,dt\\
dB_{t}  &  =B_{t}\,dZ_{t}.
\end{align*}

It is then natural to investigate the large-$N$ limits of these random matrix
processes. The candidate large-$N$ limits are the \emph{free} stochastic
processes generated by the analogous \emph{free} SDEs:
\begin{align}
du_{t}  &  =iu_{t}\,ds_{t}-\textstyle{\frac{1}{2}}u_{t}\,dt\label{SDE.f.bt}\\
db_{t}  &  =b_{t}\,dc_{t}.
\end{align}
(For the theory of free stochastic calculus, see \cite{BS1,BS2, KNPS}.) In
1997, Biane \cite{BianeFields,BianeJFA} introduced these processes, the
\textbf{free unitary Brownian motion} $u_{t}$, and the \textbf{free
multiplicative Brownian motion} $b_{t}$. (He called $b_{t}=\Lambda_{t}$, and
wrote a slightly different but equivalent free SDE for it.) The main result of
\cite{BianeFields} was the theorem that the process $u_{t}$ is indeed the
large-$N$ limit in $\ast$-distribution of the unitary Brownian motion
$U_{t}=U_{t}^{N}$. Since $U_{t}$ and $u_{t}$ are unitary (hence normal)
operators, this also means that the empirical eigenvalue distribution of
$U_{t}^{N}$ converges to the spectral measure of $u_{t}$.

Biane also computed the spectral measure $\nu_{t}$ of $u_{t}.$ We now record
this result, since it relates closely to the results of the present paper. Let
$f_{t}$ denote the holomorphic function on $\mathbb{C}\setminus\{1\}$ defined
by
\begin{equation}
f_{t}(\lambda)=\lambda e^{\frac{t}{2}\frac{1+\lambda}{1-\lambda}}.
\label{ftDef}%
\end{equation}
Then $f_{t}$ has a holomorphic inverse $\chi_{t}$ in the open unit disk, and
$\chi_{t}$ extends continuously to the closed unit disk. Biane showed that
\[
\chi_{t}=\frac{\psi_{u_{t}}}{1+\psi_{u_{t}}}%
\]
where $\psi_{u_{t}}(z)=\tau\lbrack(1-zu_{t})^{-1}]-1$ is the (recentered)
moment-generating function of $u_{t}$. From this (and other SDE computations)
he determined the following result.

\begin{theorem}
[Biane, 1997 \cite{BianeFields,BianeJFA}]\label{t.Biane.ut} The spectral
measure $\nu_{t}$ of the free unitary Brownian motion $u_{t}$ is supported in
the arc
\[
\left\{  \left.  e^{i\phi}\right\vert ~|\phi|<\phi_{\mathrm{max}}(t):=\frac
{1}{2}\sqrt{(4-t)t}+\cos^{-1}(1-t/2)\right\}
\]
for $t<4$, and is fully supported on the circle for $t\geq4$. The measure
$\nu_{t}$ has a continuous density $\kappa_{t}$, which is real analytic on the
interior of its support arc, given by
\[
\kappa_{t}(e^{i\phi})=\frac{1}{2\pi}\frac{1-\left\vert \chi_{t}(e^{i\phi
})\right\vert ^{2}}{\left\vert 1-\chi_{t}(e^{i\phi})\right\vert ^{2}}.
\]

\end{theorem}

See, for example, p. 275 in \cite{BianeJFA}. In the same papers
\cite{BianeFields,BianeJFA} in which he introduced $u_{t}$, Biane considered
the free multiplicative Brownian motion $b_{t}$ as well (for example,
computing its norm). He conjectured that it should be the large-$N$ limit of
the Brownian motion $B_{t}=B_{t}^{N}$ on $\mathsf{GL}(N;\mathbb{C})$, in
$\ast$-distribution. This was proved by Kemp in \cite{KempLargeN}, with
complementary estimates of moments given in \cite{KempHeatKernel}. The goal of
the present paper is to fully determine the Brown measure of the free
multiplicative Brownian motion $b_{t}$, giving the full complex analog of
Theorem \ref{t.Biane.ut}.

\subsection{The Brown measure of $b_{t}$\label{BrownOfbt}}

The main result of this paper is a formula for the Brown measure $\mu_{b_{t}}$
of the free multiplicative Brownian motion $b_{t}.$ We \emph{expect} that
$\mu_{b_{t}}$ coincides with the large-$N$ limit of the empirical eigenvalue
distribution of the Brownian motion $B_{t}^{N}$ on $\mathsf{GL}(N;\mathbb{C}%
)$. Techniques that can be used to prove results of this sort---that a
limiting eigenvalue distribution agrees with a Brown measure---have been
developed in the context of the general circular law, in which the entries are
independent and identically distributed but not necessarily Gaussian. Analysis
of this model began with the work of Girko \cite{Girko} and continued with
results of Bai \cite{Bai}, G\"{o}tze and Tikhomirov \cite{GT}, and ending with
the definitive version of the circular law established by Tao and Vu
\cite{TV}. All of these works compute the empirical eigenvalue distribution by
taking the Laplacian of the quantity $L_{A}$ in (\ref{LaN}). They then
consider the  limiting eigenvalue distribution of $(A-\lambda)^{\ast
}(A-\lambda)$ for each $\lambda,$ from which they compute (essentially) the
Brown measure of the limiting object, which in this case is uniform on a disk.
Then to prove convergence of the eigenvalue distribution of the random matrices
to this uniform measure, they develop techniques for controlling the
singularities of the logarithm function in (\ref{LaN}) at infinity and (especially) at zero. %

\begin{figure}[ptb]%
\centering
\includegraphics[
height=2.5365in,
width=4.8282in
]%
{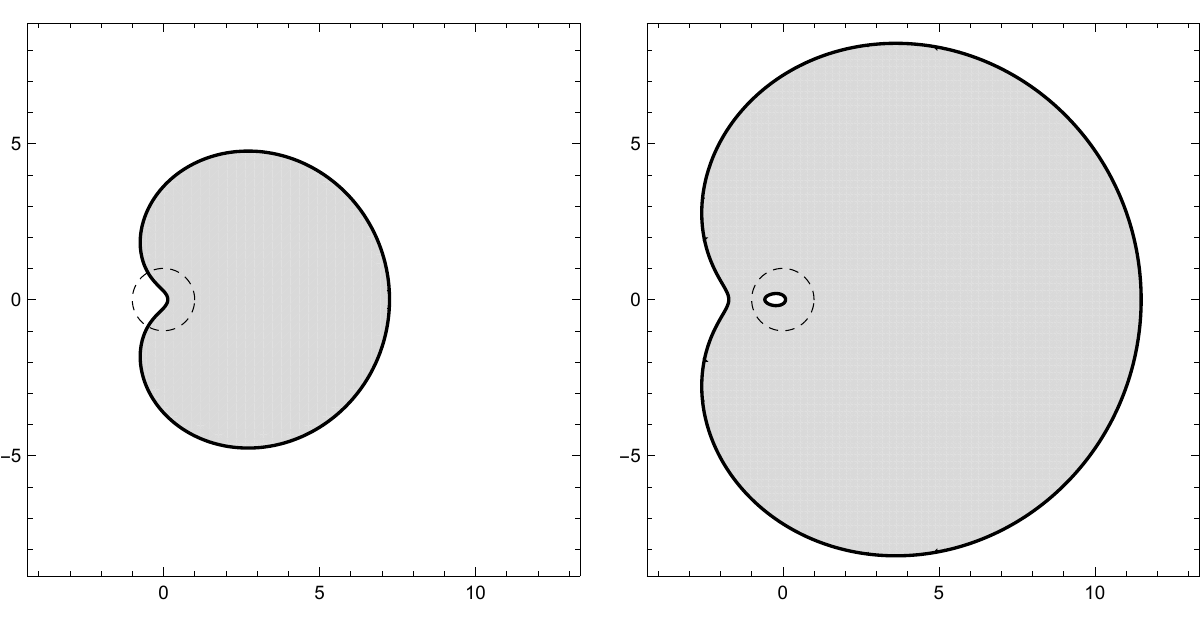}%
\caption{The regions $\Sigma_{t}$ with $t=3$ and $t=4.1$ with the unit circle
(dashed) indicated for comparison}%
\label{regions.fig}%
\end{figure}
%

\begin{figure}[ptb]%
\centering
\includegraphics[
height=2.6195in,
width=4.9882in
]%
{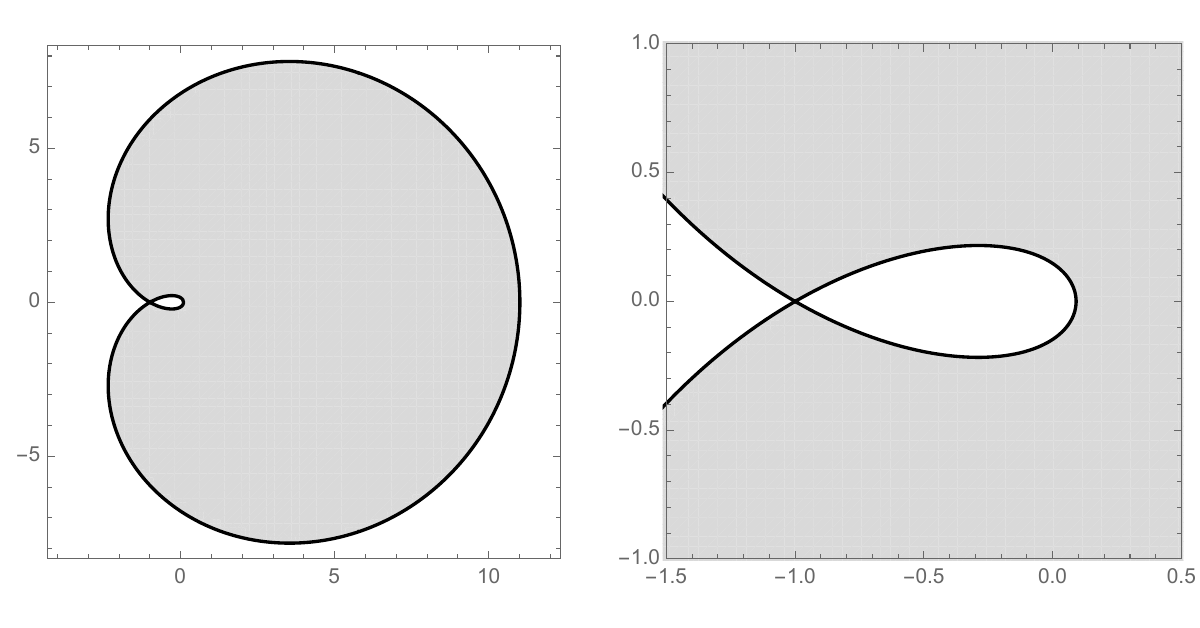}%
\caption{The region $\Sigma_{t}$ with $t=4$ (left) and a detail thereof
(right)}%
\label{t4region.fig}%
\end{figure}

A previous result \cite{HK} of the second and third authors showed that
$\mu_{b_{t}}$ is supported on the closure of a certain region $\Sigma_{t}$
introduced by Biane in \cite{BianeJFA}; see Figures \ref{regions.fig} and
\ref{t4region.fig}. (We reprove that result in the present paper by a
different method; see Theorem \ref{outside.thm} in Section \ref{outside.sec}.)
The proof in \cite{HK} is based on Biane's \textquotedblleft free Hall
transform\textquotedblright\ $\mathscr{G}_{t},$ introduced in\ \cite{BianeJFA}%
. This transform was conjectured by Biane to be the the large-$N$ limit of the
generalized Segal--Bargmann transform of Hall \cite{Ha1994,Hall2001}, a
conjecture that was verified independently by C\'{e}bron \cite{Ceb} and the
authors of the present paper \cite{DHKLargeN}. A key idea in Biane's work is
Gross and Malliavin's probabilistic interpretation \cite{GM} of the transform
in \cite{Ha1994}.

For each $\lambda$ outside $\overline{\Sigma}_{t},$ \cite{HK} uses
$\mathscr{G}_{t}$ to construct an \textquotedblleft inverse\textquotedblright%
\ of $b_{t}-\lambda$ and $(b_{t}-\lambda)^{2}.$ These inverses are not
necessarily bounded operators, but live in the noncommutative $L^{2}$ space,
that is, the completion of $\mathcal{A}$ with respect to the inner product
$\left\langle a,b\right\rangle :=\tau(a^{\ast}b)$. We then strengthen the
standard result that the Brown measure is supported on the spectrum of the
operator to show that existence of an $L^{2}$ inverse of $(b_{t}-\lambda)^{2}$
guarantees that $\lambda$ is outside the support of $\mu_{b_{t}}.$ We note,
however, that the methods of \cite{HK} do not give any information about the
distribution of the Brown measure $\mu_{b_{t}}$ inside the region $\Sigma
_{t}.$

\subsection{Connection to Physics}

The eigenvalue distribution of Brownian motion in $\mathsf{GL}(N;\mathbb{C})$,
in the large-$N$ limit, has been studied in the physics literature, first by
Gudowska-Nowak, Janik, Jurkiewicz, and Nowak \cite{Nowak} and then by
Lohmayer, Neuberger, and Wettig \cite{Lohmayer}. At least in the case of
\cite{Lohmayer}, the motivation for considering this model is a connection to
two-dimensional Yang--Mills theory. Yang--Mills quantum field theory is a key
part of the standard model in particle physics and the two-dimensional case
can be treated in a mathematically rigorous fashion. Two-dimensional
Yang--Mills theory is a much-studied model, in part as a toy model of the
four-dimensional theory and in part because of its connections to string
theory \cite{DGross,GrossTaylor}.

Yang--Mills theory with structure group $G$ describes a random connection on a
principal $G$-bundle. One typically studies the theory through the associated
\textit{Wilson loop functionals}, given by the expectation value of the trace
of the holonomy around a loop. Assume at first that $G$ is compact, that the
spacetime manifold is the plane, and that the loop is a simple closed curve in
the plane. Then the distribution of the holonomy is described by Brownian
motion in $G,$ with time-parameter proportional to the area enclosed by the
curve. (See works of Driver \cite{Driver89} and Gross--King--Sengupta
\cite{GKS} and the references therein.)

Of particular importance is the case $G=\mathsf{U}(N),$ with $N$ tending to
infinity; the resulting theory is called the \textquotedblleft master
field.\textquotedblright\ The master field in the plane is therefore built
around the large-$N$ limit of Brownian motion in $\mathsf{U}(N),$ i.e., the
free unitary Brownian motion $u_{t}.$ (See works of Singer \cite{Singer},
Anshelevich--Sengupta \cite{AS}, and L\'{e}vy \cite{Levy}.) In this context,
the change in behavior of $u_{t}$ at $t=4$---in which the support of the
spectral measure wraps all the way around the circle---is called a
\textquotedblleft topological phase transition.\textquotedblright\ (See also
\cite{DHK-MM1}, \cite{DGHK-MM2}, \cite{DN}, and \cite{HallS2} for recent
progress constructing a rigorous large-$N$ Yang--Mills theory on surfaces
other than the plane.)

Although Yang--Mills theory is typically constructed when the structure group
$G$ is compact, it requires only a small step of imagination to consider also
the case $G=\mathsf{GL}(N;\mathbb{C}).$ Thus, in \cite{Lohmayer}, Brownian
motion in $\mathsf{GL}(N;\mathbb{C})$ is considered as a sort of
\textquotedblleft complex Wilson loop\textquotedblright\ computation. It is of
interest to determine whether the large-$N$ limit---namely the free
multiplicative Brownian motion $b_{t}$---still has a topological phase
transition at $t=4.$

The papers \cite{Nowak} and \cite{Lohmayer} both derive, using nonrigorous
methods, the region into which the eigenvalues of $B_{t}^{N}$ cluster in the
large-$N$ limit. Both papers find this domain to be precisely the region
$\Sigma_{t}$ considered in the present work. Since $\Sigma_{t}$ wraps around
the origin precisely at time $t=4,$ the authors conclude that indeed the
topological phase transition persists after the change from $\mathsf{U}(N)$ to
$\mathsf{GL}(N;\mathbb{C}).$ The paper \cite{Lohmayer} also considers a
two-parameter extension of the Brownian motion of the sort considered in
\cite{DHKLargeN,Ho,KempLargeN}, and finds that the eigenvalues cluster into
the domain denoted $\Sigma_{s,t}$ in \cite{Ho}. A rigorous version of these
results---specifically, that the Brown measure of the relevant free Brownian
motion is supported in $\overline{\Sigma}_{t}$ or $\overline{\Sigma}_{s,t}%
$---was then obtained by the second and third authors in \cite{HK}.

We emphasize that the papers \cite{Nowak}, \cite{Lohmayer}, and \cite{HK} are
concerned only with the \textit{region} into which the eigenvalues cluster.
Nothing is said there about how the eigenvalues are \textit{distributed} in
the region. By contrast, in the present work, we not only prove (again) that
the Brown measure of $b_{t}$ is supported in $\overline{\Sigma}_{t},$ we
actually \textit{compute} the Brown measure (Theorem \ref{main.thm}).
Furthermore, we not only see the same transition at $t=4$ for the
$\mathsf{GL}(N;\mathbb{C})$ case as for the $\mathsf{U}(N)$ case, we actually
find a \textit{direct connection} (Proposition \ref{connectToBiane2.prop})
between the Brown measure of $b_{t}$ and the spectral measure of $u_{t}.$

\subsection{Subsequent work}

Since the first version of this paper appeared on the arXiv, three subsequent
works have appeared that use the techniques developed here to analyze Brown
measures of other operators. First, work of Ho and Zhong \cite{HZ} has
extended the results of the present paper to the case of a free multiplicative
Brownian motion with an arbitrary unitary initial condition. This means that
they compute the Brown measure of $ub_{t},$ where $u$ is a unitary element
freely independent of $b_{t}.$ Ho and Zhong also compute the Brown measure of
$x_{0}+c_{t},$ where $c_{t}$ is a free circular Brownian motion and $x_{0}$ is
a self-adjoint element freely independent of $c_{t}.$ Second, Demni and Hamdi
\cite{DH} have analyzed the support of the Brown measure of $u_{t}P,$ where
$u_{t}$ is the free unitary Brownian motion and $P$ is a projection freely
independent of $u_{t}.$ Last, Hall and Ho \cite{HH} have computed the Brown
measure of $x_{0}+ix_{t},$ where $x_{t}$ is the free additive Brownian motion
and $x_{0}$ is a self-adjoint element freely independent of $x_{t}.$

The reader may also consult the expository article \cite{PDEmethods} by the
second author, which provides a nontechnical introduction to the techniques
used in the present paper.

\section{Statement of main results}

\subsection{A formula for the Brown measure}

In this paper, we compute the Brown measure $\mu_{b_{t}}$ of the free
multiplicative Brownian motion $b_{t},$ using completely different methods
from those in \cite{HK}. To state our main result, we need to briefly describe
the regions $\Sigma_{t}.$ For each $t>0,$ consider the holomorphic function
$f_{t}$ on $\mathbb{C}\setminus\{1\}$ defined by (\ref{ftDef}). It is easily
verified that if $\left\vert \lambda\right\vert =1$ then $\left\vert
f_{t}(\lambda)\right\vert =1.$ There are, however, \textit{other} points where
$\left\vert f_{t}(\lambda)\right\vert =1.$ We then define%
\begin{equation}
F_{t}=\left\{  \lambda\in\mathbb{C}\left\vert \left\vert \lambda\right\vert
\neq1,~\left\vert f_{t}(\lambda)\right\vert =1\right.  \right\}  \label{FtDef}%
\end{equation}
and%
\begin{equation}
E_{t}=\overline{F}_{t}. \label{EtDef}%
\end{equation}

\begin{definition}
\label{SigmaT.def}For each $t>0,$ we define $\Sigma_{t}$ to be the connected
component of the complement of $E_{t}$ containing 1.
\end{definition}

We will show (Theorem \ref{domainGobbles.thm}) that $\Sigma_{t}$ may also be
characterized as%
\begin{equation}
\Sigma_{t}=\left\{  \left.  \lambda\in\mathbb{C}\right\vert T(\lambda
)<t\right\}  , \label{SigmaTwithT}%
\end{equation}
where the function $T$ is defined in (\ref{Tlambda}). Each region $\Sigma_{t}$
is invariant under the maps $\lambda\mapsto1/\lambda$ and $\lambda\mapsto
\bar{\lambda}.$ If we consider a ray from the origin with angle $\theta,$ if
this ray intersects $\Sigma_{t}$ at all, it does so in an interval of the form
$1/r_{t}(\theta)<r<r_{t}(\theta)$ for some $r_{t}(\theta)>1.$ (See Figures
\ref{r1r2.fig} and \ref{rplots.fig}.) See Section \ref{regionPropeties.sec}
for more information.%

\begin{figure}[ptb]%
\centering
\includegraphics[
height=2.5979in,
width=2.5278in
]%
{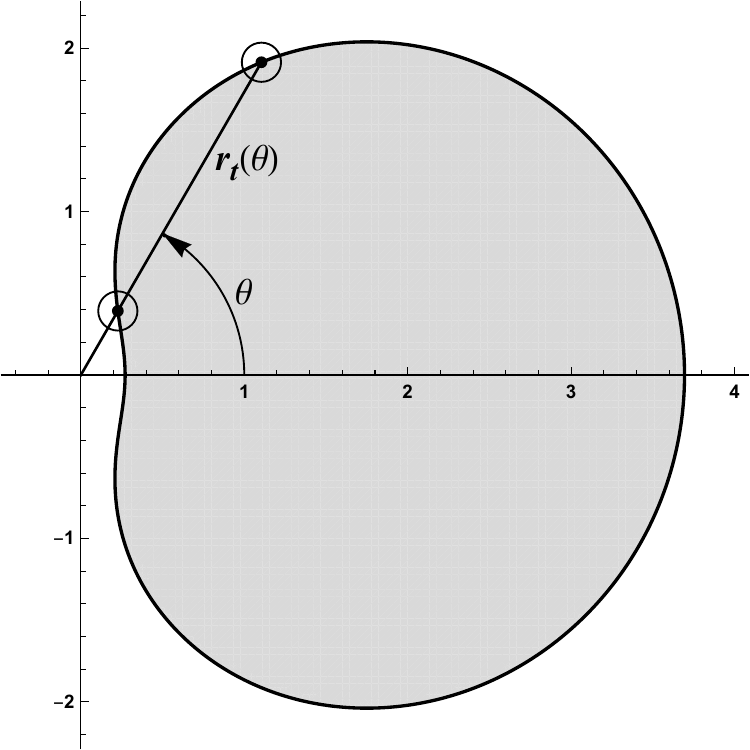}%
\caption{We let $r_{t}(\theta)$ denote the larger of the two radii where the
ray with angle $\theta$ intersects $\partial\Sigma_{t}.$ Shown for $t=1.5$}%
\label{r1r2.fig}%
\end{figure}

We are now ready to state our main result.

\begin{theorem}
\label{main.thm}For all $t>0,$ the Brown measure $\mu_{b_{t}}$ of $b_{t}$ is
absolutely continuous with respect to the Lebesgue measure on the plane and
supported in the domain $\overline{\Sigma}_{t}$. In $\Sigma_{t},$ the density
$W_{t}$ of $\mu_{b_{t}}$ with respect to the Lebesgue measure is strictly
positive and real analytic, with the following form in polar coordinates:%
\begin{equation}
W_{t}(r,\theta)=\frac{1}{r^{2}}w_{t}(\theta) \label{WtFormula}%
\end{equation}
for a certain even function $w_{t}.$ This function may be computed as%
\begin{equation}
w_{t}(\theta)=\frac{1}{4\pi}\left(  \frac{2}{t}+\frac{\partial}{\partial
\theta}\frac{2r_{t}(\theta)\sin\theta}{r_{t}(\theta)^{2}+1-2r_{t}(\theta
)\cos\theta}\right)  \label{wtAnswer}%
\end{equation}
where $r_{t}(\theta)$ is the larger of the two radii where the ray with angle
$\theta$ intersects the boundary of $\Sigma_{t}.$
\end{theorem}

Since $\Sigma_{t}$ is invariant under $\lambda\mapsto\bar{\lambda},$ the
function $r_{t}(\theta)$ is an even function of $\theta$, from which it is
easy to check that the second term on the right-hand side of of
(\ref{wtAnswer}) is also an even function of $\theta.$ Although we will
customarily let $r_{t}(\theta)$ denote the \textit{larger} of the the two
radii, we note that
\begin{equation}
r\mapsto\frac{2r\sin\theta}{r^{2}+1-2r\cos\theta} \label{dsdtheta}%
\end{equation}
is invariant under $r\mapsto1/r.$ Thus, the value of $w_{t}$ does not actually
depend on which radius is used. It is noteworthy that the one nonexplicit part
of the formula for $w_{t},$ namely the second term on the right-hand side of
(\ref{wtAnswer}), is computable entirely in terms of the geometry of the
region $\Sigma_{t}.$ According to Proposition \ref{dPhiDtheta.prop}, $w_{t}$
can also be computed as a logarithmic derivative along the boundary of
$\Sigma_{t}$ of the function $f_{t}$ in (\ref{ftDef}).%

\begin{figure}[ptb]%
\centering
\includegraphics[
height=2.9948in,
width=4.8282in
]%
{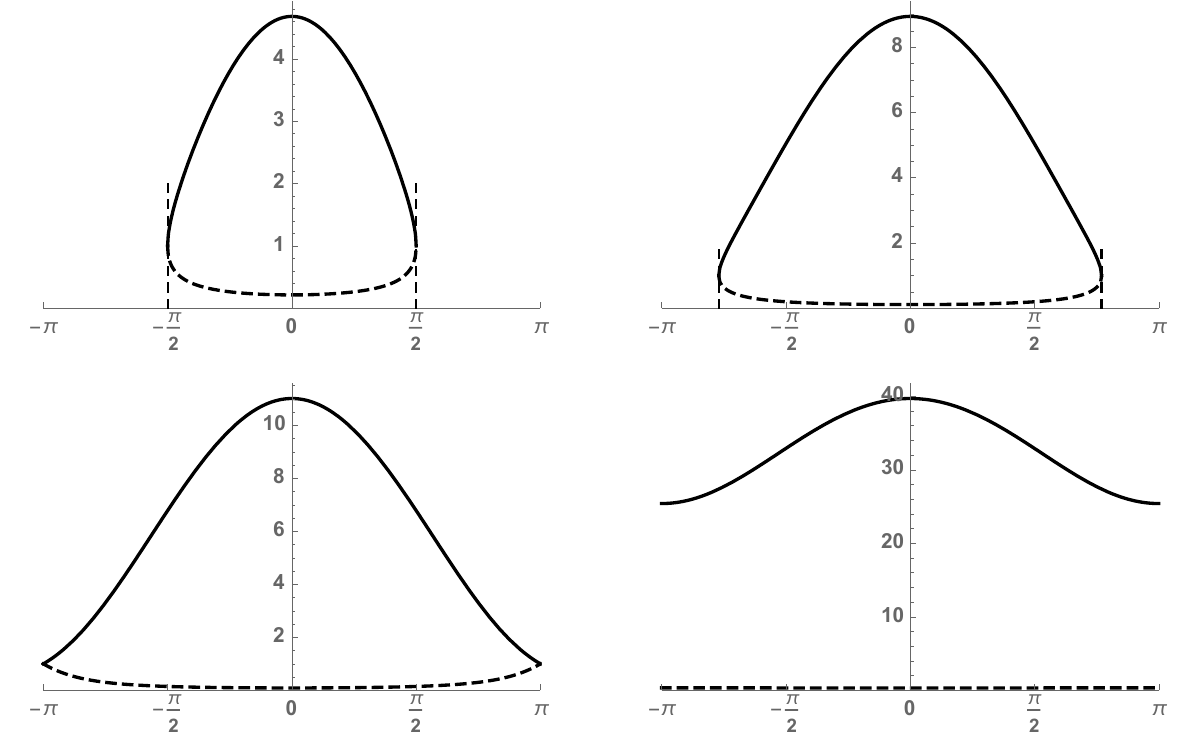}%
\caption{Graphs of $r_{t}(\theta)$ (black) and $1/r_{t}(\theta)$ (dashed) for
$t=2,$ 3.5, 4, and 7}%
\label{rplots.fig}%
\end{figure}

It follows from (\ref{SigmaTwithT}) that the function $T$ equals $t$ on the
boundary of $\Sigma_{t}.$ It is then possible to use implicit differentiation
in the equation $T(\lambda)=t$ to compute $dr_{t}(\theta)/d\theta$ as a
function of $r_{t}(\theta)$ and $\theta.$ We may then use this computation to
rewrite (\ref{wtAnswer}) in a form that no longer involves a derivative with
respect to $\theta,$ as follows.

\begin{proposition}
\label{omega.prop}The function $w_{t}$ in Theorem \ref{main.thm} may also be
computed in the form%
\[
w_{t}(\theta)=\frac{1}{2\pi t}\omega(r_{t}(\theta),\theta).
\]
Here%
\begin{equation}
\omega(r,\theta)=1+h(r)\frac{\alpha(r)\cos\theta+\beta(r)}{\beta(r)\cos
\theta+\alpha(r)}, \label{omegaFormula}%
\end{equation}
where%
\[
h(r)=r\frac{\log(r^{2})}{r^{2}-1};\quad\alpha(r)=r^{2}+1-2rh(r);\quad
\beta(r)=(r^{2}+1)h(r)-2r.
\]

\end{proposition}

Thus, to compute $w_{t}(\theta),$ we evaluate $\omega/(2\pi t)$ on the
boundary of $\Sigma_{t}$ and then parametrize the boundary by the angle
$\theta$; see Figure \ref{omegaplot.fig}. Using Proposition \ref{omega.prop},
we can derive small- and large-$t$ asymptotics of $w_{t}(\theta)$ as follows:%
\begin{align*}
w_{t}(\theta)  &  \sim\frac{1}{\pi t},\quad t\text{ small;}\\
w_{t}(\theta)  &  \sim\frac{1}{2\pi t},\quad t\text{ large.}%
\end{align*}
See Section \ref{omega.sec} for details.%

\begin{figure}[ptb]%
\centering
\includegraphics[
height=2.015in,
width=3.7905in
]%
{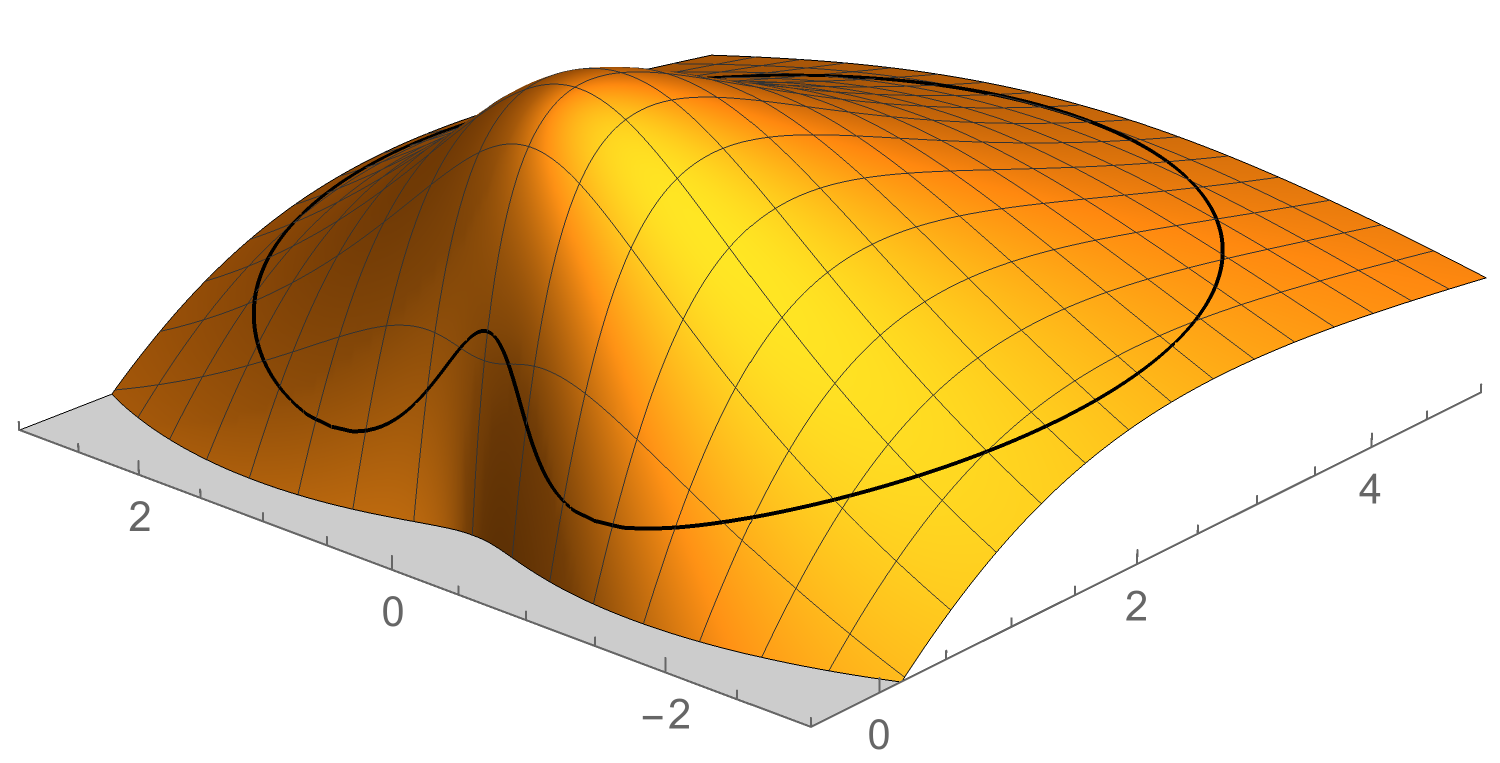}%
\caption{The function $w_{t}(\theta)$ is computed by evaluating $\omega$ on
the boundary of $\Sigma_{t}$ and parametrizing the boundary by the angle
$\theta.$ Shown for $t=2$}%
\label{omegaplot.fig}%
\end{figure}

The following simple consequences of Theorem \ref{main.thm} helps explain the
significance of the factor of $1/r^{2}$ in the formula (\ref{WtFormula}) for
$W_{t}$.

%

\begin{figure}[ptb]%
\centering
\includegraphics[
height=2.9948in,
width=4.8282in
]%
{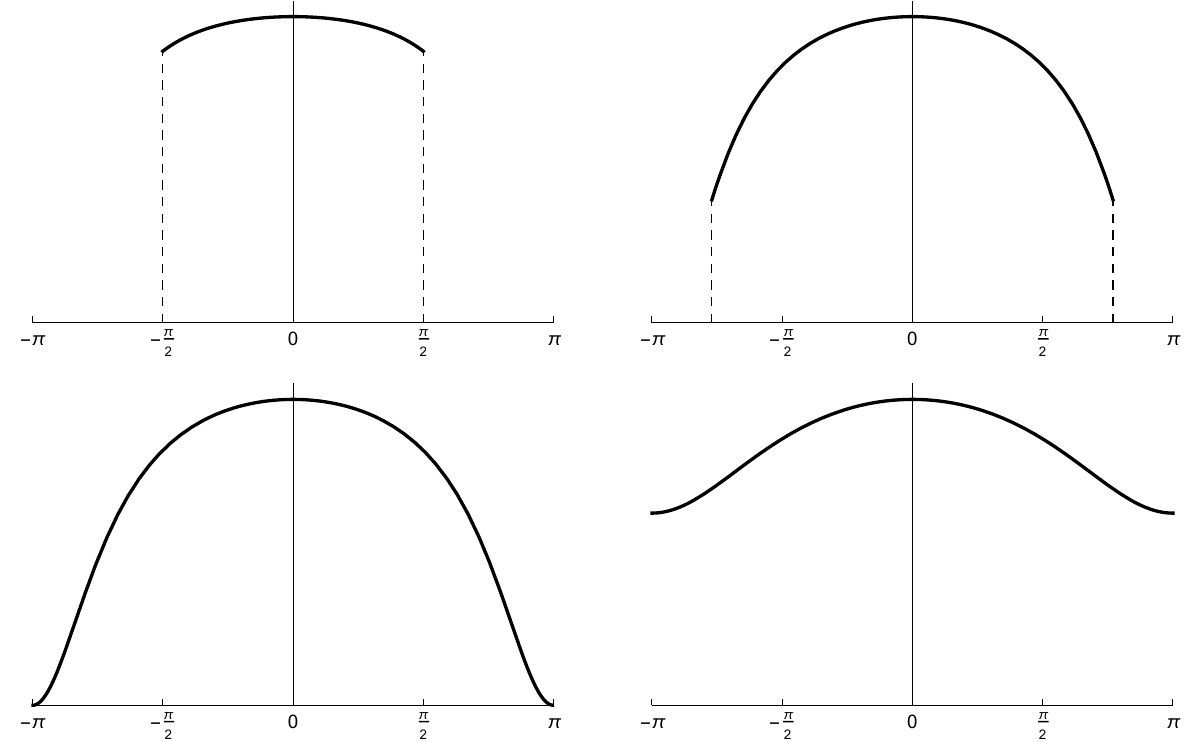}%
\caption{Plots of $w_{t}(\theta)$ for $t=2,$ $3.5,$ $4,$ and $7$}%
\label{vplots.fig}%
\end{figure}

\begin{corollary}
\label{logIndependent.cor}The Brown measure $\mu_{b_{t}}$ of $b_{t}$ has the
following properties.

\begin{enumerate}
\item \label{invariance.point}$\mu_{b_{t}}$ is invariant under the maps
$\lambda\mapsto1/\lambda$ and $\lambda\mapsto\bar{\lambda}.$

\item \label{logInd.point}Let $\Xi_{t}$ denote the image of $\Sigma
_{t}\setminus(-\infty,0)$ under the complex logarithm map, using the standard
branch cut along the negative real axis. We write points $z\in\Xi_{t}$ as
$(\rho,\theta).$ Then for points in $\Xi_{t},$ the pushforward of $\mu_{b_{t}%
}$ by the logarithm map has density $\omega_{t}(\rho,\theta)$ given by%
\[
\omega_{t}(\rho,\theta)=w_{t}(\theta),
\]
independent of $\rho.$
\end{enumerate}
\end{corollary}

\begin{proof}
As we have stated above, the region $\Sigma_{t}$ is invariant under the maps
$\lambda\mapsto1/\lambda$ and $\lambda\mapsto\bar{\lambda}.$ The invariance of
$\mu_{b_{t}}$ under $\lambda\mapsto\bar{\lambda}$ follows from the fact that
$w_{t}$ is even. Then in polar coordinates, the map $\lambda\mapsto1/\lambda$
is $(r,\theta)\mapsto(1/r,-\theta)$. Now, we may compute $\mu_{b_{t}}$ in
polar coordinates as%
\[
d\mu_{b_{t}}=\left(  \frac{1}{r^{2}}w_{t}(\theta)\right)  r~dr~d\theta
=w_{t}(\theta)~\frac{1}{r}dr~d\theta.
\]
If $s=1/r$ and $\phi=-\theta,$ then $r=1/s$ and
\[
\frac{1}{r}dr=s\left(  -\frac{1}{s^{2}}~\right)  ds=-\frac{1}{s}~ds,
\]
so that $d\mu_{b_{t}}=w_{t}(\phi)~\frac{1}{s}ds~d\phi.$ Similarly, if
$\rho=\log r,$ then $d\rho=(1/r)~dr,$ so that $d\mu_{b_{t}}=w_{t}%
(\theta)~d\rho~d\theta.$
\end{proof}

Plots of $w_{t}(\theta)$ are shown in Figure \ref{vplots.fig}. Note that for
$t<4,$ not all angles $\theta$ actually occur in the domain $\Sigma_{t}$.
Thus, for $t<4,$ the function $w_{t}(\theta)$ is only defined for $\theta$ in
a certain interval $(-\theta_{\max}(t),\theta_{\max}(t))$---where, as shown in
Section \ref{regionPropeties.sec}, $\theta_{\max}(t)=\cos^{-1}(1-t/2).$ Plots
of $W_{t}$ for $t=1$ and $t=4$ are then shown in Figures \ref{t1density.fig},
\ref{t4density.fig}, and \ref{t4densitydetail.fig}. Actually, when $t=1,$ the
function $w_{t}$ is almost constant (see Figure \ref{asymptotics1.fig}). Thus,
the variation in $W_{t}$ in Figure \ref{t1density.fig} comes almost entirely
from the variation in the factor of $1/r^{2}$ in (\ref{WtFormula}).

We also observe that by Point \ref{invariance.point} of Corollary
\ref{logIndependent.cor}, half the mass of $\mu_{b_{t}}$ is contained in the
unit disk and half in the complement of the unit disk. Thus, although the
density $W_{t}$ becomes large near the origin in, say, Figures
\ref{t4density.fig} and \ref{t4densitydetail.fig}, it is \textit{not} correct
to say that most of the mass of $\mu_{b_{t}}$ is near the origin.%

\begin{figure}[ptb]%
\centering
\includegraphics[
height=2.5771in,
width=4.5282in
]%
{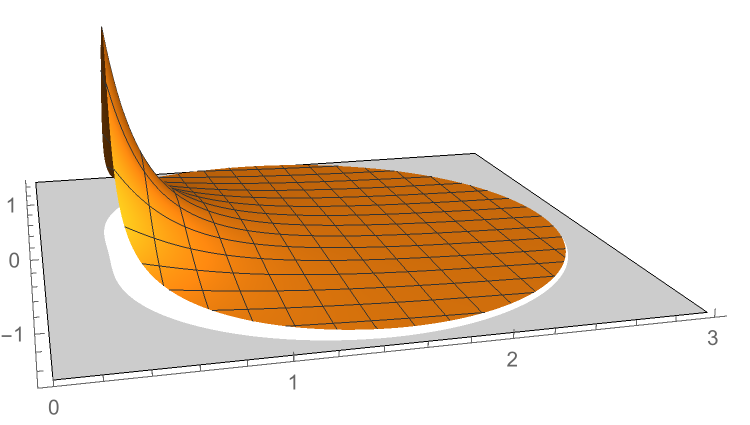}%
\caption{The density $W_{t}$ with $t=1$}%
\label{t1density.fig}%
\end{figure}
%

\begin{figure}[ptb]%
\centering
\includegraphics[
height=2.853in,
width=4.5282in
]%
{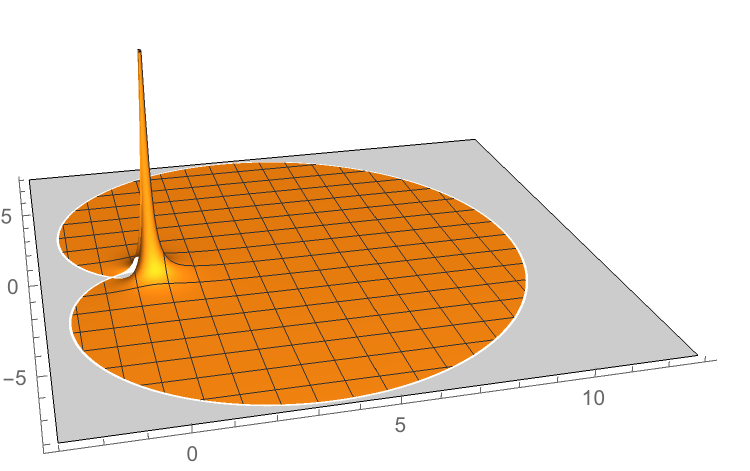}%
\caption{The density $W_{t}$ for $t=4$}%
\label{t4density.fig}%
\end{figure}
%

\begin{figure}[ptb]%
\centering
\includegraphics[
height=2.9776in,
width=4.5282in
]%
{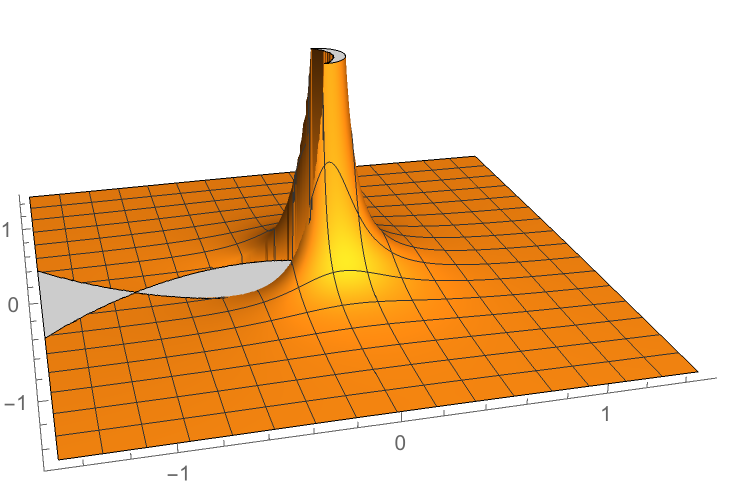}%
\caption{A detail of the density $W_{t}$ for $t=4$}%
\label{t4densitydetail.fig}%
\end{figure}

\subsection{A connection to free unitary Brownian motion}

It follows easily from Theorem \ref{main.thm} that the distribution of the
argument of $\lambda$ with respect to $\mu_{b_{t}}$ has a density given by
\begin{equation}
a_{t}(\theta)=2\log[r_{t}(\theta)]w_{t}(\theta), \label{aTtheta}%
\end{equation}
where, as in Theorem \ref{main.thm}, we take $r_{t}(\theta)$ to be the outer
radius of the domain (with $r_{t}(\theta)>1$). After all, the Brown measure in
the domain is computed in polar coordinates as $(1/r^{2})w_{t}(\theta
)r~dr~d\theta$. Integrating with respect to $r$ from $1/r_{t}(\theta)$ to
$r_{t}(\theta)$ then gives the claimed density for $\theta$.

Recall from Theorem \ref{t.Biane.ut} that the limiting eigenvalue distribution
$\nu_{t}$ for Brownian motion in the unitary group was determined by Biane. We
now claim that the distribution in (\ref{aTtheta}) is related to Biane's
measure $\nu_{t}$ by a natural change of variable. To each angle $\theta$
arising in the region $\Sigma_{t},$ we associate another angle $\phi$ by the
formula%
\begin{equation}
f_{t}(r_{t}(\theta)e^{i\theta})=e^{i\phi}, \label{thetaAndPhi}%
\end{equation}
where $f_{t}$ is as in (\ref{ftDef}). (Recall that, by Definition
\ref{SigmaT.def}, the boundary of $\Sigma_{t}$ maps into the unit circle under
$f_{t}.$) We then have the following remarkable direct connection between the
Brown measure of $b_{t}$ and Biane's measure $\nu_{t}.$

\begin{proposition}
\label{connectToBiane.prop}If $\theta$ is distributed according to the density
in (\ref{aTtheta}) and $\phi$ is defined by (\ref{thetaAndPhi}), then $\phi$
is distributed as Biane's measure $\nu_{t}.$
\end{proposition}

%

\begin{figure}[ptb]%
\centering
\includegraphics[
height=2.1179in,
width=4.0283in
]%
{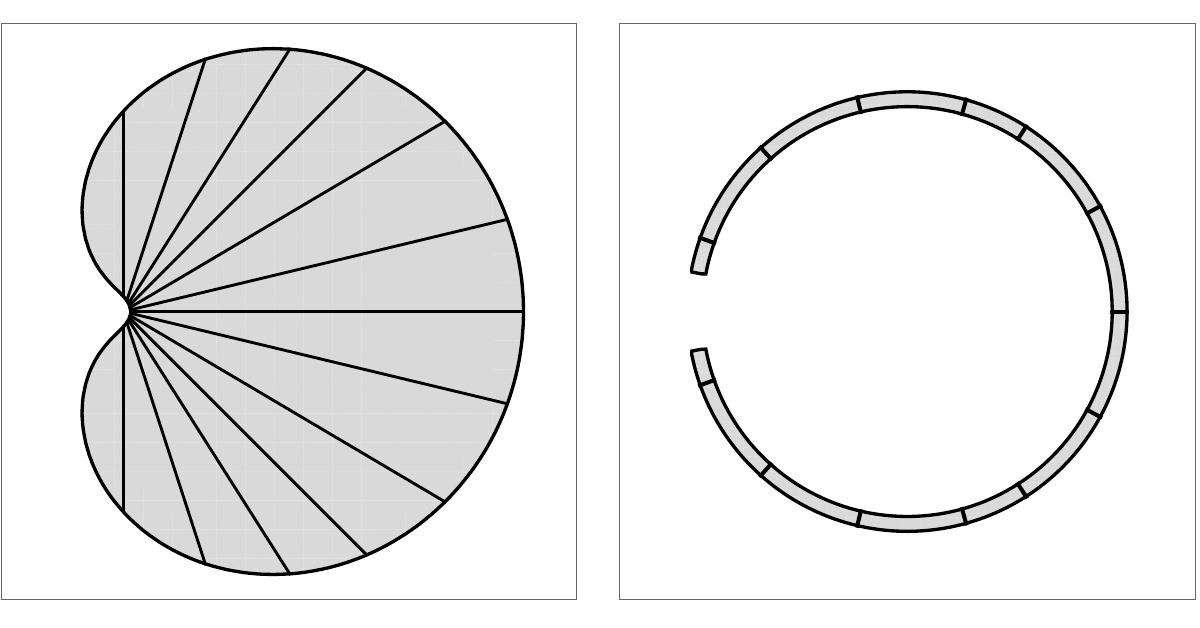}%
\caption{The map $\Phi_{t}:\overline{\Sigma}_{t}\rightarrow S^{1}$ coincides
with $f_{t}$ on $\partial\Sigma_{t}$ and is constant along each radial segment
}%
\label{biane.fig}%
\end{figure}

We may think of this result in a more geometric way, as follows. Define a map
\[
\Phi_{t}:\overline{\Sigma}_{t}\rightarrow S^{1}%
\]
by requiring (a) that $\Phi_{t}$ should agree with $f_{t}$ on the boundary of
$\Sigma_{t},$ and (b) that $\Phi_{t}$ should be constant along each radial
segment inside $\overline{\Sigma}_{t},$ as in Figure \ref{biane.fig}. (This
specification makes sense because $f_{t}$ has the same value at the two
boundary points on each radial segment.) Explicitly, $\Phi_{t}$ may computed
as%
\[
\Phi_{t}(\lambda)=f_{t}(r_{t}(\arg\lambda)e^{i\arg\lambda}).
\]
Then Proposition \ref{connectToBiane.prop} gives the following result, which
may be summarized by saying that \emph{the distribution $\nu_{t}$ of free
unitary Brownian motion is a \textquotedblleft shadow\textquotedblright\ of
the Brown measure of $b_{t}$}.

\begin{proposition}
\label{connectToBiane2.prop}The push-forward of the Brown measure of $b_{t}$
under the map $\Phi_{t}$ is Biane's measure $\nu_{t}$ on $S^{1}.$ Indeed, the
Brown measure of $b_{t}$ is the \emph{unique} measure $\mu$ on $\overline
{\Sigma}_{t}$ with the following two properties: (1) the push-forward of $\mu$
by $\Phi_{t}$ is $\nu_{t}$ and (2) $\mu$ is absolutely continuous with respect
to Lebesgue measure with a density $W$ having the form%
\[
W(r,\theta)=\frac{1}{r^{2}}g(\theta)
\]
in polar coordinates, for some continuous function $g.$
\end{proposition}

Now, the results of \cite{BianeJFA} and \cite{HK} already indicate a
relationship between the free unitary Brownian motion $u_{t}$ (whose spectral
measure is $\nu_{t}$) and the free multiplicative Brownian motion $b_{t}$
(whose Brown measure we are studying in this paper). It is nevertheless
striking to see such a direct relationship between $\mu_{b_{t}}$ and $\nu_{t}%
$. Indeed, Proposition \ref{connectToBiane2.prop} precisely mirrors the
relationship between the semicircle law and the circular law. If $c_{t}$ is a
circular random variable of variance $t$, and $x_{t}$ is semicircular of
variance $t$, then the distribution of $x_{t}$ (the semicircle law on the
interval $[-2\sqrt{t},2\sqrt{t}]$) is the push-forward of the Brown measure of
$c_{t}$ (the uniform probability measure on the disk $\overline{\mathbb{D}%
}(\sqrt{t})$ of radius $\sqrt{t}$) under a similar \textquotedblleft shadow
map\textquotedblright: first project the disk onto its upper boundary circle
via $(x,y)\mapsto(x,\sqrt{t-x^{2}})$, and then use the conformal map $z\mapsto
z+\frac{t}{z}$ from $\mathbb{C}\setminus\overline{\mathbb{D}}(\sqrt{t})$ onto
$\mathbb{C}\setminus\lbrack-2\sqrt{t},2\sqrt{t}]$. (The net result of these
two operations is $(x,y)\mapsto2x.$) Since, as described in the introduction,
$u_{t}$ and $b_{t}$ are the \textquotedblleft Lie group\textquotedblright%
\ versions of the \textquotedblleft Lie algebra\textquotedblright\ operators
$x_{t}$ and $c_{t},$ it is pleasing that this shadow relationship between
their Brown measures persists.

\subsection{The structure of the formula}

We now explain the significance of the two terms on the right-hand side of the
formula (\ref{wtAnswer}) for $w_{t}$. Following the general construction of
Brown measures in (\ref{weakLim}), the density of the Brown measure is
computed as $\frac{1}{4\pi}\Delta s_{t}(\lambda),$ where $\Delta$ is the
Laplacian with respect to $\lambda$ and where
\[
s_{t}(\lambda)=\lim_{\varepsilon\rightarrow0^{+}}\tau\lbrack\log
((b_{t}-\lambda)^{\ast}(b_{t}-\lambda)+\varepsilon)].
\]
It is then convenient to work in polar coordinates $(r,\theta).$ In these
coordinates, we may write $\Delta$ as%
\begin{equation}
\Delta=\frac{1}{r^{2}}\left(  \left(  r\frac{\partial}{\partial r}\right)
^{2}+\frac{\partial^{2}}{\partial\theta^{2}}\right)  . \label{lapPolar}%
\end{equation}

\begin{theorem}
\label{theAnswerExplained.thm}For each $t>0,$ the function $s_{t}(\lambda)$ is
real analytic for $\lambda\in\Sigma_{t}$ and also for $\lambda\in
(\overline{\Sigma}_{t})^{c}.$ At each boundary point, $s_{t}(\lambda)$ and its
first derivatives with respect to $\lambda$ approach the same value from the
inside of $\Sigma_{t}$ as from the outside of $\overline{\Sigma}_{t}.$ For
$\lambda$ inside $\Sigma_{t},$ we have%
\begin{equation}
\left(  r\frac{\partial}{\partial r}\right)  ^{2}s_{t}=\frac{2}{t}.
\label{secondDeriv}%
\end{equation}
For $\lambda$ inside $\Sigma_{t},$ we also have that $\partial s_{t}%
/\partial\theta$ \emph{is independent of} $r$ with $t$ and $\theta$ fixed.
Indeed, $\partial s_{t}/\partial\theta(\lambda)$ is the \emph{unique} function
on $\Sigma_{t}$ that is independent of $r$ and agrees with the angular
derivative of $\log(\left\vert \lambda-1\right\vert ^{2})$ as we approach
$\partial\Sigma_{t}.$
\end{theorem}

The formula (\ref{secondDeriv})---along with (\ref{lapPolar})---accounts for
the first term on the right-hand side of (\ref{wtAnswer}). Then since the
angular derivative of $\log(\left\vert \lambda-1\right\vert ^{2})$ is
computable as%
\[
\frac{\partial}{\partial\theta}\log(\left\vert \lambda-1\right\vert
^{2})=\frac{2\operatorname{Im}\lambda}{\left\vert \lambda-1\right\vert ^{2}%
}=\frac{2r\sin\theta}{r^{2}+1-2\cos\theta}%
\]
as in (\ref{dsdtheta}), we can recognize the second term on the right-hand
side of (\ref{wtAnswer}) as the $\theta$-derivative of $\partial
s_{t}/\partial\theta.$ Thus, Theorem \ref{theAnswerExplained.thm}, together
with the formula (\ref{lapPolar}) for the Laplacian in polar coordinates,
accounts for the formula (\ref{wtAnswer}) for $w_{t}.$

\subsection{Deriving the formula\label{brown.sec}}

We now briefly indicate the method we will use to compute the Brown measure
$\mu_{b_{t}}$. Following the general construction of the Brown measure in
(\ref{weakLim}), we consider the function $S$ defined by%
\begin{equation}
S(t,\lambda,\varepsilon)=\tau\lbrack\log((b_{t}-\lambda)^{\ast}(b_{t}%
-\lambda)+\varepsilon)] \label{Sdefinition}%
\end{equation}
for $\lambda\in\mathbb{C}$ and $\varepsilon>0,$ where $b_{t}$ is the free
multiplicative Brownian motion and $\tau$ is the trace in the von Neumann
algebra in which $b_{t}$ lives. It is easily verified that as $\varepsilon$
decreases with $t$ and $\lambda$ fixed, $S(t,\lambda,\varepsilon)$ also
decreases. Hence, the limit%
\[
s_{t}(\lambda)=\lim_{\varepsilon\rightarrow0^{+}}S(t,\lambda,\varepsilon)
\]
exists, possibly with the value $-\infty.$

The general theory developed by Brown \cite{Br} shows that $s_{t}(\lambda)$ is
a subharmonic function of $\lambda$ for each fixed $t,$ so that the Laplacian
(in the distribution sense) of $s_{t}(\lambda)$ with respect to $\lambda$ is a
positive measure. If this measure happens to be absolutely continuous with
respect to the Lebesgue measure, then the density $W(t,\lambda)$ of the Brown
measure is computed in terms of the value of $s_{t}(\lambda),$ as follows:
\begin{equation}
W(t,\lambda)=\frac{1}{4\pi}\Delta s_{t}(\lambda). \label{WtDef}%
\end{equation}
See also Chapter 11 in \cite{MS} and Section 2.3 in \cite{HK} for general
information on Brown measures.

The first major step toward proving Theorem \ref{main.thm} is the following result.

\begin{theorem}
\label{thePDE.thm}The function $S$ in (\ref{Sdefinition}) satisfies the
following PDE:%
\begin{equation}
\frac{\partial S}{\partial t}=\varepsilon\frac{\partial S}{\partial
\varepsilon}\left(  1+(\left\vert \lambda\right\vert ^{2}-\varepsilon
)\frac{\partial S}{\partial\varepsilon}-a\frac{\partial S}{\partial a}%
-b\frac{\partial S}{\partial b}\right)  ,\quad\lambda=a+ib, \label{thePDE}%
\end{equation}
with the initial condition%
\begin{equation}
S(0,\lambda,\varepsilon)=\log(\left\vert \lambda-1\right\vert ^{2}%
+\varepsilon). \label{SinitialCond}%
\end{equation}

\end{theorem}

We emphasize that $S(t,\lambda,\varepsilon)$ is only defined for
$\varepsilon>0.$ Although, as we will see, $\lim_{\varepsilon\rightarrow
0^{+}}S(t,\lambda,\varepsilon)$ is finite, $\partial S/\partial\varepsilon$
develops singularities in this limit. Thus, it is \textit{not} correct to
formally set $\varepsilon=0$ in (\ref{thePDE}) to obtain $\partial
s_{t}/\partial t=0.$ (Actually, it will turn out that $s_{t}(\lambda)$
\textit{is} independent of $t$ for as long as $\lambda$ remains outside
$\overline{\Sigma}_{t},$ but not after this time; see Section
\ref{outside.sec}.)

After verifying this equation (Section \ref{BrownAndPDE.sec}), we will use the
Hamilton--Jacobi formalism to analyze the solution (Section \ref{HJ.sec}). In
the remaining sections, we will then analyze the limit of the solution as
$\varepsilon$ tends to zero and compute the Laplacian in (\ref{WtDef}). The
expository article \cite{PDEmethods} of the second author provides an
introduction to the methods used in the present paper.

By way of comparison, we mention that a similar PDE was used in Biane's paper
\cite{BianeConvolution}. There he studies the spectral measure $\mu_{t}$ of
$x_{0}+x_{t}$, the free additive Brownian motion with a nonconstant initial
distribution $x_{0}$ freely independent from $x_{t}$. Biane studies the Cauchy
transform $G$ of $\mu_{t}$:
\[
G(t,z)=\int_{\mathbb{R}}\frac{\mu_{t}(dx)}{z-x},\quad\operatorname{Im}(z)>0,
\]
and shows that $G$ satisfies the complex inviscid Burger's equation%
\begin{equation}
\frac{\partial G(t,z)}{\partial t}=-G(t,z)\frac{\partial G}{\partial z}.
\label{Burgers}%
\end{equation}
The measure $\mu_{t}$ may then be recovered, up to a constant, as
$\lim_{\varepsilon\rightarrow0^{+}}\operatorname{Im}G(t,x+i\varepsilon)$.

In our paper, we similarly use a first-order, nonlinear PDE whose solution in
a certain limit gives the desired measure. We note, however, that the PDE
(\ref{Burgers}) is not actually the main source of information about $\mu_{t}$
in \cite{BianeConvolution}. By contrast, our analysis of the Brown measure of
the free multiplicative Brownian motion $b_{t}$ is based entirely on the
PDE\ in Theorem \ref{thePDE.thm}.

Finally, we mention that for the case of the circular Brownian motion $c_{t},$
a PDE similar to the one in Theorem \ref{thePDE} appeared in work of Burda,
Grela, Nowak, Tarnowski, and Warcho\l \ \cite[Equation (9)]{BGNTW}.

\section{Comparison with the eigenvalue distribution of $B_{t}^{N}$}%

\begin{figure}[ptb]%
\centering
\includegraphics[
height=2.0686in,
width=4.7504in
]%
{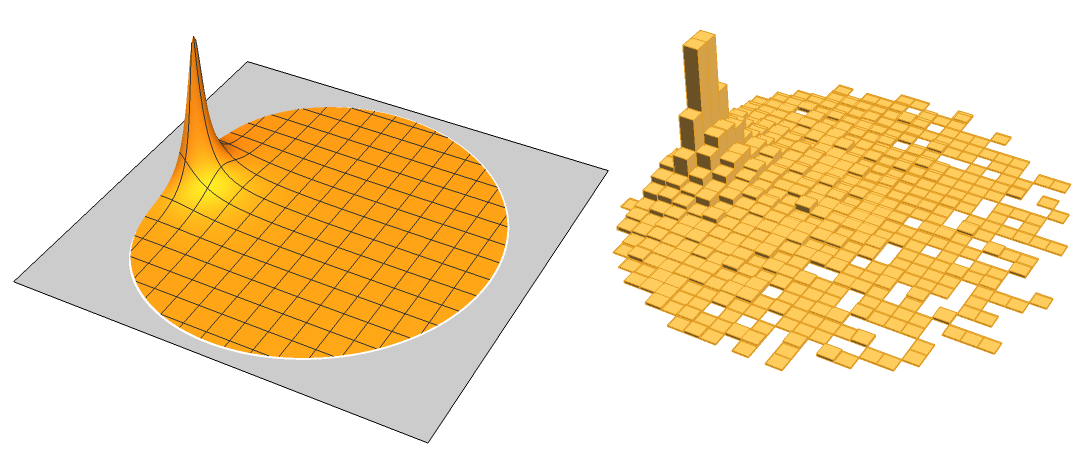}%
\caption{A plot of $W_{t}$ with $t=2$ (left). A histogram of the eigenvalues
of $B_{t}^{N}$ with $N=2,000$ and $t=2$ (right) is shown for comparison}%
\label{3dplotwithhisthor.fig}%
\end{figure}
%

\begin{figure}[ptb]%
\centering
\includegraphics[
height=4.4278in,
width=4.4278in
]%
{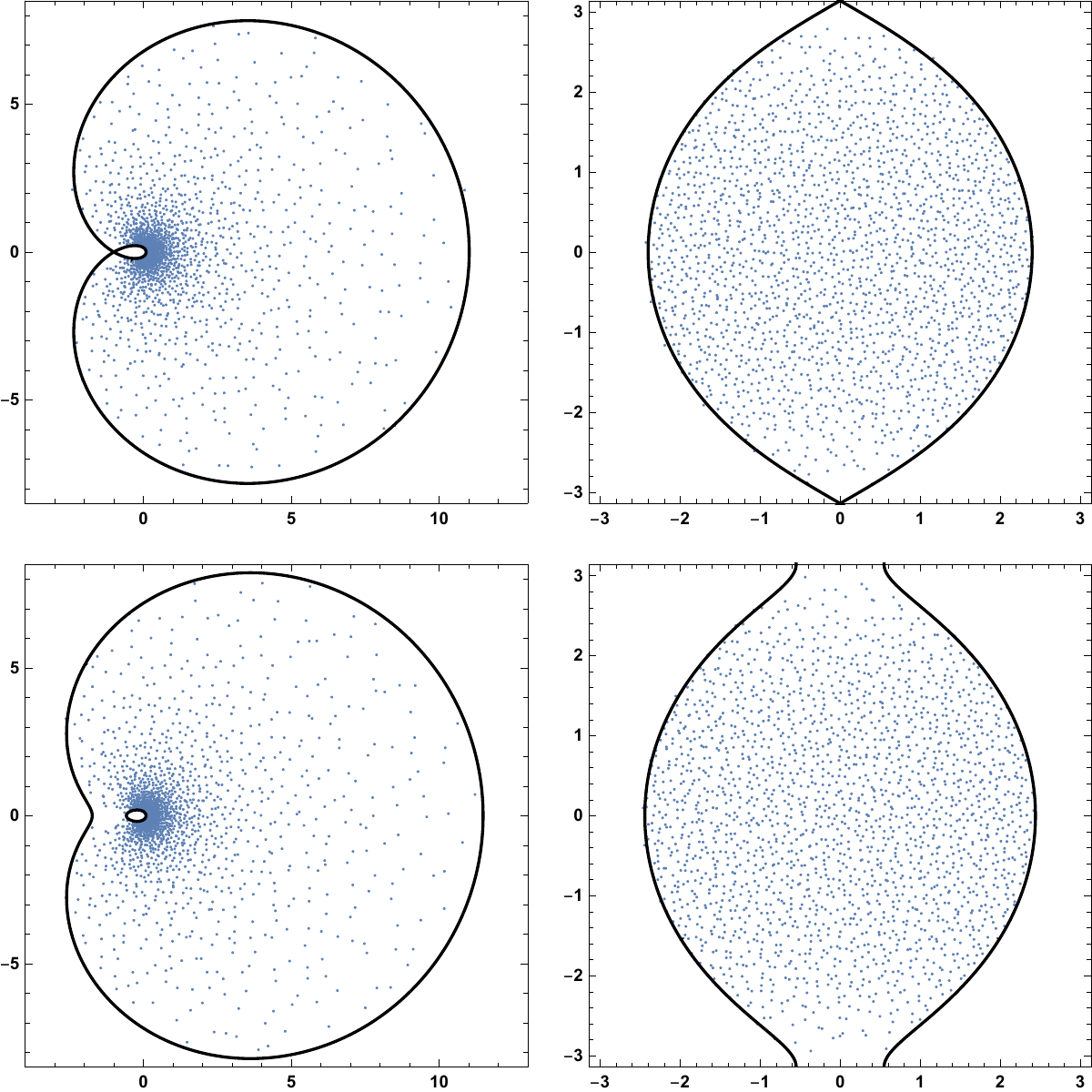}%
\caption{Simulation of the eigenvalues of $B_{t}^{N}$ (left) and their
logarithms (right) with $N=2,000$ for $t=4.0$ and $t=$ $4.1.$ The distribution
of points on the right-hand side of the figure is approximately constant in
the horizontal direction}%
\label{evals41.fig}%
\end{figure}
%

\begin{figure}[ptb]%
\centering
\includegraphics[
height=2.7466in,
width=4.4278in
]%
{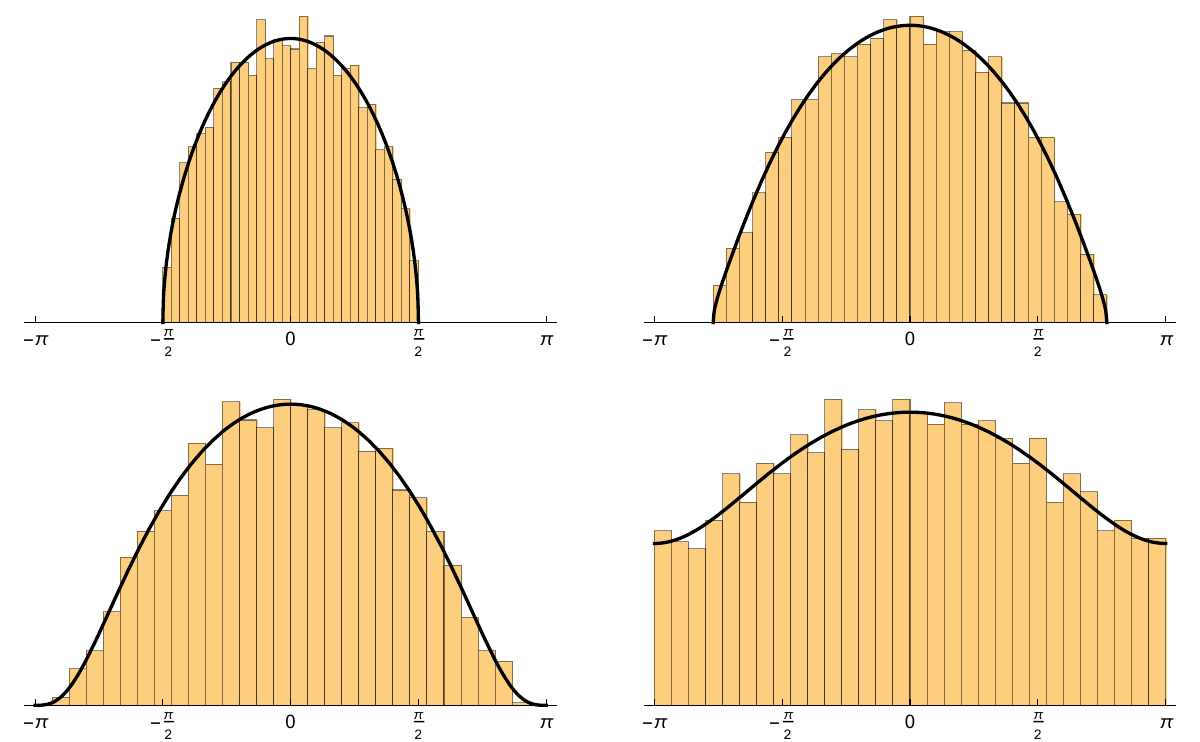}%
\caption{The density $a_{t}(\theta)$ in (\ref{aTtheta}) plotted against a
histogram of the arguments of the eigenvalues of $B_{t}^{N}$, for
\thinspace$N=2,000$ and $t=2,$ $3.5,$ $4,$ and $7$}%
\label{histograms.fig}%
\end{figure}
%

\begin{figure}[ptb]%
\centering
\includegraphics[
height=2.7475in,
width=4.4278in
]%
{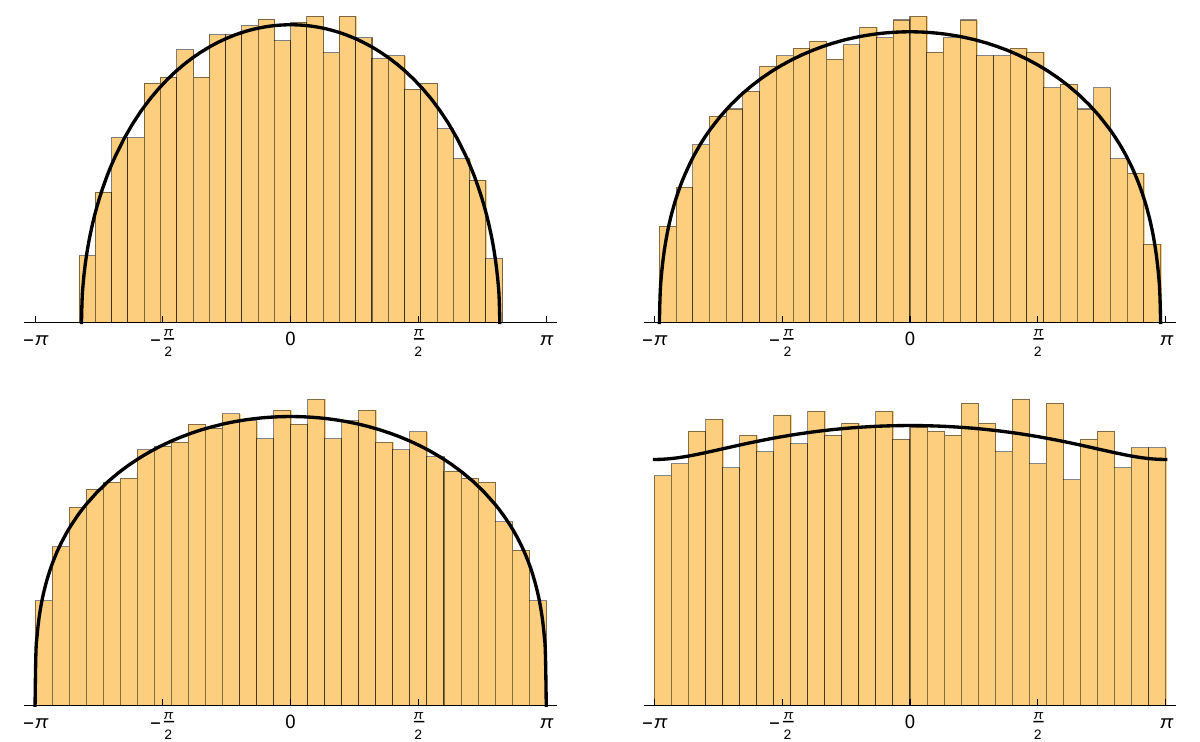}%
\caption{The density of Biane's measure $\nu_{t}(\phi)$ plotted against a
histogram of $\{\Phi_{t}(\lambda_{j})\}_{j=1}^{N}$, for \thinspace$N=2,000$
and $t=2,$ $3.5,$ $4,$ and $7$}%
\label{bianeplots.fig}%
\end{figure}

As mentioned in Section \ref{BrownOfbt}, the Brown measure of the free
multiplicative Brownian motion $b_{t}$ is a natural candidate for the limiting
empirical eigenvalue distribution of the Brownian motion $B_{t}^{N}$ in
$\mathsf{GL}(N;\mathbb{C}).$ We may express this idea formally as a conjecture.

\begin{conjecture}
\label{limitDistribution.conj}For all $t>0,$ the empirical eigenvalue
distribution of $B_{t}^{N}$ converges almost surely as $N\rightarrow\infty$ to
the Brown measure $\mu_{b_{t}}$ of $b_{t}.$
\end{conjecture}

While natural, Conjecture \ref{limitDistribution.conj} is technically
difficult to approach. It is by now well known that the logarithmic
singularity in the definition of the Brown measure can result in failure of
convergence of the empirical distribution of eigenvalues to the Brown measure
of the limit (in $\ast$-distribution) of the random matrix ensembles. Suppose,
for example, $T_{N}$ is an $N\times N$ matrix with all $0$ entries except for
1's just below the diagonal. Then all of $T_{N}$'s eigenvalues are $0$, and
hence the empirical eigenvalue distribution is a point-mass at $0$ for each
$N$. However, in $\ast$-distribution, $T_{N}$ converges to a Haar unitary $u$,
whose Brown measure is the uniform probability measure on the unit circle.
(See, for example, Section 2.6 of \cite{Sniady}.)

In \cite{Sniady}, \'{S}niady proved that convergence to the Brown measure is
\textquotedblleft generic\textquotedblright, in the sense that a small
(vanishing in the limit) independent Gaussian perturbation of the original
ensemble will always yield convergence to the Brown measure. (The required
size of the perturbation was more recently explored in \cite{FPZ}.) A main
step in \'{S}niady's proof is to show (in our language) that the empirical
eigenvalue distribution of a complex Brownian motion $Z_{t}$ on $\mathrm{gl}%
(N;\mathbb{C})$, with any deterministic initial condition, converges to the
appropriate Brown measure. That is to say, there is enough regular noise in
such matrix diffusions to kill any pseudo-spectral discontinuities. It is
natural to expect that the same should hold true for our geometric matrix
diffusion $B_{t}.$

In support of Conjecture \ref{limitDistribution.conj}, we first note that for
small $t,$ the distribution of $B_{t}^{N}$---namely, the time-$t$ heat kernel
measure on $\mathrm{GL}(N;\mathbb{C}),$ based at the identity---is
approximately Gaussian. Thus, $B_{t}^{N}$ is distributed, for small $t,$
similarly to the Ginibre ensemble, shifted by the identity and scaled by
$\sqrt{t}.$ Thus, Conjecture \ref{limitDistribution.conj} leads us to expect
that that $\mu_{b_{t}}$ will be close, for small $t,$ to the uniform
probability measure on a disk of radius $\sqrt{t}.$ This expectation is
confirmed by the asymptotics in Section \ref{OmegaFormula.sec}.

We now offer several numerical tests of the conjecture. First, Figure
\ref{3dplotwithhisthor.fig} directly compares the density $W_{t}$ with $t=2$
to the distribution of eigenvalues of $B_{t}^{N}$ with $t=2$ and $N=2,000.$
Next, in light of Conjecture \ref{limitDistribution.conj} and Point
\ref{logInd.point} of Corollary \ref{logIndependent.cor}, we expect that the
limiting distribution of the \textit{logarithms} of the eigenvalues of
$B_{t}^{N}$ will be constant in the horizontal direction. This expectation is
confirmed by simulations, as in the right-hand side of Figure
\ref{evals41.fig}.

Furthermore, Conjecture \ref{limitDistribution.conj} predicts that the
large-$N$ distribution of the arguments of the eigenvalues will be given by
the density $a_{t}$ in (\ref{aTtheta}). Furthermore, Conjecture
\ref{limitDistribution.conj} and Proposition \ref{connectToBiane2.prop}
predict that, if $\{\lambda_{j}\}_{j=1}^{N}$ are the eigenvalues of $B_{t}%
^{N},$ then for large $N,$ the empirical distribution of $\{\Phi_{t}%
(\lambda_{j})\}_{j=1}^{N}$ will be given by Biane's measure $\nu_{t}.$ Both of
these predictions are confirmed by simulations; see Figures
\ref{histograms.fig} and \ref{bianeplots.fig}.

\section{Properties of $\Sigma_{t}$\label{regionPropeties.sec}}

We now verify some important properties of the regions $\Sigma_{t}$ in
Definition \ref{SigmaT.def}. Define%
\begin{equation}
T(\lambda)=\left\{
\begin{array}
[c]{ll}%
\left\vert \lambda-1\right\vert ^{2}\frac{\log(\left\vert \lambda\right\vert
^{2})}{\left\vert \lambda\right\vert ^{2}-1}, & \left\vert \lambda\right\vert
\neq1\\
\left\vert \lambda-1\right\vert ^{2}, & \left\vert \lambda\right\vert =1
\end{array}
\right.  . \label{Tlambda}%
\end{equation}
Note that the function%
\[
x\mapsto\frac{\log(x)}{x-1}%
\]
has a removable singularity at $x=1,$ with a limiting value of $1$ at $x=1$.
Thus, $T(\lambda)$ is a real analytic function on all of $\mathbb{C}%
\setminus\{0\}.$ Since, also,%
\[
\lim_{x\rightarrow0}\frac{\log(x)}{x-1}=+\infty,
\]
we see that $T(\lambda)\rightarrow+\infty$ as $\lambda\rightarrow0.$ By
checking the three cases $\left\vert \lambda\right\vert >1,$ $\left\vert
\lambda\right\vert =1,$ and $\left\vert \lambda\right\vert <1,$ we may verify
that $T(\lambda)\geq0$ for all $\lambda,$ with equality only if $\lambda=1.$

\begin{theorem}
\label{domainGobbles.thm}For all $t>0,$ the region $\Sigma_{t}$ may be
expressed as%
\[
\Sigma_{t}=\left\{  \left.  \lambda\in\mathbb{C}\right\vert T(\lambda
)<t\right\}
\]
and the boundary of $\Sigma_{t}$ may be expressed as
\[
\partial\Sigma_{t}=\left\{  \left.  \lambda\in\mathbb{C}\right\vert
T(\lambda)=t\right\}  .
\]

\end{theorem}

Thus, each fixed $\lambda\in\mathbb{C}$ will be outside $\overline{\Sigma}%
_{t}$ until $t=T(\lambda)$ and will be inside $\Sigma_{t}$ for all
$t>T(\lambda).$ We may therefore say that $T(\lambda)$ is \textit{the time
that the domain }$\Sigma_{t}$\textit{ gobbles up }$\lambda.$ See Figures
\ref{tplot3d.fig} and \ref{levelsets.fig}.

\begin{theorem}
\label{regionProperties.thm}For each $t>0,$ the region $\Sigma_{t}$ has the
following properties.

\begin{enumerate}
\item \label{thetamax.point}For $t\leq4,$ we have $\left\vert \arg
\lambda\right\vert \leq\cos^{-1}(1-t/2)$ for all $\lambda\in\overline{\Sigma
}_{t},$ with equality precisely for the points on the unit circle with
$\cos\theta=1-t/2.$

\item \label{radialInterval.point}Consider the ray from the origin with angle
$\theta$; if $t\leq4,$ assume $\left\vert \theta\right\vert <\cos^{-1}%
(1-t/2)$. Then this ray intersects $\Sigma_{t}$ precisely in an open interval
of the form $1/r_{t}(\theta)<r<r_{t}(\theta)$ for some $r_{t}(\theta)>1.$

\item \label{domainSmooth.point}The boundary of $\Sigma_{t}$ is smooth for all
$t>0$ with $t\neq4.$ When $t=4,$ the boundary of $\Sigma_{t}$ is smooth except
at $\lambda=-1,$ near which it looks like the transverse intersection of two
smooth curves.

\item The region $\Sigma_{t}$ is invariant under $\lambda\mapsto1/\lambda$ and
under $\lambda\mapsto\bar{\lambda}.$

\item The region $\Sigma_{t}$ coincides with the one defined by Biane in
\cite{BianeJFA}.
\end{enumerate}
\end{theorem}

For Point \ref{thetamax.point}, see Figure \ref{thetamax.fig}.%

\begin{figure}[ptb]%
\centering
\includegraphics[
height=3.5449in,
width=3.0286in
]%
{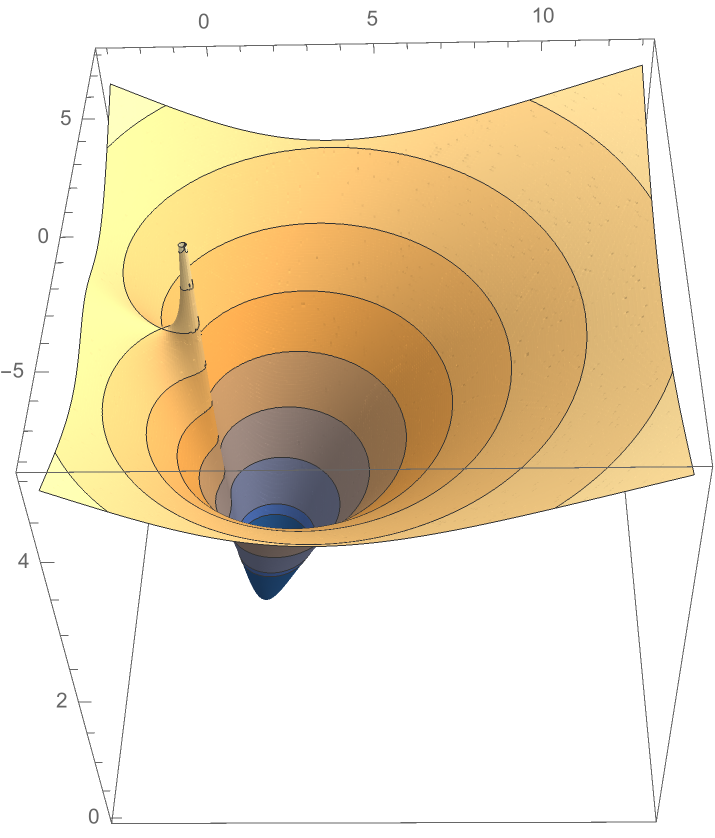}%
\caption{Plot of the function $T(\lambda),$ showing values between $0$ and
$5.$ The function has a global minimum at $\lambda=1,$ a saddle point at
$\lambda=-1,$ and a pole at $\lambda=0$}%
\label{tplot3d.fig}%
\end{figure}
%

\begin{figure}[ptb]%
\centering
\includegraphics[
height=2.5365in,
width=4.8274in
]%
{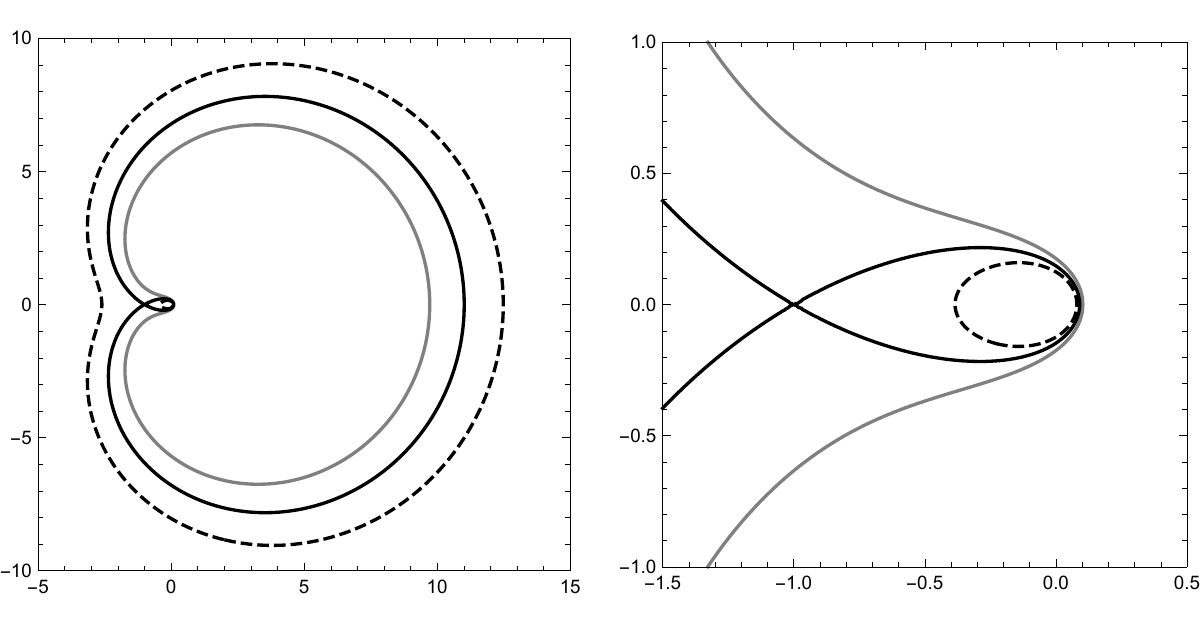}%
\caption{Level sets of the function $T(\lambda)$ form the boundaries of the
regions $\Sigma_{t}.$ Shown for $t=3.7$ (gray), $t=4$ (black), and $t=4.3$
(dashed). The right-hand side of the figure gives a close-up view near
$\lambda=0.$}%
\label{levelsets.fig}%
\end{figure}

We now begin working toward the proofs of Theorems \ref{domainGobbles.thm} and
\ref{regionProperties.thm}.

\begin{lemma}
\label{ftOne.lem}For $\lambda\in\mathbb{C}$ with $\left\vert \lambda
\right\vert \neq1,$ we have $\left\vert f_{t}(\lambda)\right\vert =1$ if and
only if $T(\lambda)=t.$
\end{lemma}

\begin{proof}
Since $f_{t}(0)=0,$ we must have $\lambda\neq0$ if $\left\vert f_{t}%
(\lambda)\right\vert $ is going to equal 1. For nonzero $\lambda,$ we compute
that%
\begin{align}
\log(\left\vert f_{t}(\lambda)\right\vert )  &  =\log\left\vert \lambda
\right\vert +\operatorname{Re}\left(  \frac{t}{2}\frac{1+\lambda}{1-\lambda
}\right) \nonumber\\
&  =\log\left\vert \lambda\right\vert +\frac{t}{2}\frac{1-\left\vert
\lambda\right\vert ^{2}}{\left\vert \lambda-1\right\vert ^{2}}. \label{logFt}%
\end{align}
Thus, for nonzero $\lambda,$ the condition $\left\vert f_{t}(\lambda
)\right\vert =1$ is equivalent to
\[
0=\log\left\vert \lambda\right\vert +\frac{t}{2}\frac{1-\left\vert
\lambda\right\vert ^{2}}{\left\vert \lambda-1\right\vert ^{2}}.
\]
When $\left\vert \lambda\right\vert \neq1,$ this condition simplifies to%
\[
t=\left\vert \lambda-1\right\vert ^{2}\frac{\log(\left\vert \lambda\right\vert
^{2})}{\left\vert \lambda\right\vert ^{2}-1}=T(\lambda),
\]
as claimed.
\end{proof}

We now state some important properties of the function $r_{t}$ occurring in
the statement of Theorem \ref{main.thm}; the proof is given on p.
\pageref{setFtproof}.

\begin{proposition}
\label{setFt.prop}Consider a real number $t>0$ and an angle $\theta\in
(-\pi,\pi],$ where if $t\leq4,$ we require $\left\vert \theta\right\vert
<\cos^{-1}(1-t/2).$ Then there exist exactly two radii $r\neq1$ for which
$\left\vert f_{t}(re^{i\theta})\right\vert =1,$ and these radii have the form
$r=r_{t}(\theta)$ and $r=1/r_{t}(\theta)$ with $r_{t}(\theta)>1.$ Furthermore,
$r_{t}(\theta)$ depends analytically on $\theta$ and if $t\leq4,$ then
$r_{t}(\theta)\rightarrow1$ as $\theta\rightarrow\pm\cos^{-1}(1-t/2).$

If $t\leq4$ and $\theta\in(-\pi,\pi]$ satisfies $\left\vert \theta\right\vert
\geq\cos^{-1}(1-t/2),$ then there are no radii $r\neq1$ with $\left\vert
f_{t}(r)\right\vert =1.$
\end{proposition}

Using the proposition, we can now compute the sets $F_{t}$ and $E_{t}%
=\overline{F}_{t}$ that enter into the definition of $\Sigma_{t}$. (Recall
(\ref{FtDef}) and (\ref{EtDef}).)%

\begin{figure}[ptb]%
\centering
\includegraphics[
height=2.5278in,
width=2.5278in
]%
{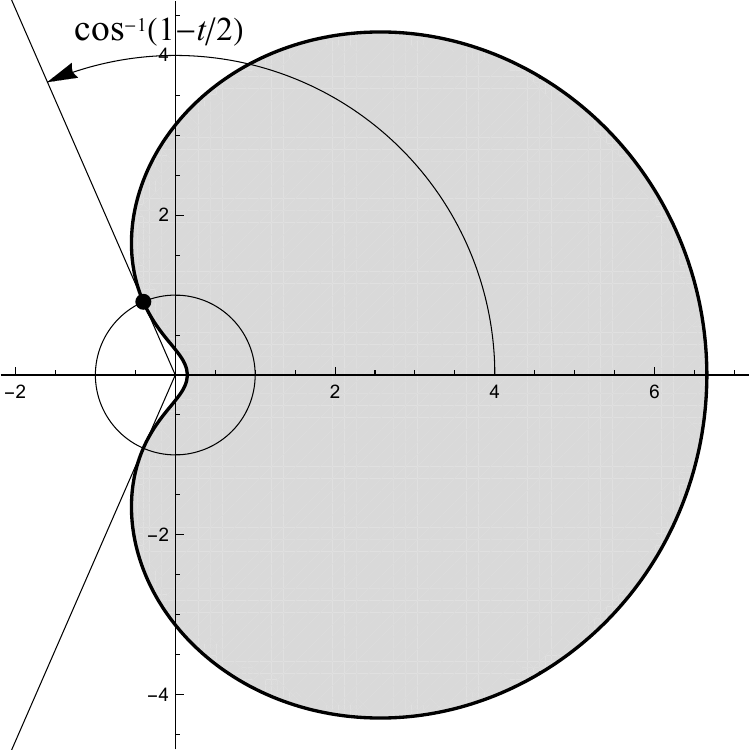}%
\caption{Illustration of Point \ref{thetamax.point} of Proposition
\ref{regionProperties.thm}, for $t=2.8$}%
\label{thetamax.fig}%
\end{figure}

\begin{corollary}
\label{setFt.cor}For $t\leq4,$ the set $F_{t}$ consists of points of the form
$r_{t}(\theta)e^{i\theta}$ and $(1/r_{t}(\theta))e^{i\theta}$ for $-\cos
^{-1}(1-t/2)<\theta<\cos^{-1}(1-t/2).$ In this case, the closure of $F_{t}$
consists of $F_{t}$ together with the points $e^{i\theta}$ on the unit circle
with $\cos\theta=1-t/2.$ There are two such points when $t<4$ and one such
point when $t=4,$ namely $-1.$

For $t>4,$ the set $F_{t}$ consists of points of the form $r_{t}%
(\theta)e^{i\theta}$ and $(1/r_{t}(\theta))e^{i\theta},$ where $\theta$ ranges
over all possible angles, and this set is closed.
\end{corollary}

See Figure \ref{setft.fig}.%

\begin{figure}[ptb]%
\centering
\includegraphics[
height=2.4163in,
width=2.5278in
]%
{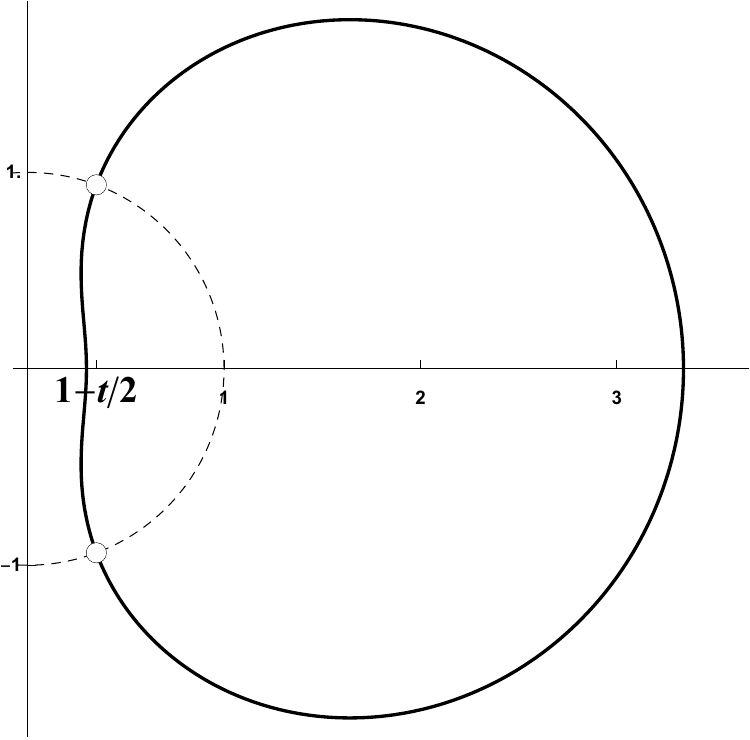}%
\caption{The set $F_{t}$ with $t=1.3,$ with the unit circle (dashed) shown for
comparison.}%
\label{setft.fig}%
\end{figure}

We now set out to prove Proposition \ref{setFt.prop}. In the proof, we will
always rewrite the equation $\left\vert f_{t}(\lambda)\right\vert =1$, for
$\left\vert \lambda\right\vert \neq1$, as $T(\lambda)=t$ (Lemma
\ref{ftOne.lem}).

\begin{lemma}
\label{Tmin.lem}Let us write the function $T$ in (\ref{Tlambda}) in polar
coordinates. Then for each $\theta,$ the function $r\mapsto T(r,\theta)$ is
strictly decreasing for $0<r<1$ and strictly increasing for $r>1.$ For each
$\theta,$ the minimum value of $T(r,\theta)$, achieved at $r=1$, is
$2(1-\cos\theta),$ and we have%
\[
\lim_{r\rightarrow0}T(r,\theta)=\lim_{r\rightarrow+\infty}T(r,\theta
)=+\infty.
\]

\end{lemma}

\begin{proof}
We will show in Proposition \ref{smallx0.prop} that the function $T(\lambda)$
is the limit of another function $t_{\ast}(\lambda_{0},\varepsilon_{0})$ as
$\varepsilon_{0}$ goes to zero. Explicitly, this amounts to saying that
$T(r,\theta)=g_{\theta}(\delta),$ where $g$ is defined in (\ref{gDelta}) and
$\delta=r+1/r.$ Now, $\delta$ is decreasing for $0<r<1$ and increasing for
$r>1.$ Thus, the claimed monotonicity of $T$ follows if $g_{\,\theta}(\delta)$
an increasing function $\delta$ for each $\theta,$ which we will show in the
proof of Proposition \ref{monotoneTstar.prop}.

For the convenience of the reader, we briefly outline how the argument goes in
the context of the function $T(r,\theta).$ We note that%
\[
T(r,\theta)=(r^{2}+1-2r\cos\theta)\frac{\log(r^{2})}{r^{2}-1},
\]
where if we assign $\log(r^{2})/(r^{2}-1)$ the value $1$ at $r=1,$ then $T$ is
analytic except at $r=0.$ We then compute that, after simplification,%
\begin{equation}
\frac{\partial T}{\partial r}=2\left[  \frac{-2r+(1+r^{2})\cos\theta}%
{(r^{2}-1)^{2}}\right]  \log(r^{2})+\frac{r^{2}+1-2r\cos\theta}{r^{2}-1}%
\frac{2}{r}. \label{dtdr}%
\end{equation}
We then claim that for all $\theta,$ we have $\partial T/\partial r>0$ for
$r>1$ and $\partial T/\partial r<0$ for $r<1.$ Note that for each fixed $r,$
the right-hand side of (\ref{dtdr}) depends linearly on $\cos\theta.$ Thus,
if, for a fixed $r,$ if $\partial T/\partial r$ is positive when $\cos
\theta=1$ and when $\cos\theta=-1,$ it will be positive for all $\theta$.
Specifically, we may say that%
\begin{equation}
\frac{\partial T}{\partial r}(r,\theta)\geq\min\left(  \frac{\partial
T}{\partial r}(r,0),\frac{\partial T}{\partial r}(r,\pi)\right)  .
\label{dtdrMin}%
\end{equation}
It is now an elementary (if slightly messy) computation to check that the
right-hand side of (\ref{dtdrMin}) is strictly positive for all $r>1.$ A
similar argument then shows that $\partial T/\partial r$ is negative for all
$\theta$ and all $0<r<1.$

We conclude that for each $\theta,$ the function $r\mapsto T(r,\theta)$ is
decreasing for $0<r<1$ and increasing for $r>1.$ The minimum value therefore
occurs at $r=1,$ and this value is the value of $r^{2}+1-2\cos\theta$ at
$r=1,$ namely $2(1-\cos\theta).$ Finally, we can easily see that for $r$
approaching zero, we have%
\[
T(r,\theta)\sim-\log(r^{2})\rightarrow+\infty
\]
and for $r$ approaching infinity, we have%
\[
T(r,\theta)\sim\log(r^{2})\rightarrow+\infty,
\]
as claimed.
\end{proof}

\begin{proof}
[Proof of Proposition \ref{setFt.prop}]\label{setFtproof}The minimum value of
$T(r,\theta),$ achieved at $r=1,$ is $2-2\cos\theta.$ This value is always
less than $t,$ as can be verified separately in the cases $t>4$ (all $\theta$)
and $t\leq4$ ($\left\vert \theta\right\vert <\cos^{-1}(1-t/2)$). Thus, Lemma
\ref{Tmin.lem} tells us that the equation $T(r,\theta)=t$ has exactly one
solution for $r$ with $0<r<1$ and exactly one solution for $r>1.$ Since, as is
easily verified, $T(1/r,\theta)=T(r,\theta),$ the two solutions are
reciprocals of each other, and we let $r_{t}(\theta)$ denote the solution with
$r>1.$ Since $\partial T/\partial r$ is nonzero for all $r\neq1,$ the implicit
function theorem tells us that $r_{t}(\theta)$ depend analytically on
$\theta.$

Now, if $t\leq4$ and $\theta$ approaches $\pm\cos^{-1}(1-t/2),$ the minimum
value of $2-2\cos\theta$---achieved at $r=1$---approaches $2-2(1-t/2)=t.$ It
should then be plausible that $r_{t}(\theta)$ will approach $r=1.$ To make
this claim rigorous, we need to show that $T(r,\theta)$ increases rapidly
enough as $r$ increases from $1$ that the $T(r,\theta)=t$ is achieved close to
$r=1.$ To that end, let $g(r)$ denote the function on the right-hand side of
(\ref{dtdrMin}), which is continuous everywhere and strictly positive for
$r>1.$ Then for $r>1,$ we have%
\begin{equation}
T(r,\theta)-(2-2\cos\theta)\geq G(r):=\int_{1}^{r}g(s)~ds. \label{TminusT}%
\end{equation}
Now, $G(r)$ is continuous and strictly increasing for $r\geq0,$ with $G(1)=0.$
Thus, $G$ it has a continuous inverse function satisfying $G^{-1}(0)=1.$

For $\varepsilon>0,$ choose $\delta>0$ so that $G^{-1}(R)<1+\varepsilon$ when
$0<R<\delta.$ Then take $\theta$ sufficiently close to $\pm\cos^{-1}(1-t/2)$
that $2-2\cos\theta$ is within $\delta$ of $t.$ Then%
\[
G^{-1}(t-(2-2\cos\theta))<1+\varepsilon,
\]
which is to say that there is an $R$ with $1<R<1+\varepsilon$ such that%
\[
\int_{1}^{R}g(s)~ds=t-(2-2\cos\theta).
\]
From (\ref{TminusT}) we can then see that $T(R,\theta)>t.$ Thus, $r_{t}%
(\theta)$ will satisfy%
\[
1<r_{t}(\theta)<R<1+\varepsilon.
\]
We have therefore shown that $r_{t}(\theta)\rightarrow1$ as $\theta
\rightarrow\pm\cos^{-1}(1-t/2).$

Finally, if $t\leq4$ and $\theta\in(-\pi,\pi]$ satisfies $\left\vert
\theta\right\vert \geq\cos^{-1}(1-t/2),$ the minimum value of $T(r,\theta),$
achieved at $r=1,$ is $2-2\cos\theta\geq t.$ Thus, there are no values of
$r\neq1$ where $T(r,\theta)=t.$
\end{proof}

We are now ready for the proofs of our main results about $\Sigma_{t}.$

\begin{proof}
[Proof of Theorem \ref{domainGobbles.thm}]We first claim that the set
$E_{t}=\overline{F_{t}}$ is precisely the set where $T(\lambda)=t.$ To see
this, first note that $F_{t}$ is, by Lemma \ref{ftOne.lem}, the set of
$\lambda$ with $\left\vert \lambda\right\vert \neq1$ where $T(\lambda)=t.$
Then by Corollary \ref{setFt.cor}, the closure of $F_{t}$ is obtained by
adding in the points on the unit circle (zero, one, or two such points,
depending on $t$) where $\cos\theta=1-t/2.$ But these points are easily seen
to be the points on the unit circle where $T(\lambda)=t.$

Using Corollary \ref{setFt.cor}, we see that the complement of the set
$E_{t}=\left\{  \lambda|T(\lambda)=t\right\}  $ has two connected components
when $t<4$ and three connected components when $t\geq4.$ Since $T(1)=0<t,$ we
have $T(\lambda)<t$ on the entire connected component of $E_{t}^{c}$
containing 1, which is, by definition, the region $\Sigma_{t}.$ The remaining
components of $E_{t}^{c}$ are the unbounded component and (for $t\geq4$) the
component containing 0. Since $T(\lambda)$ tends to $+\infty$ at zero and at
infinity, we see that $T(\lambda)>t$ on these regions, so that $T(\lambda)<t$
precisely on $\Sigma_{t}.$

It is also clear from Corollary \ref{setFt.cor} that the boundary of the
region $\Sigma_{t}$ (i.e., the connected component of $E_{t}^{c}$ containing
1) contains the entire set $E_{t}=$ $\left\{  \lambda|T(\lambda)=t\right\}  .$
\end{proof}

\begin{proof}
[Proof of Theorem \ref{regionProperties.thm}]Point \ref{thetamax.point}
follows easily from Corollary \ref{setFt.cor}. For Point
\ref{radialInterval.point}, we note that by Proposition \ref{setFt.prop}, we
have $T(r,\theta)<t$ for $1/r_{t}(\theta)<r<r_{t}(\theta),$ and $T(r,\theta
)\geq t$ for $0<r\leq1/r_{t}(\theta)$ and for $r\geq r_{t}(\theta).$ Thus, by
Theorem \ref{domainGobbles.thm}, the ray with angle $\theta$ intersects
$\Sigma_{t}$ precisely in the claimed interval.

For Point \ref{domainSmooth.point}, we have already shown that $\partial
T/\partial r$ is nonzero except when $r=1.$ When $r=1,$ we know from
(\ref{Tlambda}) that
\[
T(r,\theta)=\left\vert \lambda-1\right\vert ^{2}=2-2\cos\theta.
\]
Thus, when $r=1,$ we have $\partial T/\partial\theta=2\sin\theta,$ which is
nonzero except when $\theta=0$ or $\theta=\pi.$ Thus, the gradient of
$T(\lambda)$ is nonzero except when $\lambda=0$ (where $T(\lambda)$ is
undefined), when $\lambda=1$, and when $\lambda=-1.$ Since 0 is never in
$\Sigma_{t}$ and 1 is always in $\Sigma_{t},$ the only possible singular point
in the boundary of $\Sigma_{t}$ is at $\lambda=-1.$ Since $T(r,\theta
)=2-2\cos\pi=4$ when $r=1$ and $\theta=\pi,$ the point $\lambda=-1$ belongs to
the boundary of $\Sigma_{4}.$

Meanwhile, the Taylor expansion of $T$ to second order at $\lambda=-1$ is
easily found to be $T(\lambda)\approx4+(\operatorname{Re}\lambda
+1)^{2}/3-(\operatorname{Im}\lambda)^{2}.$ By the Morse lemma, we can then
make a smooth change of variables so that in the new coordinate system,
\[
T(u,v)=4+u^{2}-v^{2}=4+(u+v)(u-v).
\]
Thus, near $\lambda=-1,$ the set $T(\lambda)=4$ is the union of the curves
$u+v=0$ and $u-v=0.$

The invariance of $\Sigma_{t}$ under $\lambda\mapsto1/\lambda$ and under
$\lambda\mapsto\bar{\lambda}$ follows from the easily verified invariance of
$T(\lambda)$ under these transformations. Finally, we verify that the domain
$\Sigma_{t},$ as we have defined it, coincides with the one originally
introduced by Biane in \cite{BianeJFA}. Let us start with the case $t<4.$
According to the discussion at the bottom of p. 273 in \cite{BianeJFA}, the
boundary of Biane's domain $\Sigma_{t}$ consists in this case of two analytic
arcs. The interior of one arc lies in the open unit disk and the interior of
the other arc lies in the complement of the closed unit disk, while the
endpoints of both arcs lie on the unit circle. The first arc is then computed
by applying a certain holomorphic function $\chi(t,\cdot)$ to the support of
Biane's measure $\nu_{t}$ in the unit circle. Now, $\chi(t,\cdot)$ satisfies
$f_{t}(\chi(t,z))=z$ on the closed unit disk. (Combine the identity involving
$\kappa$ on p. 266 of \cite{BianeJFA} with the definition of $\chi$ on p.
273.) We see that the interior of the first arc consists of points with
$\left\vert \lambda\right\vert \neq1$ but $\left\vert f_{t}(\lambda
)\right\vert =1.$ This arc must, therefore, coincide with the arc of points
with radius $1/r_{t}(\theta).$ The second arc is obtained from the first by
the map $\lambda\mapsto1/\lambda$ and therefore coincides with the points of
radius $r_{t}(\theta).$ We can now see that the boundary of Biane's domain
coincides with the boundary of the domain we have defined. A similar analysis
applies to the cases $t>4$ and $t=4,$ using the description of the boundary of
$\Sigma_{t}$ in those cases at the top of p. 274 in \cite{BianeJFA}.
\end{proof}

\section{The PDE for $S$\label{BrownAndPDE.sec}}

In this section, we will verify the PDE for $S$ in Theorem \ref{thePDE.thm}.
The claimed initial condition (\ref{SinitialCond}) holds because $b_{0}=1.$ We
now proceed to verify the equation (\ref{thePDE}) itself.

Let $(c_{t})_{t\geq0}$ denote a free circular Brownian motion. The rules of
free stochastic calculus, in \textquotedblleft stochastic
differential\textquotedblright\ form, are as follows; see \cite[Lemma 2.5,
Lemma 4.3]{KempLargeN}. If $g_{t}$ and $h_{t}$ are processes adapted to
$c_{t}$, then
\begin{align}
dc_{t}\,g_{t}\,dc_{t}^{\ast}  &  =dc_{t}^{\ast}\,g_{t}\,dc_{t}=\tau
(g_{t})\,dt\label{Ito1}\\
dc_{t}\,g_{t}\,dc_{t}  &  =dc_{t}^{\ast}\,g_{t}\,dc_{t}^{\ast}=0\label{Ito2}\\
dc_{t}\,dt  &  =dc_{t}^{\ast}\,dt=0\label{Ito3}\\
\tau(g_{t}\,dc_{t}\,h_{t})  &  =\tau(g_{t}\,dc_{t}^{\ast}\,h_{t})=0.
\label{Ito4}%
\end{align}
In addition, we have the following It\^{o} product rule: if $a_{t}^{1}%
,\ldots,a_{t}^{n}$ are processes adapted to $c_{t}$, then
\begin{align}
d(a_{t}^{1}\cdots a_{t}^{n})  &  =\sum_{j=1}^{n}(a_{t}^{1}\cdots a_{t}%
^{j-1})\,da_{t}^{j}\,(a_{t}^{j+1}\cdots a_{t}^{n})\label{productRule1}\\
&  +\sum_{1\leq j<k\leq n}(a_{t}^{1}\cdots a_{t}^{j-1})\,da_{t}^{j}%
\,(a_{t}^{j+1}\cdots a_{t}^{k-1})\,da_{t}^{k}\,(a_{t}^{k+1}\cdots a_{t}^{n}).
\label{productRule2}%
\end{align}

We let $b_{t}$ be the free multiplicative Brownian motion, which satisfies the
free stochastic differential equation%
\[
db_{t}=b_{t}~dc_{t},\quad b_{0}=1.
\]
Throughout the rest of this section, we will use the notation%
\[
b_{t,\lambda}:=b_{t}-\lambda.
\]

\begin{lemma}
\label{nthPower.lem}We have%
\begin{equation}
\frac{\partial}{\partial t}\tau\lbrack(b_{t,\lambda}^{\ast}b_{t,\lambda}%
)^{n}]=n\sum_{j=0}^{n-1}\tau\lbrack(b_{t,\lambda}^{\ast}b_{t,\lambda}%
)^{j}]\tau\lbrack b_{t}b_{t}^{\ast}(b_{t,\lambda}b_{t,\lambda}^{\ast}%
)^{n-j-1}]. \label{ddtNthPower}%
\end{equation}

\end{lemma}

Note that second factor on the right-hand side of (\ref{ddtNthPower}) has
$(b_{t,\lambda}b_{t,\lambda}^{\ast})^{n-j},$ with the adjoint on the
\textit{second} factor.

\begin{proof}
Note that $db_{t,\lambda}=db_{t}=b_{t}~dc_{t}$ and $db_{t,\lambda}^{\ast
}=db_{t}^{\ast}=dc_{t}^{\ast}~b_{t}^{\ast}.$ We compute $d\tau\lbrack
(b_{t,\lambda}^{\ast}b_{t,\lambda})^{n}]$ by moving the $d$ inside the trace
and then applying the product rule in (\ref{productRule1}) and
(\ref{productRule2}). By (\ref{Ito4}), the terms arising from
(\ref{productRule1}) will not contribute. Furthermore, by (\ref{Ito2}), the
only terms from (\ref{productRule2}) that contribute are those where one $d$
goes on a factor of $b_{t,\lambda}$ and one goes on a factor of $b_{t,\lambda
}^{\ast}.$

By choosing all possible factors of $b_{t,\lambda}$ and all possible factors
of $b_{t,\lambda}^{\ast},$ we get $n^{2}$ terms. In each term, after putting
the $d$ inside the trace, we can cyclically permute the factors until, say,
the $db_{t,\lambda}$ factor is at the end. There are then only $n$
\textit{distinct} terms that occur, each of which occurs $n$ times. By
(\ref{Ito1}), each distinct term is computed as%
\begin{align*}
&  \tau\lbrack(b_{t,\lambda}^{\ast}b_{t,\lambda})^{j}dc_{t}^{\ast}~b_{t}%
^{\ast}b_{t,\lambda}(b_{t,\lambda}^{\ast}b_{t,\lambda})^{n-j-2}b_{t,\lambda
}^{\ast}b_{t}~dc_{t}]\\
&  =\tau\lbrack b_{t}^{\ast}b_{t,\lambda}(b_{t,\lambda}^{\ast}b_{t,\lambda
})^{n-j-2}b_{t,\lambda}^{\ast}b_{t}]\tau\lbrack(b_{t,\lambda}^{\ast
}b_{t,\lambda})^{j}]~dt\\
&  =\tau\lbrack(b_{t,\lambda}^{\ast}b_{t,\lambda})^{j}]\tau\lbrack b_{t}%
b_{t}^{\ast}(b_{t,\lambda}b_{t,\lambda}^{\ast})^{n-j-1}]~dt.
\end{align*}
(The reader who doubts the validity of using the cyclic invariance of the
trace when some factors are differentials may compute each term by
\textit{first} using (\ref{Ito1}) and \textit{then} using the cyclic
invariance of the trace, with the same result.) Since each distinct term
occurs $n$ times, we obtain%
\[
d\tau\lbrack(b_{t,\lambda}^{\ast}b_{t,\lambda})^{n}]=n\sum_{j=0}^{n-1}%
\tau\lbrack(b_{t,\lambda}^{\ast}b_{t,\lambda})^{j}]\tau\lbrack b_{t}%
b_{t}^{\ast}(b_{t,\lambda}b_{t,\lambda}^{\ast})^{n-j-1}]~dt
\]
as claimed.
\end{proof}

\begin{lemma}
\label{tDerivative.lem}The function $S$ in (\ref{Sdefinition}) satisfies%
\begin{equation}
\frac{\partial S}{\partial t}=\varepsilon\tau\lbrack(b_{t,\lambda}^{\ast
}b_{t,\lambda}+\varepsilon)^{-1}]\tau\lbrack b_{t}b_{t}^{\ast}(b_{t,\lambda
}b_{t,\lambda}^{\ast}+\varepsilon)^{-1}]. \label{dsdt0}%
\end{equation}

\end{lemma}

Of course, since $b_{t}=b_{t,\lambda}+\lambda,$ we can rewrite (\ref{dsdt0})
in a way that involves only $b_{t,\lambda}$ and not $b_{t}.$

\begin{proof}
We note that the definition of $S$ actually makes sense for all $\varepsilon
\in\mathbb{C}$ with $\operatorname{Re}(\varepsilon)>0,$ using the standard
branch of the logarithm function. We note that for $\left\vert \varepsilon
\right\vert >\left\vert z\right\vert ,$ we have%
\begin{align}
\frac{1}{z+\varepsilon}  &  =\frac{1}{\varepsilon\left(  1-\left(  -\frac
{z}{\varepsilon}\right)  \right)  }\nonumber\\
&  =\frac{1}{\varepsilon}\left[  1-\frac{z}{\varepsilon}+\frac{z^{2}%
}{\varepsilon^{2}}-\frac{z^{3}}{\varepsilon^{3}}+\cdots\right]  . \label{1zx}%
\end{align}
Integrating with respect to $z$ gives%
\[
\log(z+\varepsilon)=\log\varepsilon+\sum_{n=1}^{\infty}\frac{(-1)^{n-1}}%
{n}\left(  \frac{z}{\varepsilon}\right)  ^{n}.
\]
Thus, for $\left\vert \varepsilon\right\vert >\left\Vert b_{t}^{\ast}%
b_{t}\right\Vert ,$ we have%
\begin{equation}
\tau\lbrack\log(b_{t,\lambda}^{\ast}b_{t,\lambda}+\varepsilon)]=\log
\varepsilon+\,\sum_{n=1}^{\infty}\frac{(-1)^{n-1}}{n\varepsilon^{n}}%
\tau\lbrack(b_{t,\lambda}^{\ast}b_{t,\lambda})^{n}]. \label{tauLog}%
\end{equation}

Assume for the moment that it is permissible to differentiate (\ref{tauLog})
term by term with respect to $t.$ Then by Lemma \ref{nthPower.lem}, we have%
\begin{equation}
\frac{\partial S}{\partial t}=\sum_{n=1}^{\infty}\frac{(-1)^{n-1}}%
{\varepsilon^{n}}\sum_{j=0}^{n-1}\tau\lbrack(b_{t,\lambda}^{\ast}b_{t,\lambda
})^{j}]\tau\lbrack b_{t}b_{t}^{\ast}(b_{t,\lambda}b_{t,\lambda}^{\ast
})^{n-j-1}]. \label{dsdt1}%
\end{equation}
Now, by \cite[Proposition 3.2.3]{BS1}, the map $t\mapsto b_{t}$ is continuous
in the operator norm topology; in particular, $\left\Vert b_{t}\right\Vert $
is a locally bounded function of $t.$ From this observation, it is easy to see
that the right-hand side of (\ref{dsdt1}) converges locally uniformly in $t.$
Thus, a standard result about interchange of limit and derivative (e.g.,
Theorem 7.17 in \cite{blueRudin}) shows that the term-by-term differentiation
is valid.

Now, in (\ref{dsdt1}), we let $k=j$ and $l=n-j-1,$ so that $n=k+l+1.$ Then $k$
and $l$ go from 0 to $\infty,$ and we get%
\[
\frac{\partial S}{\partial t}=\varepsilon\left(  \frac{1}{\varepsilon}%
\sum_{k=0}^{\infty}\frac{(-1)^{k}}{\varepsilon^{k}}\tau\lbrack(b_{t,\lambda
}^{\ast}b_{t,\lambda})^{k}]\right)  \left(  \frac{1}{\varepsilon}\sum
_{l=0}^{\infty}\frac{(-1)^{l}}{\varepsilon^{l}}\tau\lbrack b_{t}b_{t}^{\ast
}(b_{t,\lambda}b_{t,\lambda}^{\ast})^{l}]\right)  .
\]
(We may check that the power of $\varepsilon$ in the denominator is $k+l+1=n$
and that the power of $-1$ is $k+l=n-1.$) Thus, moving the sums inside the
traces and using (\ref{1zx}), we obtain the claimed form of $\partial
S/\partial t.$

We have now established the claimed formula for $\partial S/\partial t$ for
$\varepsilon$ in the right half-plane, provided $\left\vert \varepsilon
\right\vert $ is sufficiently large, depending on $t$ and $\lambda.$ Since,
also, $S(0,\lambda,\varepsilon)=\log(\left\vert \lambda-1\right\vert
^{2}+\varepsilon),$ we have, for sufficiently large $\left\vert \varepsilon
\right\vert ,$%
\begin{equation}
S(t,\lambda,\varepsilon)=\log(\left\vert \lambda-1\right\vert ^{2}%
+\varepsilon)+\int_{0}^{t}\varepsilon\tau\lbrack(b_{s,\lambda}^{\ast
}b_{s,\lambda}+\varepsilon)^{-1}]\tau\lbrack b_{s}b_{s}^{\ast}(b_{s,\lambda
}b_{s,\lambda}^{\ast}+\varepsilon)^{-1}]~ds. \label{Sintegrated}%
\end{equation}
We now claim that both sides of (\ref{Sintegrated}) are well-defined,
holomorphic functions of $\varepsilon,$ for $\varepsilon$ in the right
half-plane. This claim is easily established from the standard power-series
representation of the inverse:%
\begin{align*}
(A+\varepsilon+h)^{-1}  &  =(A+\varepsilon)^{-1}(1+h(A+\varepsilon)^{-1}%
)^{-1}\\
&  =(A+\varepsilon)^{-1}\sum_{n=0}^{\infty}(-1)^{n}h^{n}(A+\varepsilon)^{-n},
\end{align*}
and a similar power-series representation of the logarithm. Thus,
(\ref{Sintegrated}) actually holds for all $\varepsilon$ in the right
half-plane. Differentiating with respect to $t$ then establishes the claimed
formula (\ref{dsdt0}) for $dS/dt$ for all $\varepsilon$ in the right half-plane.
\end{proof}

\begin{lemma}
\label{otherDerivatives.lem}We have the following formulas for the derivatives
of $S$ with respect to $\varepsilon$ and $\lambda$:%
\begin{align*}
\frac{\partial S}{\partial\varepsilon}  &  =\tau\lbrack(b_{t,\lambda}^{\ast
}b_{t,\lambda}+\varepsilon)^{-1}]\\
\frac{\partial S}{\partial\lambda}  &  =-\tau\lbrack b_{t,\lambda}^{\ast
}(b_{t,\lambda}^{\ast}b_{t,\,\lambda}+\varepsilon)^{-1}]\\
\frac{\partial S}{\partial\bar{\lambda}}  &  =-\tau\lbrack b_{t,\lambda
}(b_{t,\lambda}^{\ast}b_{t,\lambda}+\varepsilon)^{-1}].
\end{align*}

\end{lemma}

\begin{proof}
We use the formula for the derivative of the trace of a logarithm (Lemma 1.1
in \cite{Br}):%
\[
\frac{d}{du}\tau\lbrack\log(f(u))]=\tau\left[  f(u)^{-1}\frac{df}{du}\right]
.
\]
(We emphasize that there is no such simple formula for the derivative of
$\log(f(u))$ without the trace, unless $df/du$ commutes with $f(u).$) The
lemma easily follows from this formula.
\end{proof}

We are now ready for the verification of the differential equation for $S$.

\begin{proof}
[Proof of Theorem \ref{thePDE.thm}]We note that
\[
b_{t,\lambda}(b_{t,\lambda}^{\ast}b_{t,\lambda}+\varepsilon)=(b_{t,\lambda
}b_{t,\lambda}^{\ast}+\varepsilon)b_{t,\lambda}.
\]
Multiplying by $(b_{t,\lambda}^{\ast}b_{t,\lambda}+\varepsilon)^{-1}$ on the
right and $(b_{t,\lambda}b_{t,\lambda}^{\ast}+\varepsilon)^{-1}$ on the left
gives a useful identity:%
\begin{equation}
(b_{t,\lambda}b_{t,\lambda}^{\ast}+\varepsilon)^{-1}b_{t,\lambda}%
=b_{t,\lambda}(b_{t,\lambda}^{\ast}b_{t,\lambda}+\varepsilon)^{-1}.
\label{basic1}%
\end{equation}
Replacing $b_{t,\lambda}$ by its adjoint gives another version of the
identity:%
\begin{equation}
b_{t,\lambda}^{\ast}(b_{t,\lambda}b_{t,\lambda}^{\ast}+\varepsilon
)^{-1}=(b_{t,\lambda}^{\ast}b_{t,\lambda}+\varepsilon)^{-1}b_{t,\lambda}%
^{\ast}. \label{basic2}%
\end{equation}
Note that in both (\ref{basic1}) and (\ref{basic2}), both side have the same
pattern of starred and unstarred variables, always with the two outer
variables being the same and the middle one being different.

We also claim that%
\begin{equation}
\tau\lbrack(b_{t,\lambda}^{\ast}b_{t,\lambda}+\varepsilon)^{-1}]=\tau
\lbrack(b_{t,\lambda}b_{t,\lambda}^{\ast}+\varepsilon)^{-1}].
\label{resolventSwitch}%
\end{equation}
To verify this identity for large $\varepsilon,$ we replace $z$ by
$b_{t,\lambda}^{\ast}b_{t,\lambda}$ in the series (\ref{1zx}) and note that by
the cyclic invariance of the trace,%
\[
\tau\lbrack(b_{t,\lambda}^{\ast}b_{t,\lambda})^{n}]=\tau\lbrack(b_{t,\lambda
}b_{t,\lambda}^{\ast})^{n}].
\]
The result for general $\varepsilon$ follows by an analyticity argument as in
the proof of Lemma \ref{tDerivative.lem}.

We start from the formula for $\partial S/\partial t$ in Lemma
\ref{tDerivative.lem}. Noting that%
\begin{align*}
b_{t}b_{t}^{\ast}  &  =(b_{t,\lambda}+\lambda)^{\ast}(b_{t,\lambda}+\lambda)\\
&  =b_{t,\lambda}b_{t,\lambda}^{\ast}+\lambda b_{t,\lambda}^{\ast}%
+\bar{\lambda}b_{t,\lambda}+\left\vert \lambda\right\vert ^{2},
\end{align*}
we expand the second factor on the right-hand side of (\ref{dsdt0}) as%
\begin{align*}
\tau\lbrack b_{t}b_{t}^{\ast}(b_{t,\lambda}b_{\lambda}^{\ast}+\varepsilon
)^{-1}]  &  =\tau\lbrack b_{t,\lambda}b_{t,\lambda}^{\ast}(b_{t,\lambda
}b_{t,\lambda}^{\ast}+\varepsilon)^{-1}]\\
&  +\lambda\tau\lbrack b_{t,\lambda}^{\ast}(b_{t,\lambda}b_{t,\lambda}^{\ast
}+\varepsilon)^{-1}]\\
&  +\bar{\lambda}\tau\lbrack b_{t,\lambda}(b_{t,\lambda}b_{t,\lambda}^{\ast
}+\varepsilon)^{-1}]\\
&  +\left\vert \lambda\right\vert ^{2}\tau\lbrack(b_{t,\lambda}b_{t,\lambda
}^{\ast}+\varepsilon)^{-1}].
\end{align*}
We then simplify the first term by writing $b_{t,\lambda}b_{t,\lambda}^{\ast
}=b_{t,\lambda}b_{t,\lambda}^{\ast}+\varepsilon-\varepsilon.$ In the middle
two terms, we use (\ref{basic1}), (\ref{basic2}), and cyclic invariance of the
trace. Using also (\ref{resolventSwitch}), we get
\begin{align}
\tau\lbrack b_{t}b_{t}^{\ast}(b_{t,\lambda}b_{\lambda}^{\ast}+\varepsilon
)^{-1}]  &  =1+(\left\vert \lambda\right\vert ^{2}-\varepsilon)\tau
\lbrack(b_{t,\lambda}^{\ast}b_{t,\lambda}+\varepsilon)^{-1}]\nonumber\\
&  +\lambda\tau\lbrack b_{t,\lambda}^{\ast}(b_{t,\lambda}^{\ast}b_{t,\lambda
}+\varepsilon)^{-1}]\nonumber\\
&  +\bar{\lambda}\tau\lbrack b_{t,\lambda}(b_{t,\lambda}^{\ast}b_{t,\lambda
}+\varepsilon)^{-1}]. \label{secondFactor}%
\end{align}

Thus,%
\begin{equation}
\frac{\partial S}{\partial t}=\varepsilon\tau\lbrack(b_{t,\lambda}^{\ast
}b_{t,\lambda}+\varepsilon)^{-1}](\text{all the terms in (\ref{secondFactor}%
)}). \label{dsdt2}%
\end{equation}
All terms on the right-hand side of (\ref{dsdt2}) are expressible using Lemma
\ref{otherDerivatives.lem} in terms of derivatives of $S,$ and the claimed
differential equation follows.
\end{proof}

\section{The Hamilton--Jacobi method\label{HJ.sec}}

\subsection{Setting up the method}

The equation (\ref{thePDE}) is a first-order, nonlinear PDE of
Hamilton--Jacobi type. (The reader may consult, for example, Section 3.3 in
the book of Evans \cite{Evans}, but we will give a brief self-contained
account of the theory in the proof of Proposition \ref{HJgeneral.prop}.) We
consider a Hamiltonian function obtained from the right-hand side of
(\ref{thePDE}) by replacing each partial derivative with momentum variable,
with an overall minus sign. Thus, we define%
\begin{equation}
H(a,b,\varepsilon,p_{a},p_{b},p_{\varepsilon})=-\varepsilon p_{\varepsilon
}(1+(a^{2}+b^{2})p_{\varepsilon}-\varepsilon p_{\varepsilon}-ap_{a}-bp_{b}).
\label{theHamiltonian}%
\end{equation}
We then consider Hamilton's equations for this Hamiltonian. That is to say, we
consider this system of six coupled ODEs:%
\begin{align}
\frac{da}{dt}  &  =\frac{\partial H}{\partial p_{a}};\quad~~\frac{db}%
{dt}=\frac{\partial H}{\partial p_{b}};\quad~~~\frac{d\varepsilon}{dt}%
=\frac{\partial H}{\partial p_{\varepsilon}};\nonumber\\
\frac{dp_{a}}{dt}  &  =-\frac{\partial H}{\partial a};\quad\frac{dp_{b}}%
{dt}=-\frac{\partial H}{\partial b};\quad\frac{dp_{\varepsilon}}{dt}%
=-\frac{\partial H}{\partial\varepsilon}. \label{theODEs}%
\end{align}
As convenient, we will let%
\[
\lambda(t)=a(t)+ib(t).
\]

The initial conditions for $a,$ $b,$ and $\varepsilon$ are arbitrary:
\begin{equation}
a(0)=a_{0};\quad b(0)=b_{0};\quad\varepsilon(0)=\varepsilon_{0},
\label{initialConditions1}%
\end{equation}
while those for $p_{a},$ $p_{b},$ and $p_{\varepsilon}$ are determined by
those for $a,$ $b,$ and $\varepsilon$ as follows:%
\begin{equation}
p_{a}(0)=2(a_{0}-1)p_{0};\quad p_{b}(0)=2b_{0}p_{0};\quad p_{\varepsilon
}(0)=p_{0}, \label{initialConditions2}%
\end{equation}
where%
\begin{equation}
p_{0}=\frac{1}{\left\vert \lambda_{0}-1\right\vert ^{2}+\varepsilon_{0}}%
=\frac{1}{(a_{0}-1)^{2}+b_{0}^{2}+\varepsilon_{0}}. \label{p0definition}%
\end{equation}
The motivation for (\ref{initialConditions2}) is that the momentum variables
$p_{a},$ $p_{b},$ and $p_{\varepsilon}$ will correspond to the derivatives of
$S$ along the curves $(a(t),b(t),\varepsilon(t))$; see (\ref{Sderivatives}).
Thus, the initial momenta are simply the derivatives of the initial value
(\ref{SinitialCond}) of $S,$ evaluated at $(a_{0},b_{0},\varepsilon_{0}).$

For future reference, we record the value $H_{0}$ of the Hamiltonian at time
$t=0$.

\begin{lemma}
\label{H0.lem}The value of the Hamiltonian at $t=0$ is%
\begin{equation}
H_{0}=-\varepsilon_{0}p_{0}^{2}. \label{H0Formula}%
\end{equation}

\end{lemma}

\begin{proof}
Plugging $t=0$ into (\ref{theHamiltonian}) and using (\ref{initialConditions2}%
) gives%
\[
H_{0}=-\varepsilon_{0}p_{0}(1+(a_{0}^{2}+b_{0}^{2})p_{0}-\varepsilon_{0}%
p_{0}-2a_{0}(a_{0}-1)p_{0}-2b_{0}^{2}p_{0}),
\]
which simplifies to%
\[
H_{0}=-\varepsilon_{0}p_{0}(1-p_{0}(a_{0}^{2}-2a_{0}+b_{0}^{2}+\varepsilon
_{0})).
\]
But using the formula (\ref{p0definition}) for $p_{0},$ we see that $a_{0}%
^{2}-2a_{0}+b_{0}^{2}+\varepsilon_{0}$ equals $1/p_{0}-1,$ from which
(\ref{H0Formula}) follows.
\end{proof}

The main result of this section is the following; the proof is given on p.
\pageref{HJproof}.

\begin{theorem}
\label{HJ.thm}Assume $\lambda_{0}\neq0$ and $\varepsilon_{0}>0.$ Suppose a
solution to the system (\ref{theODEs}) with initial conditions
(\ref{initialConditions1}) and (\ref{initialConditions2}) exists with
$\varepsilon(t)>0$ for $0\leq t<T.$ Then we have%
\begin{align}
S(t,\lambda(t),\varepsilon(t))  &  =\log(\left\vert \lambda_{0}-1\right\vert
^{2}+\varepsilon_{0})-\frac{\varepsilon_{0}t}{(\left\vert \lambda
_{0}-1\right\vert ^{2}+\varepsilon_{0})^{2}}\nonumber\\
&  +\log\left\vert \lambda(t)\right\vert -\log\left\vert \lambda
_{0}\right\vert \label{Sformula}%
\end{align}
for all $t\in\lbrack0,T).$ Furthermore, the derivatives of $S$ with respect to
$a,$ $b,$ and $\varepsilon$ satisfy%
\begin{align}
\frac{\partial S}{\partial\varepsilon}(t,\lambda(t),\varepsilon(t))  &
=p_{\varepsilon}(t);\nonumber\\
\frac{\partial S}{\partial a}(t,\lambda(t),\varepsilon(t))  &  =p_{a}%
(t);\nonumber\\
\frac{\partial S}{\partial b}(t,\lambda(t),\varepsilon(t))  &  =p_{b}(t).
\label{Sderivatives}%
\end{align}

\end{theorem}

Note that $S(t,\lambda,\varepsilon)$ is only defined for $\varepsilon>0.$
Thus, (\ref{Sformula}) and (\ref{Sderivatives}) only make sense as long as the
solution to (\ref{theODEs}) exists with $\varepsilon(t)>0.$

Since our objective is to compute $\Delta s_{t}(\lambda)=\partial^{2}%
s_{t}/\partial a^{2}+\partial^{2}s_{t}/\partial^{2}b^{2},$ the formula
(\ref{Sderivatives}) for the derivatives of $S$ will ultimately be of as great
importance as the formula (\ref{Sformula}) for $S$ itself. We emphasize that
we are not using the Hamilton--Jacobi method to \textit{construct} a solution
to (\ref{thePDE}); the function $S(t,\lambda,\varepsilon)$ is already defined
in (\ref{Sdefinition}) in terms of free probability and is known (Theorem
\ref{thePDE.thm}) to satisfy (\ref{thePDE}). Rather, we are using the
Hamilton--Jacobi method to \textit{analyze} a solution that is already known
to exist.

We begin by briefly recapping the general form of the Hamilton--Jacobi method.

\begin{proposition}
\label{HJgeneral.prop}Fix an open set $U\subset\mathbb{R}^{n},$ a
time-interval $[0,T],$ and a function $H(\mathbf{x},\mathbf{p}).$ Consider a
function $S(t,\mathbf{x})$ satisfying%
\begin{equation}
\frac{\partial S}{\partial t}=-H(\mathbf{x},\nabla_{\mathbf{x}}S),\quad
\mathbf{x}\in U,~t\in\lbrack0,T]. \label{HJeqn}%
\end{equation}
Consider a pair $(\mathbf{x}(t),\mathbf{p}(t))$ with $\mathbf{x}(t)\in U,$
$\mathbf{p}(t)\in\mathbb{R}^{n},$ and $t\in\lbrack0,T_{1}]$ with $T_{1}\leq
T.$ Assume this pair satisfies Hamilton's equations:%
\[
\frac{dx_{j}}{dt}=\frac{\partial H}{\partial p_{j}}(\mathbf{x}(t),\mathbf{p}%
(t));\quad\frac{dp_{j}}{dt}=-\frac{\partial H}{\partial x_{j}}(\mathbf{x}%
(t),\mathbf{p}(t))
\]
with initial conditions%
\begin{equation}
\mathbf{x}(0)=\mathbf{x}_{0};\quad\mathbf{p}(0)=(\nabla_{\mathbf{x}%
}S)(0,\mathbf{x}_{0}). \label{HJinit}%
\end{equation}
Then we have%
\begin{equation}
S(t,\mathbf{x}(t))=S(0,\mathbf{x}_{0})-H(\mathbf{x}_{0},\mathbf{p}_{0}%
)~t+\int_{0}^{t}\mathbf{p}(s)\cdot\frac{d\mathbf{x}}{ds}~ds
\label{HJformulaGen}%
\end{equation}
and%
\begin{equation}
(\nabla_{\mathbf{x}}S)(t,\mathbf{x}(t))=\mathbf{p}(t). \label{derivFormulaGen}%
\end{equation}

\end{proposition}

Again, we are not trying to construct solutions to (\ref{HJeqn}), but rather
to analyze a solution that is already assumed to exist.

\begin{proof}
Take an arbitrary (for the moment) smooth curve $\mathbf{x}(t)$ and note that%
\begin{align}
\frac{d}{dt}S(t,\mathbf{x}(t))  &  =\frac{\partial S}{\partial t}%
(t,\mathbf{x}(t))+\frac{\partial S}{\partial x_{j}}(t,\mathbf{x}%
(t))\frac{dx_{j}}{dt}\nonumber\\
&  =-H(\mathbf{x}(t),(\nabla_{\mathbf{x}}S)(t,\mathbf{x}(t)))+(\nabla
_{\mathbf{x}}S)(t,\mathbf{x}(t))\cdot\frac{d\mathbf{x}}{dt},
\label{totalDeriv1}%
\end{align}
where we use the Einstein summation convention. Let us use the notation%
\[
\mathbf{p}(t)=(\nabla_{\mathbf{x}}S)(t,\mathbf{x}(t));
\]
that is $p_{j}(t)=\partial S/\partial x_{j}(t,\mathbf{x}(t)).$ Then
(\ref{totalDeriv1}) may be rewritten as%
\begin{equation}
\frac{d}{dt}S(t,\mathbf{x}(t))=-H(\mathbf{x}(t),\mathbf{p}(t))+\mathbf{p}%
(t)\cdot\frac{d\mathbf{x}}{dt}. \label{totalDeriv2}%
\end{equation}
If we can choose $\mathbf{x}(t)$ so that $\mathbf{p}(t)$ is somehow
computable, then the right-hand side of (\ref{totalDeriv2}) would be known and
we could integrate to get $S(t,\mathbf{x}(t)).$

To see how we might be able to compute $\mathbf{p}(t),$ we try
differentiating:%
\begin{align}
\frac{dp_{j}}{dt}  &  =\frac{d}{dt}\frac{\partial S}{\partial x_{j}%
}(t,\mathbf{x}(t))\nonumber\\
&  =\frac{\partial^{2}S}{\partial t\partial x_{j}}(t,\mathbf{x}(t))+\frac
{\partial^{2}S}{\partial x_{k}\partial x_{j}}(t,\mathbf{x}(t))\frac{dx_{k}%
}{dt}. \label{pDeriv}%
\end{align}
Now, from (\ref{HJeqn}), we have%
\begin{align*}
\frac{\partial^{2}S}{\partial t\partial x_{j}}  &  =\frac{\partial^{2}%
S}{\partial x_{j}\partial t}\\
&  =-\frac{\partial}{\partial x_{j}}H(\mathbf{x},\nabla_{\mathbf{x}}S)\\
&  =-\frac{\partial H}{\partial x_{j}}(\mathbf{x},\nabla_{\mathbf{x}}%
S)-\frac{\partial H}{\partial p_{k}}(\mathbf{x},\nabla_{\mathbf{x}}%
S)\frac{\partial^{2}S}{\partial x_{j}\partial x_{k}}.
\end{align*}
Thus, (\ref{pDeriv}) becomes (suppressing the dependence on the path)%
\begin{equation}
\frac{dp_{j}}{dt}=-\frac{\partial H}{\partial x_{j}}+\left(  \frac{dx_{k}}%
{dt}-\frac{\partial H}{\partial p_{k}}\right)  \frac{\partial^{2}S}{\partial
x_{k}\partial x_{j}}. \label{pDeriv2}%
\end{equation}

If we now take $\mathbf{x}(t)$ to satisfy%
\begin{equation}
\frac{dx_{j}}{dt}=\frac{\partial H}{\partial p_{j}}, \label{Ham1}%
\end{equation}
the second term on the right-hand side of (\ref{pDeriv2}) vanishes, and we
find that $\mathbf{p}(t)$ satisfies%
\begin{equation}
\frac{dp_{j}}{dt}=-\frac{\partial H}{\partial x_{j}}. \label{Ham2}%
\end{equation}
With this choice of $\mathbf{x}(t),$ (\ref{totalDeriv2}) becomes%
\begin{equation}
\frac{d}{dt}S(t,\mathbf{x}(t))=-H(\mathbf{x}_{0},\mathbf{p}_{0})+\mathbf{p}%
(t)\cdot\frac{d\mathbf{x}}{dt}, \label{dSdtGeneral}%
\end{equation}
because $H$ is constant along the solutions to Hamilton's equations.

Note that not all solutions $(\mathbf{x}(t),\mathbf{p}(t))$ to Hamilton's
equations (\ref{Ham1}) and (\ref{Ham2}) will arise by the above method. After
all, we are assuming that $\mathbf{p}(t)=(\nabla_{\mathbf{x}}S)(t,\mathbf{x}%
(t)),$ from which it follows that the initial conditions $(\mathbf{x}%
_{0},\mathbf{p}_{0})$ will be of the form in (\ref{HJinit}).

On the other hand, suppose we take a pair $(\mathbf{x}_{0},\mathbf{p}_{0})$ as
in (\ref{HJinit}). Let us then take $\mathbf{x}(t)$ to be the solution to%
\begin{equation}
\frac{dx_{j}}{dt}=\frac{\partial H}{\partial p_{j}}(\mathbf{x}(t),(\nabla
_{\mathbf{x}}S)(t,\mathbf{x}(t))),\quad\mathbf{x}(0)=\mathbf{x}_{0},
\label{dxjdt}%
\end{equation}
where since $S$ is a fixed, \textquotedblleft known\textquotedblright%
\ function, this ODE for $\mathbf{x}(t)$ will have unique solutions for as
long as they exist. If we set $\mathbf{p}(t)=(\nabla_{\mathbf{x}%
}S)(t,\mathbf{x}(t)),$ then $\mathbf{p}(0)=\mathbf{p}_{0}$ as in
(\ref{HJinit}) and (\ref{dxjdt}) says that the pair $(\mathbf{x}%
(t),\mathbf{p}(t))$ satisfies the first of Hamilton's equations. Applying
(\ref{pDeriv2}) with this choice of $\mathbf{x}(t)$ shows that the pair
$(\mathbf{x}(t),\mathbf{p}(t))$ also satisfies the second of Hamilton's
equations. Thus, $(\mathbf{x}(t),\mathbf{p}(t))$ must be the \textit{unique}
solution to Hamilton's equations with the given initial condition
$(\mathbf{x}_{0},\mathbf{p}_{0}).$

We conclude that for \textit{any} solution to Hamilton's equations with
initial conditions of the form (\ref{HJinit}), the formula (\ref{totalDeriv2})
holds. Since, also, $H$ is constant along solutions to Hamilton's equations,
we may replace $H(\mathbf{x}(t),\mathbf{p}(t))$ by $H(\mathbf{x}%
_{0},\mathbf{p}_{0})$ in (\ref{totalDeriv2}), at which point, integration with
respect to $t$ gives (\ref{HJformulaGen}). Finally, (\ref{derivFormulaGen})
holds by the definition of $\mathbf{p}(t).$
\end{proof}

We are now ready for the proof of Theorem \ref{HJ.thm}.

\begin{proof}
[Proof of Theorem \ref{HJ.thm}]\label{HJproof}We apply Proposition \ref{HJgeneral.prop} with $n=3$ and the open set $U$ consisting of triples $(a,b,\varepsilon)$ with $\varepsilon>0$. 
The PDE (\ref{thePDE}) is of the
type in (\ref{HJeqn}), with $H$ given by (\ref{theHamiltonian}). The initial
conditions (\ref{initialConditions2}) are obtained by differentiating the
initial condition $S(0,\lambda,\varepsilon)=\log(\left\vert \lambda
-1\right\vert ^{2}+\varepsilon).$

We let $\mathbf{x}(t)=(a(t),b(t),\varepsilon(t))$ and $\mathbf{p}%
(t)=(p_{a}(t),p_{b}(t),p_{\varepsilon}(t)).$ For the case of the Hamiltonian
(\ref{theHamiltonian}), a simple computation shows that%
\begin{align*}
\mathbf{p}\cdot\frac{d\mathbf{x}}{dt}  &  =\mathbf{p}\cdot\nabla_{\mathbf{p}%
}H\\
&  =2H+\varepsilon p_{\varepsilon}\\
&  =2H_{0}+\varepsilon p_{\varepsilon}.
\end{align*}
Thus, the general formula (\ref{HJformulaGen}) becomes, in this case,%
\begin{equation}
S(t,\mathbf{x}(t))=S(0,\mathbf{x}_{0})+H(\mathbf{x}_{0},\mathbf{p}_{0}%
)~t+\int_{0}^{t}\varepsilon(s)p_{\varepsilon}(s)~ds. \label{HJformula}%
\end{equation}
But we also may compute that%
\begin{align*}
\frac{d}{dt}\log\left(  \sqrt{a^{2}+b^{2}}\right)   &  =\frac{1}{2}\frac
{1}{a^{2}+b^{2}}(a\dot{a}+b\dot{b})\\
&  =\frac{1}{2}\frac{1}{a^{2}+b^{2}}\left(  a\frac{\partial H}{\partial p_{a}%
}+b\frac{\partial H}{\partial p_{b}}\right) \\
&  =\varepsilon p_{\varepsilon}.
\end{align*}
Thus,
\begin{equation}
\int_{0}^{t}\varepsilon(s)p_{\varepsilon}(s)~ds=\log\left\vert \lambda
(t)\right\vert -\log\left\vert \lambda_{0}\right\vert .
\label{logLambdaIntegral}%
\end{equation}
If we now plug in the value of $S(0,\mathbf{x}_{0})=S(0,\lambda_{0}%
,\varepsilon_{0})$ and use Lemma \ref{H0.lem} along with the definition
(\ref{p0definition}) of $p_{0},$ we obtain (\ref{Sformula}). Finally,
(\ref{Sderivatives}) is just the general formula (\ref{derivFormulaGen}),
applied to the case at hand.
\end{proof}

\subsection{Constants of motion\label{constantsOfMotion.sec}}

We now identify several constants of motion for the system (\ref{theODEs}),
from which various useful formulas can be derived. Throughout the section, we
assume we have a solution to (\ref{theODEs}) with the initial conditions
(\ref{initialConditions1}) and (\ref{initialConditions2}), defined on a
time-interval of the form $0\leq t<T.$ We continue the notation $\lambda
(t)=a(t)+ib(t).$

\begin{proposition}
\label{constantsOfMotion.prop}Along any solution of (\ref{theODEs}), following
quantities remain constant:

\begin{enumerate}
\item The Hamiltonian $H,$

\item \label{angMomConserved.point}The \textquotedblleft angular
momentum\textquotedblright\ in the $(a,b)$ variables, namely $ap_{b}-bp_{a},$ and

\item \label{argConserved.point}The argument of $\lambda,$ assuming
$\lambda_{0}\neq0.$
\end{enumerate}
\end{proposition}

\begin{proof}
For any system of the form (\ref{theODEs}), the Hamiltonian $H$ itself is a
constant of motion, as may be verified easily from the equations. The
conservation of the angular momentum is a consequence of the invariance of $H$
under simultaneous rotations of $(a,b)$ and $(p_{a},p_{b})$; see Proposition
2.30 and Conclusion 2.31 in \cite{QMbook}. This result can also be verified by
direct computation from (\ref{theODEs}).

Finally, note from (\ref{logLambdaIntegral}) that if $\lambda_{0}\neq0,$ then
$\log\left\vert \lambda(t)\right\vert $ remains finite as long as the solution
to (\ref{theODEs}) exists, so that $\lambda(t)$ cannot pass through the
origin. We then compute that%
\[
\frac{d}{dt}\tan(\arg\lambda(t))=\frac{d}{dt}\frac{b}{a}=\frac{\dot{b}%
a-b\dot{a}}{a^{2}}=\frac{\varepsilon p_{\varepsilon}ba-b\varepsilon
p_{\varepsilon}a}{a^{2}}=0.
\]
(If $a=0,$ we instead compute the time-derivative of $\cot(\arg\lambda),$
which also equals zero.)
\end{proof}

\begin{proposition}
The Hamiltonian $H$ in (\ref{theHamiltonian}) in invariant under the
one-parameter group of symplectic linear transformations given by%
\begin{equation}
(a,b,\varepsilon,p_{a},p_{b},p_{\varepsilon})\mapsto(e^{\sigma/2}%
a,e^{\sigma/2}b,e^{\sigma}\varepsilon,e^{-\sigma/2}p_{a},e^{-\sigma/2}%
p_{b},e^{-\sigma}p_{\varepsilon}), \label{oneParameter}%
\end{equation}
with $\sigma$ varying over $\mathbb{R}.$ Thus, the generator of this family of
transformations, namely,%
\begin{equation}
\Psi:=\varepsilon p_{\varepsilon}+\frac{1}{2}(ap_{a}+bp_{b})
\label{constantOfMotion}%
\end{equation}
is a constant of motion for the system (\ref{theODEs}). The constant $\Psi$
may be computed in terms of $\varepsilon_{0}$ and $\lambda_{0}$ as%
\begin{equation}
\Psi=p_{0}(a_{0}(a_{0}-1)+b_{0}^{2}+\varepsilon_{0}) \label{PhiAtZero}%
\end{equation}
where $p_{0}$ is as in (\ref{p0definition}).
\end{proposition}

\begin{proof}
The claimed invariance of $H$ is easily checked from the formula
(\ref{theHamiltonian}). One can easily check that $\Psi$ is the generator of
this family. That is to say, if we replace $H$ by $\Psi$ in (\ref{theODEs}),
the solution is given by the map in\ (\ref{oneParameter}). Thus, by a simple
general result, $\Psi$ will be a constant of motion; see Conclusion 2.31 in
\cite{QMbook}. Of course, one can also check by direct computation that the
function in (\ref{constantOfMotion}) is constant along solutions to
(\ref{theODEs}). The expression (\ref{PhiAtZero}) then follows easily from the
initial conditions in (\ref{initialConditions2}).
\end{proof}

\begin{proposition}
\label{xp2.prop}For all $t,$ we have%
\begin{equation}
\varepsilon(t)p_{\varepsilon}(t)^{2}=\varepsilon_{0}p_{0}^{2}e^{-Ct},
\label{xp2Formula}%
\end{equation}
where $C=2\Psi-1$ and $\Psi$ is as in (\ref{constantOfMotion}). The constant
$C$ in (\ref{xp2Formula}) may be computed in terms of $\varepsilon_{0}$ and
$\lambda_{0}$ as%
\begin{equation}
C=p_{0}(\left\vert \lambda_{0}\right\vert ^{2}-1+\varepsilon_{0}%
)=\frac{\left\vert \lambda_{0}\right\vert ^{2}-1+\varepsilon_{0}}{\left\vert
\lambda_{0}-1\right\vert ^{2}+\varepsilon_{0}}. \label{CfromX0}%
\end{equation}

\end{proposition}

\begin{proof}
We compute that%
\begin{align}
\dot{\varepsilon}  &  =\frac{\partial H}{\partial p_{\varepsilon}}=\frac
{H}{p_{\varepsilon}}-\varepsilon p_{\varepsilon}(a^{2}+b^{2}-\varepsilon
)\nonumber\\
\dot{p}_{\varepsilon}  &  =-\frac{\partial H}{\partial\varepsilon}=-\frac
{H}{\varepsilon}-\varepsilon p_{\varepsilon}^{2} \label{pDot}%
\end{align}
and then that%
\begin{align*}
\frac{d}{dt}(\varepsilon p_{\varepsilon}^{2})  &  =\dot{\varepsilon
}p_{\varepsilon}^{2}+2\varepsilon p_{\varepsilon}\dot{p}_{\varepsilon}\\
&  =p_{\varepsilon}H-\varepsilon p_{\varepsilon}^{3}(a^{2}+b^{2}%
-\varepsilon)-2Hp_{\varepsilon}-2\varepsilon^{2}p_{\varepsilon}^{3}%
\end{align*}
This result simplifies to%
\begin{align*}
\frac{d}{dt}(\varepsilon p_{\varepsilon}^{2})  &  =\varepsilon p_{\varepsilon
}^{2}\left[  1-2\left(  \varepsilon p_{\varepsilon}+\frac{1}{2}(ap_{a}%
+bp_{b})\right)  \right] \\
&  =-\varepsilon p_{\varepsilon}^{2}(2\Psi-1).
\end{align*}
The unique solution to this equation is (\ref{xp2Formula}). The expression
(\ref{CfromX0}) is obtained by evaluating $\Psi$ at $t=0,$ using the initial
conditions (\ref{initialConditions2}), and simplifying.
\end{proof}

We now make an important application of preceding results.

\begin{theorem}
\label{atTimeTstar.thm}Suppose a solution to (\ref{theODEs}) exists with
$\varepsilon(t)>0$ for $0\leq t<t_{\ast},$ but that $\lim_{t\rightarrow
t^{\ast}}\varepsilon(t)=0.$ Then%
\begin{equation}
\lim_{t\rightarrow t_{\ast}}\log\left\vert \lambda(t)\right\vert
=\frac{Ct_{\ast}}{2}, \label{logLambda}%
\end{equation}
where $C=2\Psi-1$ is as in Proposition \ref{xp2.prop}. Furthermore, we have%
\begin{equation}
\lim_{t\rightarrow t_{\ast}}(ap_{a}+bp_{b})=\lim_{t\rightarrow t_{\ast}}%
\frac{2\log\left\vert \lambda(t)\right\vert }{t}+1. \label{apa}%
\end{equation}

\end{theorem}

Equation (\ref{apa}) is a key step in the derivation of our main result; see
Section \ref{outline.sec}. We will write (\ref{logLambda}) in a more explicit
way in Proposition \ref{tStar.prop}, after the time $t_{\ast}$ has been
determined. We note also from Proposition \ref{xp2.prop} that since
$\varepsilon(t)$ approaches zero as $t$ approaches $t_{\ast},$ then
\thinspace$p_{\varepsilon}(t)$ must be blowing up, so that $\varepsilon
(t)p_{\varepsilon}(t)^{2}$ can remain positive in this limit.

\begin{proof}
Using the constant of motion $\Psi$ in (\ref{constantOfMotion}), we can
rewrite the Hamiltonian $H$ as%
\begin{equation}
H=-\varepsilon p_{\varepsilon}(1+(a^{2}+b^{2})p_{\varepsilon}-2\Psi
+\varepsilon p_{\varepsilon}). \label{HfromPhi}%
\end{equation}
Now, by assumption, the variable $\varepsilon$ approaches zero as $t$
approaches $t_{\ast}.$ Furthermore, by Proposition \ref{xp2.prop},
$\varepsilon p_{\varepsilon}^{2}$ remains finite in this limit, so that
$\varepsilon p_{\varepsilon}=\sqrt{\varepsilon}\sqrt{\varepsilon
p_{\varepsilon}^{2}}$ tends to zero. Thus, in the $t\rightarrow t_{\ast}$
limit, the $\varepsilon p_{\varepsilon}$ terms in (\ref{HfromPhi}) vanish
while $\varepsilon p_{\varepsilon}^{2}$ remains finite, leaving us with
\[
H=-\lim_{t\rightarrow t\ast}\varepsilon p_{\varepsilon}^{2}(a^{2}+b^{2}).
\]
Since $H$ is a constant of motion, we may write this result as
\begin{equation}
\lim_{t\rightarrow t_{\ast}}(a^{2}+b^{2})=-\lim_{t\rightarrow t_{\ast}}%
\frac{H_{0}}{\varepsilon p_{\varepsilon^{2}}}=\lim_{t\rightarrow t_{\ast}%
}\frac{\varepsilon_{0}p_{0}^{2}}{\varepsilon p_{\varepsilon}^{2}}=e^{Ct_{\ast
}} \label{a2b2}%
\end{equation}
where we have used Lemma \ref{H0.lem} in the second equality and Proposition
\ref{xp2.prop} in the third. The formula (\ref{logLambda}) follows.

Meanwhile, as $t$ approaches $t_{\ast},$ the $\varepsilon p_{\varepsilon}$
term in the formula (\ref{constantOfMotion}) for $\Psi$ vanishes and we find,
using (\ref{logLambda}), that
\[
\lim_{t\rightarrow t^{\ast}}(ap_{a}+bp_{b})=2\Psi=C+1=\lim_{t\rightarrow
t_{\ast}}\frac{2\log\left\vert \lambda(t)\right\vert }{t}+1,
\]
as claimed in (\ref{apa}), where we have used (\ref{logLambda}) in the last equality.
\end{proof}

\subsection{Solving the equations\label{solving.sec}}

We now solve the system (\ref{theODEs}) subject to the initial conditions
(\ref{initialConditions1}) and (\ref{initialConditions2}). The formula in
Proposition \ref{xp2.prop} for $\varepsilon(t)p_{\varepsilon}(t)^{2}$ will be
a key tool. Although we are mainly interested in the case $\varepsilon_{0}>0,$
we will need in Section \ref{outside.sec} to allow $\varepsilon_{0}$ to be
slightly negative.

We begin by with the following elementary lemma.

\begin{lemma}
\label{quadratic.lem}Consider a number $a^{2}\in\mathbb{R}$ and let $a$ be
either of the two square roots of $a^{2}.$ Then the solution to the equation%
\begin{equation}
\dot{y}=y^{2}-a^{2} \label{quadraticEquation}%
\end{equation}
subject to the initial condition $y(0)=y_{0}>0$ is%
\begin{equation}
y(t)=\frac{y_{0}\cosh(at)-a\sinh(at)}{\cosh(at)-y_{0}\frac{\sinh(at)}{a}}
\label{quadraticSolution}%
\end{equation}

If $a^{2}\geq y_{0}^{2},$ the solution exists for all $t>0.$ If $a^{2}%
<y_{0}^{2}$, then $y(t)$ is a strictly increasing function of $t$ until the
first positive time $t_{\ast}$ at which the solution blows up. This time is
given by%
\begin{align}
t_{\ast}  &  =\frac{1}{a}\tanh^{-1}\left(  \frac{a}{y_{0}}\right)
\label{quadTstar1}\\
&  =\frac{1}{2a}\log\left(  \frac{1+a/y_{0}}{1-a/y_{0}}\right)  .
\label{quadTstar}%
\end{align}
Here, we use the principal branch of the inverse hyperbolic tangent, with
branch cuts $(-\infty,-1]$ and $[1,\infty)$ on the real axes, which
corresponds to using the principal branch of the logarithm. When $a=0,$ we
interpret the right-hand side of (\ref{quadTstar1}) or (\ref{quadTstar}) as
having its limiting value as $a$ approaches zero, namely $1/y_{0}.$
\end{lemma}

In passing from (\ref{quadTstar1}) to (\ref{quadTstar}), we have used the
standard formula for the inverse hyperbolic tangent,
\begin{equation}
\tanh^{-1}(x)=\frac{1}{2}\log\left(  \frac{1+x}{1-x}\right)  .
\label{arcTanhLog}%
\end{equation}
In (\ref{quadraticSolution}), we interpret $\sinh(at)/a$ as having the value
$t$ when $a=0.$ If $a^{2}<0,$ so that $a$ is pure imaginary, one can rewrite
the solution in terms of ordinary trigonometric functions, using the
identities $\cosh(i\alpha)=\cos\alpha$ and $\sinh(i\alpha)=i\sin\alpha.$ For
each fixed $t,$ the solution is an even analytic function of $a$ and therefore
an analytic function of $a^{2}.$

\begin{proof}
If $a$ is nonzero and real, we may integrate (\ref{quadraticEquation}) to
obtain%
\begin{align*}
t  &  =\frac{1}{2a}\int_{0}^{t}\left(  \frac{1}{y(\tau)-a}-\frac{1}{y(\tau
)+a}\right)  \dot{y}(\tau)~d\tau\\
&  =\frac{1}{2a}\left.  \log\left(  \frac{y(\tau)-a}{y(\tau)+a}\right)
\right\vert _{\tau=0}^{t}\\
&  =\frac{1}{2a}\log\left(  \frac{y(t)-a}{y(t)+a}\frac{y_{0}+a}{y_{0}%
-a}\right)  .
\end{align*}
It is then straightforward to solve for $y(t)$ and simplify to obtain
(\ref{quadraticSolution}). Similar computations give the result when $a$ is
zero (recalling that we interpret $\sinh(at)/a$ as equaling $t$ when $a=0$)
and when $a$ is nonzero and pure imaginary. Alternatively, one may check by
direct computation that the function on the right-hand side of
(\ref{quadraticSolution}) satisfies the equation (\ref{quadraticEquation}) for
all $a\in\mathbb{C}.$

Now, if $a^{2}\geq y_{0}^{2}>0,$ the denominator in (\ref{quadraticSolution})
is easily seen to be nonzero for all $t$ and there is no singularity. If
$a^{2}$ is positive but less than $y_{0}^{2},$ the denominator remains
positive until it becomes zero when $\tanh(at)=a/y_{0}.$ If $a^{2}$ is
negative, so that $a=i\alpha$ for some nonzero $\alpha\in\mathbb{R}$, we write
the solution using ordinary trigonometric functions as%
\begin{equation}
y(t)=y_{0}\frac{\cos(\alpha t)+\frac{\alpha}{y_{0}}\sin(\alpha t)}{\cos(\alpha
t)-\frac{y_{0}}{\alpha}\sin(\alpha t)}. \label{yTrig}%
\end{equation}
The denominator in (\ref{yTrig}) becomes zero at $\alpha t=\tan^{-1}%
(\alpha/y_{0})<\pi/2.$ Finally, if $a^{2}=0,$ the solution is $y(t)=y_{0}%
/(1-y_{0}t),$ which blows up at $t=1/y_{0}.$

It is then not hard to check that for all cases with $a^{2}<y_{0}^{2},$ the
blow-up time can be computed as $t_{\ast}=\frac{1}{a}\tanh^{-1}(a/y_{0}),$
where we use the principal branch of the inverse hyperbolic tangent, with
branch cuts $(-\infty,-1]$ and $[1,\infty)$ on the real axis. (At $a=0$ we
have a removable singularity with a value of $1/y_{0}.$) This recipe
corresponds to using the principal branch of the logarithm in the last
expression in (\ref{quadTstar}).
\end{proof}

We now apply Lemma \ref{quadratic.lem} to compute the $p_{\varepsilon}%
$-component of the solution to (\ref{theODEs}). We use the following
notations, some of which have been introduced previously:%
\begin{align}
p_{0}  &  =\frac{1}{\left\vert \lambda_{0}-1\right\vert ^{2}+\varepsilon_{0}%
}\label{p0Notation}\\
\delta &  =\frac{\left\vert \lambda_{0}\right\vert ^{2}+1+\varepsilon_{0}%
}{\left\vert \lambda_{0}\right\vert }\label{deltaNotation}\\
C  &  =2\Psi-1=p_{0}(\left\vert \lambda_{0}\right\vert ^{2}-1+\varepsilon
_{0})\label{CNotation}\\
y_{0}  &  =p_{0}+\frac{C}{2}=\frac{1}{2}p_{0}\left\vert \lambda_{0}\right\vert
\delta\label{y0Notation}\\
a^{2}  &  =C^{2}/4+\varepsilon_{0}p_{0}^{2}. \label{aNotation}%
\end{align}
We now make the following standing assumptions:%
\begin{align}
\lambda_{0}  &  \neq0\nonumber\\
p_{0}  &  >0\label{standingAssumptions}\\
\delta &  >0.\nonumber
\end{align}

We note that under these assumptions, $y_{0}$ is positive. Furthermore, we may
compute that%
\begin{equation}
a=\frac{1}{2}p_{0}\left\vert \lambda_{0}\right\vert \sqrt{\delta^{2}-4}.
\label{aFormula}%
\end{equation}
from which we obtain%
\begin{equation}
\frac{a^{2}}{y_{0}^{2}}=\frac{\delta^{2}-4}{\delta^{2}}<1, \label{aOverY0}%
\end{equation}
so that $a^{2}<y_{0}^{2}.$ Now, the assumptions $p_{0}>0$ and $\delta>0$ can
be written as $\varepsilon_{0}>-\left\vert \lambda_{0}-1\right\vert ^{2}$ and
$\varepsilon_{0}>-(1+\left\vert \lambda_{0}\right\vert ^{2}).$ Thus, for
$\lambda_{0}\neq0,$ the assumptions (\ref{standingAssumptions}) are always
satisfied if $\varepsilon_{0}>0.$ Furthermore, except when $\lambda_{0}=1,$
some negative values of $\varepsilon_{0}$ are allowed.

\begin{proposition}
\label{pxSolution.prop}Under the assumptions (\ref{standingAssumptions}), the
$p_{\varepsilon}$-component of the solution to (\ref{theODEs}) subject to the
initial conditions (\ref{initialConditions1}) and (\ref{initialConditions2})
is given by%
\begin{equation}
p_{\varepsilon}(t)=p_{0}\frac{\cosh(at)+\frac{2\left\vert \lambda
_{0}\right\vert -\delta}{\sqrt{\delta^{2}-4}}\sinh(at)}{\cosh(at)-\frac
{\delta}{\sqrt{\delta^{2}-4}}\sinh(at)}e^{-Ct} \label{pxFormula}%
\end{equation}
for as long as the solution to the system (\ref{theODEs}) exists. Here we
write $a$ as in (\ref{aFormula}) and we use the same choice of $\sqrt
{\delta^{2}-4}$ in the computation of $a$ as in the two times $\sqrt
{\delta^{2}-4}$ appears explicitly in (\ref{pxFormula}). If $\delta=2,$ we
interpret $\sinh(at)/\sqrt{\delta^{2}-4}$ as equaling $\frac{1}{2}%
p_{0}\left\vert \lambda_{0}\right\vert t.$

If $\varepsilon_{0}\geq0,$ the numerator in the fraction on the right-hand
side of (\ref{pxFormula}) is positive for all $t$. Hence when $\varepsilon
_{0}\geq0,$ we see that $p_{\varepsilon}(t)$ is positive for as long as the
solution exists and $1/p_{\varepsilon}(t)$ extends to a real-analytic function
of $t$ defined for all $t\in\mathbb{R}.$

The first time $t_{\ast}(\lambda_{0},\varepsilon_{0})$ at which the expression
on the right-hand side of (\ref{pxFormula}) blows up is%
\begin{align}
t_{\ast}(\lambda_{0},\varepsilon_{0})  &  =\frac{2(\delta-2\cos\theta_{0}%
)}{\sqrt{\delta^{2}-4}}\tanh^{-1}\left(  \frac{\sqrt{\delta^{2}-4}}{\delta
}\right) \label{tStar1}\\
&  =\frac{\delta-2\cos\theta_{0}}{\sqrt{\delta^{2}-4}}\log\left(  \frac
{\delta+\sqrt{\delta^{2}-4}}{\delta-\sqrt{\delta^{2}-4}}\right)  ,
\label{tStar2}%
\end{align}
where $\theta_{0}=\arg\lambda_{0}$ and $\sqrt{\delta^{2}-4}$ is either of the
two square roots of $\delta^{2}-4.$ The principal branch of the inverse
hyperbolic tangent should be used in (\ref{tStar1}), with branch cuts
$(-\infty,-1]$ and $[1,\infty)$ on the real axis, which corresponds to using
the principal branch of the logarithm in (\ref{tStar2}). When $\delta=2,$ we
interpret $t_{\ast}(\lambda_{0},\varepsilon_{0})$ as having its limiting value
as $\delta$ approaches 2, namely $\delta-2\cos\theta_{0}.$
\end{proposition}

Note that the expression
\[
\frac{1}{a}\tanh^{-1}\left(  \frac{a}{b}\right)
\]
is an even function of $a$ with $b$ fixed, with a removable singularity at
$a=0.$ This expression is therefore an analytic function of $a^{2}$ near the
origin. In particular, the value of $t_{\ast}(\lambda_{0},\varepsilon_{0})$
does not depend on the choice of square root of $\delta^{2}-4.$

\begin{proof}
[Proof of Proposition \ref{pxSolution.prop}]We assume at first that
$\varepsilon_{0}\neq0.$ We recall from Proposition \ref{xp2.prop} that
$\varepsilon(t)p_{\varepsilon}(t)^{2}$ is equal to $\varepsilon_{0}p_{0}%
^{2}e^{-Ct},$ which is never zero, since we assume $\varepsilon_{0}$ is
nonzero and $p_{0}$ is positive. Thus, as long as the solution to the system
(\ref{theODEs}) exists, both $\varepsilon(t)$ and $p_{\varepsilon}(t)$ must be
nonzero---and must have the same signs they had at $t=0.$ Using (\ref{pDot})
and the fact that $H$ is a constant of motion, we obtain
\[
\dot{p}_{\varepsilon}(t)=\frac{\varepsilon_{0}p_{0}^{2}}{\varepsilon
(t)}-\varepsilon_{0}p_{0}^{2}e^{-Ct}%
\]
But $\varepsilon_{0}p_{0}^{2}/\varepsilon(t)=p_{\varepsilon}(t)^{2}e^{Ct}$ and
we obtain%
\[
\dot{p}_{\varepsilon}(t)=p_{\varepsilon}(t)^{2}e^{Ct}-\varepsilon_{0}p_{0}%
^{2}e^{-Ct}.
\]
Then if $y(t)=e^{Ct}p_{\varepsilon}(t)+C/2,$ we find that $y$ satisfies
(\ref{quadraticEquation}). Thus, we obtain $p_{\varepsilon}%
(t)=(y(t)-C/2)e^{-Ct},$ where $y(t)$ is as in (\ref{quadraticSolution}), which
simplifies to the claimed formula for $p_{\varepsilon}.$ The same formula
holds for $\varepsilon_{0}=0$, by the continuous dependence of the solutions
on initial conditions. (It is also possible to solve the system (\ref{theODEs}%
) with $\varepsilon_{0}=0$ by postulating that $\varepsilon(t)$ is identically
zero and working out the equations for the other variables.)

In this paragraph only, we assume $\varepsilon_{0}\geq0.$ Then $a^{2}\geq0,$
with $a=0$ occurring only if $\varepsilon_{0}=0$ and $\left\vert \lambda
_{0}\right\vert =1,$ so that $\delta=2.$ In that case, the numerator on the
right-hand side of (\ref{pxFormula}) is identically equal to 1. If $a^{2}>0,$
then the numerator will always be positive provided that
\[
\left(  \frac{2\left\vert \lambda_{0}\right\vert -\delta}{\sqrt{\delta^{2}-4}%
}\right)  ^{2}\leq1,
\]
which is equivalent to%
\[
(\delta^{2}-4)-(2\left\vert \lambda_{0}\right\vert -\delta)^{2}\geq0.
\]
But a computation shows that
\begin{equation}
(\delta^{2}-4)-(2\left\vert \lambda_{0}\right\vert -\delta)^{2}=4\varepsilon
_{0}, \label{deltaX0identity}%
\end{equation}
and we are assuming $\varepsilon_{0}\geq0.$ Now, since the numerator in
(\ref{pxFormula}) is always positive, we conclude that $p_{\varepsilon}$
remains positive until it blows up.

For any value of $\varepsilon_{0},$ the blow-up time for the function on the
right-hand side of (\ref{pxFormula}) is computed by plugging the expression
(\ref{aOverY0}) for $a/y_{0}$ into the formula (\ref{quadTstar}), giving%
\begin{align*}
t_{\ast}(\lambda_{0},\varepsilon_{0})  &  =\frac{1}{y_{0}}\frac{1}{a/y_{0}%
}\tanh^{-1}\left(  \frac{a}{y_{0}}\right) \\
&  =\frac{2}{p_{0}\left\vert \lambda_{0}\right\vert \delta}\frac{\delta}%
{\sqrt{\delta^{2}-4}}\tanh^{-1}\left(  \frac{\sqrt{\delta^{2}-4}}{\delta
}\right)  .
\end{align*}
After computing that
\[
\frac{1}{p_{0}\left\vert \lambda_{0}\right\vert }=\frac{\left\vert \lambda
_{0}-1\right\vert ^{2}+\varepsilon_{0}}{\left\vert \lambda_{0}\right\vert
}=\delta-2\cos\theta_{0},
\]
we obtain the claimed formula (\ref{tStar1}) for $t_{\ast}(\lambda
_{0},\varepsilon_{0}).$
\end{proof}

\begin{remark}
\label{negativeX0.remark}If $\varepsilon_{0}<0,$ then numerator on the
right-hand side of (\ref{pxFormula}) can become zero. The time $\sigma$ at
which this happens is computed using (\ref{aFormula}) and
(\ref{deltaX0identity}) as%
\[
\sigma=\frac{2}{p_{0}\left\vert \lambda_{0}\right\vert \sqrt{\delta^{2}-4}%
}\tanh^{-1}\left(  -\left(  1+\frac{4\varepsilon_{0}}{(2\left\vert \lambda
_{0}\right\vert -\delta)^{2}}\right)  ^{1/2}\right)  .
\]
By considering separately the cases $\left\vert \lambda_{0}\right\vert \neq1$
and $\left\vert \lambda_{0}\right\vert =1,$ we can verify that $\sigma$ tends
to infinity, locally uniformly in $\lambda_{0},$ as $\varepsilon_{0}$ tends to
zero from below. Thus, for small negative values of $\varepsilon_{0},$ the
function on the right-hand side of (\ref{pxFormula}) will remain positive
until the time $t_{\ast}(\lambda_{0},\varepsilon_{0})$ at which it blows up.
\end{remark}

We now show that the whole system (\ref{theODEs}) has a solution up to the
time at which the function on the right-hand side of (\ref{pxFormula}) blows up.

\begin{proposition}
\label{existence.prop}Assume that $\varepsilon_{0}$ and $\lambda_{0}$ satisfy
the assumptions (\ref{standingAssumptions}). Assume further that if
$\varepsilon_{0}<0,$ then $\left\vert \varepsilon_{0}\right\vert $ is
sufficiently small that $p_{\varepsilon}$ remains positive until it blows up,
as in Remark \ref{negativeX0.remark}. Then the solution to the system
(\ref{theODEs}) exists up to the time $t_{\ast}(\lambda_{0},\varepsilon_{0})$
in Proposition \ref{pxSolution.prop}.

For any $\varepsilon_{0},$ we have
\begin{equation}
\lim_{t\rightarrow t_{\ast}(\lambda_{0},\varepsilon_{0})}\varepsilon(t)=0.
\label{xBecomesZero}%
\end{equation}
If $\varepsilon_{0}=0,$ the solution has $\varepsilon(t)\equiv0$ and
$\lambda(t)\equiv\lambda_{0}.$
\end{proposition}

\begin{proof}
Let $T$ be the maximum time such that the solution to (\ref{theODEs}) exists
on $[0,T)$. We now compute formulas for the solution on this interval. Recall
from Proposition \ref{pxSolution.prop} that if $\varepsilon_{0}\geq0,$ then
$p_{\varepsilon}(t)$ remains positive for as long as the solution exists; by
Remark \ref{negativeX0.remark}, the same assertion holds if $\varepsilon_{0}$
is small and negative.

Now, since $\varepsilon p_{\varepsilon}^{2}=\varepsilon_{0}p_{0}^{2}e^{-Ct}$,
we see that%
\begin{equation}
\varepsilon(t)=\frac{1}{p_{\varepsilon}(t)^{2}}\varepsilon_{0}p_{0}^{2}%
e^{-Ct}. \label{xOfT}%
\end{equation}
Since $p_{\varepsilon}(t)$ remains positive until it blows up, $\varepsilon
(t)$ remains bounded until time $t_{\ast}(\lambda_{0},\varepsilon_{0}),$ at
which time $\varepsilon(t)$ approaches zero, as claimed in (\ref{xBecomesZero}%
). We recall from Proposition \ref{constantsOfMotion.prop} that the argument
of $\lambda(t)$ remains constant. Then as in shown in (\ref{logLambdaIntegral}%
), we have%
\begin{equation}
\log\left\vert \lambda(t)\right\vert =\log\left\vert \lambda_{0}\right\vert
+\int_{0}^{t}\varepsilon(s)p_{\varepsilon}(s)~ds. \label{logLambdOfT}%
\end{equation}
Finally,
\begin{equation}
\frac{dp_{a}}{dt}=-\frac{\partial H}{\partial a}=-2a\varepsilon p_{\varepsilon
}^{2}+\varepsilon p_{\varepsilon}p_{a} \label{paOfT}%
\end{equation}
which is a first-order, linear equation for $p_{a},$ which can be solved using
an integrating factor. A similar calculation applies to $p_{b}.$

Suppose now that the existence time $T$ of the whole system were smaller than
the time $t_{\ast}(\lambda_{0},\varepsilon_{0})$ at which the right-hand side
of (\ref{pxFormula}) blows up. Then from the formulas (\ref{xOfT}),
(\ref{logLambdOfT}), and (\ref{paOfT}), we see that all functions involved
would remain bounded up to time $T$. But then by a standard result, $T$ could
not actually be the maximal time. The solution to the system (\ref{theODEs})
must therefore exist all the way up to time $t_{\ast}(\lambda_{0}%
,\varepsilon_{0}).$

Finally, we note that when $\varepsilon_{0}=0,$ (\ref{xOfT}) gives
$\varepsilon(t)\equiv0$ and (\ref{logLambdOfT}) gives $\left\vert
\lambda(t)\right\vert \equiv\left\vert \lambda_{0}\right\vert .$ Since also
the argument of $\lambda(t)$ is constant, we see that $\lambda(t)\equiv
\lambda_{0}.$
\end{proof}

\subsection{More about the lifetime of the solution\label{computingTstar.sec}}

In light of Propositions \ref{pxSolution.prop} and \ref{existence.prop}, the
lifetime of the solution to the system (\ref{theODEs}) is $t_{\ast}%
(\lambda_{0},\varepsilon_{0}),$ as computed in (\ref{tStar1}) or
(\ref{tStar2}). In this subsection, we (1) analyze the behavior of
$\log\left\vert \lambda(t)\right\vert $ as $t$ approaches $t_{\ast}%
(\lambda_{0},\varepsilon_{0}),$ (2) analyze the behavior of $t_{\ast}%
(\lambda_{0},\varepsilon_{0})$ as $\varepsilon_{0}$ approaches zero, and (3)
show that $t_{\ast}(\lambda_{0},\varepsilon_{0})$ is an increasing function of
$\varepsilon_{0}$ with $\lambda_{0}$ fixed.

\begin{proposition}
\label{tStar.prop}Assume that $\varepsilon_{0}$ and $\lambda_{0}$ satisfy the
assumptions (\ref{standingAssumptions}). Then
\begin{equation}
\lim_{t\rightarrow t_{\ast}(\lambda_{0},\varepsilon_{0})}\log\left\vert
\lambda(t)\right\vert =\frac{\delta-2/\left\vert \lambda_{0}\right\vert
}{2\sqrt{\delta^{2}-4}}\log\left(  \frac{\delta+\sqrt{\delta^{2}-4}}%
{\delta-\sqrt{\delta^{2}-4}}\right)  , \label{logLambdaFormula}%
\end{equation}
where $\delta$ is as in (\ref{deltaNotation}).
\end{proposition}

Notice that there is a strong similarity between the formula (\ref{tStar2})
for $t_{\ast}(\lambda_{0},\varepsilon_{0})$ and the expression on the
right-hand side of (\ref{logLambdaFormula}).

\begin{proof}
By (\ref{logLambda}) in Theorem \ref{atTimeTstar.thm}, we have $\lim
_{t\rightarrow t_{\ast}(\lambda_{0},\varepsilon_{0})}\log\left\vert
\lambda(t)\right\vert =Ct_{\ast}(\lambda_{0},\varepsilon_{0})/2,$ where by
(\ref{CfromX0}),
\[
C=\frac{(\left\vert \lambda_{0}\right\vert ^{2}-1+\varepsilon_{0})/\left\vert
\lambda_{0}\right\vert }{(\left\vert \lambda_{0}-1\right\vert ^{2}%
+\varepsilon_{0})/\left\vert \lambda_{0}\right\vert }=\frac{\delta
-2/\left\vert \lambda_{0}\right\vert }{\delta-2\cos\theta_{0}}.
\]
From this result and the second expression (\ref{tStar2}) for $t_{\ast
}(\lambda_{0},\varepsilon_{0})$, (\ref{logLambdaFormula}) follows easily.
\end{proof}

\begin{proposition}
\label{smallx0.prop}If $t_{\ast}(\lambda_{0},\varepsilon_{0})$ is defined by
(\ref{tStar2}), then for all nonzero $\lambda_{0}$ we have%
\begin{equation}
t_{\ast}(\lambda_{0},0)=T(\lambda_{0}), \label{tStarZero}%
\end{equation}
where the function $T$ is defined in (\ref{Tlambda}). Furthermore, when
$\varepsilon_{0}=0,$ we have
\begin{equation}
\lim_{t\rightarrow t_{\ast}(\lambda_{0},\varepsilon_{0})}\log\left\vert
\lambda(t)\right\vert =\log\left\vert \lambda_{0}\right\vert . \label{rhoZero}%
\end{equation}

\end{proposition}

Recall that the formula for $t_{\ast}(\lambda_{0},\varepsilon_{0})$ is defined
under the standing assumptions in (\ref{standingAssumptions}). Note that for
all $\lambda_{0}\neq0,$ the value $\varepsilon_{0}=0$ satisfies these assumptions.

Since $\log(x)/(x-1)\rightarrow1$ as $x\rightarrow1,$ we see that $t_{\ast
}(\lambda_{0},0)$ is a continuous function of $\lambda_{0}\in\mathbb{C}^{\ast
}.$ Comparing the formula for $t_{\ast}(\lambda_{0},0)$ to Theorem
\ref{domainGobbles.thm}, we have the following consequence.

\begin{corollary}
\label{smallx0.cor}For $\lambda_{0}\in\Sigma_{t},$ we have $t_{\ast}%
(\lambda_{0},0)<t,$ while for $\lambda_{0}\in\partial\Sigma_{t},$ we have
$t_{\ast}(\lambda_{0},0)=t$, and for $\lambda_{0}\notin\overline{\Sigma}_{t},$
we have $t_{\ast}(\lambda_{0},0)>t.$
\end{corollary}

\begin{proof}
[Proof of Proposition \ref{smallx0.prop}]In the limit as $\varepsilon
_{0}\rightarrow0,$ we have
\[
\delta=\frac{\left\vert \lambda_{0}\right\vert ^{2}+1}{\left\vert \lambda
_{0}\right\vert },
\]
and%
\[
\delta^{2}-4=\left(  \frac{\left\vert \lambda_{0}\right\vert ^{2}%
-1}{\left\vert \lambda_{0}\right\vert }\right)  ^{2},
\]
so that%
\begin{equation}
\sqrt{\delta^{2}-4}=\pm\frac{\left\vert \lambda_{0}\right\vert ^{2}%
-1}{\left\vert \lambda_{0}\right\vert }. \label{SqrtD2minus4}%
\end{equation}

In the case $\left\vert \lambda_{0}\right\vert =1,$ the limiting value of
$\delta$ is 2. We then make use of the elementary limit%
\begin{equation}
\lim_{\delta\rightarrow2^{+}}\frac{1}{\sqrt{\delta^{2}-4}}\log\left(
\frac{\delta+\sqrt{\delta^{2}-4}}{\delta-\sqrt{\delta^{2}-4}}\right)  =1.
\label{gLim}%
\end{equation}
Thus, using (\ref{tStar2}), we obtain in this case,
\[
\lim_{\varepsilon_{0}\rightarrow0}t_{\ast}(\lambda_{0},\varepsilon
_{0})=2-2\cos\theta_{0}=\left\vert \lambda_{0}-1\right\vert ^{2}%
,\quad\left\vert \lambda_{0}\right\vert =1,
\]
which agrees with the value of $T(\lambda_{0})$ when $\left\vert \lambda
_{0}\right\vert =1.$

In the case $\left\vert \lambda_{0}\right\vert \neq1,$ we note that the
quantity $(1/b)\log((a+b)/(a-b))$ is an even function of $b$ with $a$ fixed.
We may therefore choose the plus sign on the right-hand side of
(\ref{SqrtD2minus4}), regardless of the sign of $\left\vert \lambda
_{0}\right\vert ^{2}-1.$ We then obtain, using (\ref{tStar2}),
\begin{align}
\lim_{\varepsilon_{0}\rightarrow0}t_{\ast}(\lambda_{0},\varepsilon_{0})  &
=\frac{(\left\vert \lambda_{0}\right\vert ^{2}+1)/\left\vert \,\lambda
_{0}\right\vert -2\cos\theta_{0}}{(\left\vert \lambda_{0}\right\vert
^{2}-1)/\left\vert \lambda_{0}\right\vert }\log\left(  \frac{2\left\vert
\lambda_{0}\right\vert ^{2}/\left\vert \lambda_{0}\right\vert }{2/\left\vert
\lambda_{0}\right\vert }\right) \nonumber\\
&  =\frac{\left\vert \lambda_{0}\right\vert ^{2}+1-2\left\vert \lambda
_{0}\right\vert \cos\theta_{0}}{\left\vert \lambda_{0}\right\vert ^{2}-1}%
\log(\left\vert \lambda_{0}\right\vert ^{2})\nonumber\\
&  =T(\lambda_{0}). \label{tstarLim}%
\end{align}

A similar calculation, beginning from (\ref{logLambdaFormula}), establishes
(\ref{rhoZero}).
\end{proof}

\begin{remark}
If we began with (\ref{tStar1}) instead of (\ref{tStar2}), we would obtain by
similar reasoning%
\[
t_{\ast}(\lambda_{0},0)=\frac{2\left\vert \lambda_{0}-1\right\vert ^{2}%
}{\left\vert \lambda_{0}\right\vert ^{2}-1}\tanh^{-1}\left(  \frac{\left\vert
\lambda_{0}\right\vert ^{2}-1}{\left\vert \lambda_{0}\right\vert ^{2}%
+1}\right)  .
\]
Using (\ref{arcTanhLog}), this expression is easily seen to agree with
$T(\lambda_{0})$ but is more transparent in its behavior at $\left\vert
\lambda_{0}\right\vert =1.$
\end{remark}

\begin{proposition}
\label{monotoneTstar.prop}For each $\lambda_{0},$ the function $t_{\ast
}(\lambda_{0},\varepsilon_{0})$ is a strictly increasing function of
$\varepsilon_{0}$ for $\varepsilon_{0}\geq0,$ and
\[
\lim_{\varepsilon_{0}\rightarrow+\infty}t_{\ast}(\lambda_{0},\varepsilon
_{0})=+\infty.
\]

\end{proposition}

\begin{proof}
We note that the quantity $\delta$ in (\ref{deltaNotation}) is an increasing
function of $\varepsilon_{0}$ with $\lambda_{0}$ fixed, with $\delta$ tending
to infinity as $\varepsilon_{0}$ tends to infinity. We note also that if
$\varepsilon_{0}\geq0,$ then
\[
\delta\geq\left\vert \lambda_{0}\right\vert +\frac{1}{\left\vert \lambda
_{0}\right\vert }\geq2.
\]
It therefore suffices to show that for each angle $\theta_{0},$ the function
\begin{equation}
g_{\theta_{0}}(\delta):=\frac{\delta-2\cos\theta_{0}}{\sqrt{\delta^{2}-4}}%
\log\left(  \frac{\delta+\sqrt{\delta^{2}-4}}{\delta-\sqrt{\delta^{2}-4}%
}\right)  , \label{gDelta}%
\end{equation}
is strictly increasing, non-negative, continuous function of $\delta$ for
$\delta\geq2$ that tends to $+\infty$ as $\delta$ tends to infinity. Here when
$\delta=2,$ we interpret $g_{\theta_{0}}(\delta)$ as having the value
$2-2\cos\theta_{0},$ in accordance with the limit (\ref{gLim}).

Throughout the proof, we use the notation
\[
\gamma=\sqrt{\delta^{2}-4}.
\]
We note that%
\[
\lim_{\delta\rightarrow\infty}\frac{\delta-2\cos\theta_{0}}{\gamma}=1.
\]
Meanwhile, for large $\delta,$ we have
\begin{align*}
\delta-\gamma &  =\delta\left(  1-\sqrt{1-4/\delta^{2}}\right) \\
&  =\delta\left(  1-\left(  1-\frac{1}{2}\frac{4}{\delta^{2}}+O\left(
\frac{1}{\delta^{3}}\right)  \right)  \right) \\
&  =\frac{2}{\delta}+O\left(  \frac{1}{\delta^{3}}\right)  ,
\end{align*}
whereas%
\[
\delta+\gamma=2\delta+O\left(  \frac{1}{\delta}\right)  .
\]
Thus, $g_{\theta_{0}}(\delta)$ grows like $\log(\delta^{2})$ as $\delta
\rightarrow\infty$.

Our definition of $g_{\theta_{0}}(\delta)$ for $\delta=2,$ together with
(\ref{gLim}), shows that $g_{\theta_{0}}$ is non-negative and continuous
there. To show that $g_{\theta_{0}}$ is an increasing function of $\delta,$ we
show that $\partial g_{\theta_{0}}/\partial\delta$ is positive for $\delta>2.$
The derivative is computed, after simplification, as%
\[
\frac{\partial g_{\theta_{0}}}{\partial\delta}=\frac{2}{\gamma^{3}}\left(
(\delta-2\cos\theta_{0})\gamma+(\delta\cos\theta_{0}-2)\log\left(
\frac{\delta+\gamma}{\delta-\gamma}\right)  \right)  .
\]
Since this expression depends linearly on $\cos\theta_{0}$ with $\delta$
fixed, if it is positive when $\cos\theta_{0}=1$ and also when $\cos\theta
_{0}=-1,$ it will be positive always. Thus, it suffices to verify the
positivity of the functions%
\begin{equation}
(\delta-2)\left(  \gamma+\log\left(  \frac{\delta+\gamma}{\delta-\gamma
}\right)  \right)  \label{Derivative1}%
\end{equation}
and%
\begin{equation}
(\delta+2)\left(  \gamma-\log\left(  \frac{\delta+\gamma}{\delta-\gamma
}\right)  \right)  . \label{Derivative2}%
\end{equation}

Now, (\ref{Derivative1}) is clearly positive for all $\delta>2.$ Meanwhile, a
computation shows that
\[
\frac{d}{d\delta}\left(  \gamma-\log\left(  \frac{\delta+\gamma}{\delta
-\gamma}\right)  \right)  =\frac{\delta-2}{\gamma}>0
\]
and%
\[
\lim_{\delta\rightarrow2^{+}}\left(  \gamma-\log\left(  \frac{\delta+\gamma
}{\delta-\gamma}\right)  \right)  =0,
\]
from which we conclude that (\ref{Derivative2}) is also positive for all
$\delta>2.$
\end{proof}

\subsection{Surjectivity}

In Section \ref{inside.sec}, we will compute $s_{t}(\lambda):=\lim
_{\varepsilon\rightarrow0^{+}}S(t,\lambda,\varepsilon)$ for $\lambda$ in
$\Sigma_{t}.$ We will do so by evaluating $S$ (and its derivatives) along
curves of the form $(t,\lambda(t),\varepsilon(t))$ and then the taking the
limit as we approach the time $t_{\ast}$ when $\varepsilon(t)$ becomes zero.
For this method to be successful, we need the following result, whose proof
appears on p. \pageref{surjectivityProof}.

\begin{theorem}
\label{surjectivity.thm}Fix $t>0.$ Then for all $\lambda\in\Sigma_{t},$ there
exists a unique $\lambda_{0}\in\mathbb{C}$ and $\varepsilon_{0}>0$ such that
the solution to (\ref{theODEs}) with these initial conditions exists on
$[0,t)$ with $\lim_{u\rightarrow t^{-}}\varepsilon(u)=0$ and $\lim
_{u\rightarrow t^{-}}\lambda(u)=\lambda.$ For all $\lambda\in\Sigma_{t},$ the
corresponding $\lambda_{0}$ also belongs to $\Sigma_{t}.$

Define functions $\Lambda_{0}^{t}:\Sigma_{t}\rightarrow\Sigma_{t}$ and
$E_{0}^{t}:\Sigma_{t}\rightarrow(0,\infty)$ by letting $\Lambda_{0}%
^{t}(\lambda)$ and $E_{0}^{t}(\lambda)$ be the corresponding values of
$\lambda_{0}$ and $\varepsilon_{0},$ respectively. Then $\Lambda_{0}^{t}$ and
$E_{0}^{t}$ extend to continuous maps of $\overline{\Sigma}_{t}$ into
$\overline{\Sigma}_{t}$ and $[0,\infty)$, respectively, with the continuous
extensions satisfying $\Lambda_{t}(\lambda)=\lambda$ and $E_{0}^{t}%
(\lambda)=0$ for $\lambda\in\partial\Sigma_{t}.$
\end{theorem}

We first recall that we have shown (Proposition \ref{monotoneTstar.prop}) that
the lifetime of the path to be a strictly increasing function of
$\varepsilon_{0}\geq0$ with $\lambda_{0}$ fixed. If $\lambda_{0}$ is outside
$\Sigma_{t},$ then by Theorem \ref{domainGobbles.thm} and Proposition
\ref{smallx0.prop}, the lifetime is at least $t,$ even at $\varepsilon_{0}=0.$
(That is to say, $T(\lambda_{0})=t_{\ast}(\lambda_{0},0)\geq t$ for
$\lambda_{0}$ outside $\Sigma_{t}.$) Thus, for $\lambda_{0}$ outside
$\Sigma_{t}$, the lifetime cannot equal $t$ for $\varepsilon_{0}>0.$ On the
other hand, if $\lambda_{0}\in\Sigma_{t},$ then $t_{\ast}(\lambda_{0},0)<t$
and Proposition \ref{monotoneTstar.prop} tells us that there is a unique
$\varepsilon_{0}>0$ with $t_{\ast}(\lambda_{0},\varepsilon_{0})=t.$

\begin{lemma}
\label{lambda_t.lem}Fix $t>0.$ Define maps
\begin{align*}
\varepsilon_{0}^{t}  &  :\Sigma_{t}\rightarrow\lbrack0,\infty)\\
\lambda_{t}  &  :\Sigma_{t}\rightarrow\mathbb{C}\setminus\{0\}
\end{align*}
as follows. For $\lambda_{0}\in\Sigma_{t},$ we let $\varepsilon_{0}%
^{t}(\lambda_{0})$ denote the unique positive value of $\varepsilon_{0}$ for
which $t_{\ast}(\lambda_{0},\varepsilon_{0})=t.$ Then we set%
\[
\lambda_{t}(\lambda_{0})=\lim_{u\rightarrow t^{-}}\lambda(u),
\]
where $\lambda(\cdot)$ is computed with initial conditions $\lambda
(0)=\lambda_{0}$ and $\varepsilon(0)=\varepsilon_{0}^{t}(\lambda_{0}).$ Then
both $\varepsilon_{0}^{t}$ and $\lambda_{t}$ extend continuously from
$\Sigma_{t}$ to $\overline{\Sigma}_{t},$ with the extended maps satisfying
$\varepsilon_{0}^{t}(\lambda_{0})=0$ and $\lambda_{t}(\lambda_{0})=\lambda
_{0}$ for $\lambda_{0}\in\partial\Sigma_{t}.$ The extended map $\lambda_{t}$
is a homeomorphism of $\overline{\Sigma}_{t}$ to itself.
\end{lemma}

We note that the desired function $\Lambda_{0}^{t}$ in Theorem
\ref{surjectivity.thm} is the inverse function to $\lambda_{t}$ and that
$E_{0}^{t}(\lambda)=\varepsilon_{0}^{t}(\lambda_{t}^{-1}(\lambda)).$

Recall from Proposition \ref{constantsOfMotion.prop} that the argument of
$\lambda(t)$ is constant. By the formula (\ref{logLambda}) in Theorem
\ref{atTimeTstar.thm} together with the expression (\ref{CfromX0}) for the
constant $C,$ we can write%
\begin{equation}
\lambda_{t}(\lambda_{0})=\frac{\lambda_{0}}{\left\vert \lambda_{0}\right\vert
}e^{Ct/2}=\frac{\lambda_{0}}{\left\vert \lambda_{0}\right\vert }\exp\left(
\frac{t}{2}\frac{\left\vert \lambda_{0}\right\vert ^{2}-1+\varepsilon_{0}%
^{t}(\lambda_{0})}{\left\vert \lambda_{0}-1\right\vert ^{2}+\varepsilon
_{0}^{t}(\lambda_{0})}\right)  , \label{lambdaTformula1}%
\end{equation}
where we have used that $t_{\ast}(\lambda_{0},\varepsilon_{0}^{t}(\lambda
_{0}))=t.$ As noted in the proof of Proposition \ref{tStar.prop}, this formula
can also be written as%
\begin{equation}
\lambda_{t}(\lambda_{0})=\frac{\lambda_{0}}{\left\vert \lambda_{0}\right\vert
}\exp\left(  \frac{\delta-2/\left\vert \lambda_{0}\right\vert }{2\sqrt
{\delta^{2}-4}}\log\left(  \frac{\delta+\sqrt{\delta^{2}-4}}{\delta
-\sqrt{\delta^{2}-4}}\right)  \right)  , \label{lambdaTformula2}%
\end{equation}
where $\delta=(\left\vert \lambda_{0}\right\vert ^{2}+1+\varepsilon_{0}%
^{t}(\lambda_{0}))/\left\vert \lambda_{0}\right\vert .$

\begin{proof}
We start by trying to compute the function $\varepsilon_{0}^{t},$ which we
will do by finding the correct value of $\delta$ and then solving for
$\varepsilon_{0}^{t}.$ Recall that the lifetime $t_{\ast}(\lambda
_{0},\varepsilon_{0})$ is computed as $g_{\theta_{0}}(\delta),$ where $\delta$
is as in (\ref{deltaNotation}) and $g_{\theta_{0}}$ is as in (\ref{gDelta}).
As we have computed in (\ref{tstarLim}), we have
\[
g_{\theta_{0}}\left(  \frac{r_{0}^{2}+1}{r_{0}}\right)  =T(r_{0}e^{i\theta
_{0}}).
\]
Assume, then, that the ray with angle $\theta_{0}$ intersects $\Sigma_{t}$ and
let $r_{t}(\theta_{0})$ be the outer (for definiteness) radius at which this
ray intersects the boundary of $\Sigma_{t}.$ Then Theorem
\ref{domainGobbles.thm} tells us that $T(r_{t}(\theta_{0})e^{i\theta_{0}})=t,$
and we conclude that
\begin{equation}
g_{\theta_{0}}\left(  \frac{r_{t}(\theta_{0})^{2}+1}{r_{t}(\theta_{0}%
)}\right)  =t. \label{gDeltaIsT}%
\end{equation}

Consider, then, some $\lambda_{0}\in\Sigma_{t}$ with $\arg(\lambda_{0}%
)=\theta_{0}.$ By the formula (\ref{tStar2}), to find $\varepsilon_{0}$ with
$t_{\ast}(\lambda_{0},\varepsilon_{0})=t,$ we first find $\delta$ so that
$g_{\theta_{0}}(\delta)=t.$ (Note that the value of $\delta$ depends only on
the argument of $\lambda_{0}.$) We then adjust $\varepsilon_{0}$ so that
$(\left\vert \lambda_{0}\right\vert ^{2}+\varepsilon_{0}+1)/\left\vert
\lambda_{0}\right\vert =\delta.$ Since the correct value of $\delta$ is given
in (\ref{gDeltaIsT}), this means that we should choose $\varepsilon_{0}$ so
that%
\[
\frac{\left\vert \lambda_{0}\right\vert ^{2}+\varepsilon_{0}+1}{\left\vert
\lambda_{0}\right\vert }=\frac{r_{t}(\theta_{0})^{2}+1}{r_{t}(\theta_{0})}.
\]
We can solve this relation for $\varepsilon_{0}$ to obtain%
\begin{equation}
\varepsilon_{0}^{t}(\lambda_{0})=\left\vert \lambda_{0}\right\vert \left(
\frac{r_{t}(\arg\lambda_{0})^{2}+1}{r_{t}(\arg\lambda_{0})}-\frac{\left\vert
\lambda_{0}\right\vert ^{2}+1}{\left\vert \lambda_{0}\right\vert }\right)  .
\label{X0OfLambda0}%
\end{equation}
Now, we have shown that $r_{t}(\theta)$ is continuous for the full range of
angles $\theta$ occurring in $\overline{\Sigma}_{t}.$ Since $0$ is not in
$\overline{\Sigma}_{t}$, we can then see that the formula (\ref{X0OfLambda0})
is well defined and continuous on all of $\overline{\Sigma}_{t}.$ For
$\lambda_{0}\in\partial\Sigma_{t},$ we have that $\left\vert \lambda
_{0}\right\vert $ equals $r_{t}(\arg\lambda_{0})$ or $1/r_{t}(\arg\lambda
_{0}),$ so that $\varepsilon_{0}^{t}(\lambda_{0})$ equals zero.

Now, the point 0 is always outside $\overline{\Sigma}_{t},$ while the point
$1$ is always in $\Sigma_{t}$ and therefore not on the boundary of $\Sigma
_{t}.$ Thus, since $\varepsilon_{0}^{t}$ is continuous on $\overline{\Sigma
}_{t}$ and zero precisely on the boundary, we see from (\ref{lambdaTformula1})
that $\lambda_{t}$ is continuous on $\overline{\Sigma}_{t}.$ Furthermore, on
$\partial\Sigma_{t},$ we compute $\lambda_{t}(\lambda_{0})$ by putting
$\varepsilon_{0}^{t}(\lambda_{0})=0$ in (\ref{lambdaTformula1}). Suppose now
that $\lambda_{0}$ is in $\partial\Sigma_{t}.$ Then $\varepsilon_{0}%
^{t}(\lambda_{0})=0$ and, by Theorem \ref{domainGobbles.thm}, the function
$T(\lambda_{0})$ in (\ref{Tlambda}) has the value $t,$ so that
\[
\frac{t}{2}\frac{\left\vert \lambda_{0}\right\vert ^{2}-1}{\left\vert
\lambda_{0}-1\right\vert ^{2}}=\log(\left\vert \lambda_{0}\right\vert ).
\]
Thus, from (\ref{lambdaTformula1}), we see that $\lambda_{t}(\lambda
_{0})=\lambda_{0}.$

Consider an angle $\theta_{0}$ for which the ray $\mathrm{Ray}(\theta_{0})$
with angle $\theta_{0}$ intersects $\overline{\Sigma}_{t}$ and let $\delta$ be
chosen so that $g_{\theta_{0}}(\delta)=t,$ noting again that the value of
$\delta$ depends only on $\theta_{0}=\arg\lambda_{0}.$ We now observe from
(\ref{lambdaTformula2}) that $\left\vert \lambda_{t}(\lambda_{0})\right\vert $
is a strictly increasing function of $\left\vert \lambda_{0}\right\vert $ with
$\delta$ fixed. Thus, $\lambda_{t}$ is a strictly increasing function of the
interval $\mathrm{Ray}(\theta_{0})\cap\overline{\Sigma}_{t}$ into
$\mathrm{Ray}(\theta_{0})$ that fixes the endpoints. Thus, actually,
$\lambda_{t}$ maps this interval bijectively into itself. Since this holds for
all $\theta_{0},$ we conclude that $\lambda_{t}$ maps $\overline{\Sigma}_{t}$
bijectively into itself. The continuity of the inverse then holds because
$\lambda_{t}$ is continuous and $\overline{\Sigma}_{t}$ is compact.
\end{proof}

\begin{proof}
[Proof of Theorem \ref{surjectivity.thm}]\label{surjectivityProof}We have
noted before the statement of Lemma \ref{lambda_t.lem} that if the desired
pair $(\lambda_{0},\varepsilon_{0})$ exists, $\lambda_{0}$ must be in
$\Sigma_{t}.$ The lemma then tells us that a unique pair $(\lambda
_{0},\varepsilon_{0})$ exists with $\lambda_{0}\in\Sigma_{t}.$ We compute
$\Lambda_{0}^{t}(\lambda)$ as $\lambda_{t}^{-1}(\lambda)$ and $E_{0}%
^{t}(\lambda)$ as $\varepsilon_{0}^{t}(\lambda_{t}^{-1}(\lambda)),$ both of
which extend continuously to $\overline{\Sigma}_{t}$. For $\lambda\in
\partial\Sigma_{t},$ we have $\lambda_{t}^{-1}(\lambda)=\lambda$ and
$\varepsilon_{0}^{t}(\lambda_{t}^{-1}(\lambda))=\varepsilon_{0}^{t}%
(\lambda)=0.$
\end{proof}

\section{Letting $\varepsilon$ tend to zero}

\subsection{Outline\label{outline.sec}}

Our goal is to compute the Laplacian with respect to $\lambda$ of the function
$s_{t}(\lambda):=\lim_{\varepsilon\rightarrow0^{+}}S(t,\lambda,\varepsilon),$
using the Hamilton--Jacobi method of Theorem \ref{HJ.thm}. We want the curve
$\varepsilon(\cdot)$ occurring in (\ref{Sformula}) and (\ref{Sderivatives}) to
approach zero at time $t$; a simple way we might try to accomplish this is to
let the initial condition $\varepsilon_{0}$ approach zero. Suppose, then, that
$\varepsilon_{0}$ is very small. Using various formulas from Section
\ref{solving.sec}, we then find that \textit{for as long as the solution to
the system (\ref{theODEs}) exists}, the whole curve $\varepsilon(\cdot)$ will
be small and the whole curve $\lambda(\cdot)$ will be approximately constant.
Thus, by taking $\varepsilon_{0}\approx0$ and $\lambda_{0}\approx\lambda,$ we
obtain a curve with $\varepsilon(t)\approx0$ and $\lambda(t)\approx\lambda.$
We may then hope to compute $s_{t}(\lambda)$ by letting $\lambda_{0}$ and
$\lambda(t)$ approach $\lambda$ and $\varepsilon_{0}$ approach zero in the
Hamilton--Jacobi formula (\ref{Sformula}), with the result that%
\begin{equation}
s_{t}(\lambda)=\log(\left\vert \lambda-1\right\vert ^{2}). \label{stOutside}%
\end{equation}

It is essential to note, however, that this approach is only valid if the
solution to system (\ref{theODEs}) exists up to time $t.$ Corollary
\ref{smallx0.cor} tells us that for $\varepsilon_{0}\approx0,$ the solution
will exist beyond time $t$ provided $\lambda$ is outside $\overline{\Sigma
}_{t}.$ Thus, we expect that for $\lambda$ outside $\overline{\Sigma}_{t},$
the function $s_{t}$ will be given by (\ref{stOutside}) and therefore that
$\Delta s_{t}$ will be zero. (The function $\log(\left\vert \lambda
-1\right\vert ^{2})$ is harmonic except at the point $\lambda=1,$ which is
always inside $\Sigma_{t}.$)

To analyze $s_{t}(\lambda)$ for $\lambda$ inside $\Sigma_{t},$ we first make
use of the surjectivity result in Theorem \ref{surjectivity.thm}. The theorem
says that for each $t>0$ and $\lambda\in\Sigma_{t},$ there exist
$\varepsilon_{0}>0$ and $\lambda_{0}\in\Sigma_{t}$ such that $\varepsilon(u)$
approaches 0 and $\lambda(u)$ approaches $\lambda$ as $u$ approaches $t.$ We
then use the formula (\ref{apa}) in Theorem \ref{atTimeTstar.thm}. In light of
the second Hamilton--Jacobi formula (\ref{Sderivatives}), we can write
(\ref{apa}) as%
\begin{align}
\lim_{u\rightarrow t}\left(  a\frac{\partial S}{\partial a}+b\frac{\partial
S}{\partial b}\right)  (u,\lambda(u),\varepsilon(u))  &  =\lim_{u\rightarrow
t}\frac{2\log\left\vert \lambda(t)\right\vert }{t}+1\nonumber\\
&  =\frac{2\log\left\vert \lambda\right\vert }{t}+1. \label{sasbInside}%
\end{align}
Once we have established enough regularity in the function $S(t,\lambda
,\varepsilon)$ near $\varepsilon=0,$ we will be able to identify the left-hand
side of (\ref{sasbInside}) with the corresponding derivative of $s_{t}$,
giving the following explicit formula for one of the derivatives of $s_{t}$:%
\begin{equation}
\left(  a\frac{\partial s_{t}}{\partial a}+b\frac{\partial s_{t}}{\partial
b}\right)  (\lambda)=\frac{2\log\left\vert \lambda\right\vert }{t}+1.
\label{sasbInside2}%
\end{equation}

We now compute in logarithmic polar coordinates, with $\rho=\log\left\vert
\lambda\right\vert $ and $\theta=\arg\lambda.$ We may recognize the left-hand
side of (\ref{sasbInside2}) as the derivative of $s_{t}$ with respect to
$\rho,$ giving%
\begin{equation}
\frac{\partial s_{t}}{\partial\rho}=\frac{2\rho}{t}+1 \label{dstDrhoInside}%
\end{equation}
for points inside $\Sigma_{t}.$ Remarkably, $\partial s_{t}/\partial\rho$ is
independent of $\theta$! Thus,%
\[
\frac{\partial}{\partial\rho}\frac{\partial s_{t}}{\partial\theta}%
=\frac{\partial}{\partial\theta}\frac{\partial s_{t}}{\partial\rho}=0,
\]
meaning that $\partial s_{t}/\partial\theta$ is independent of $\rho.$

Now, we will show in Section \ref{boundary.sec} that the first derivatives of
$s_{t}$ have the same value as we approach a point $\lambda\in\partial
\Sigma_{t}$ from the inside as when we approach $\lambda$ from the outside. We
can therefore give a complete description of the function $\partial
s_{t}/\partial\theta$ on $\Sigma_{t}$ as follows. It is the unique function on
$\Sigma_{t}$ that is independent of $\rho$ (or, equivalently, independent of
$r=\left\vert \lambda\right\vert $) and whose boundary values agree
\begin{equation}
\frac{\partial}{\partial\theta}\log(\left\vert \lambda-1\right\vert
^{2})=\frac{2b}{\left\vert \lambda-1\right\vert ^{2}}=\frac{2r\sin\theta
}{r^{2}+1-2r\cos\theta}. \label{angularDeriv}%
\end{equation}
Since the points on the outer boundary of $\Sigma_{t}$ have the polar form
$(r_{t}(\theta),\theta),$ we conclude that%
\[
\frac{\partial s_{t}}{\partial\theta}=\frac{2r_{t}(\theta)\sin\theta}%
{r_{t}(\theta)^{2}+1-2r_{t}(\theta)\cos\theta}.
\]
From this result, the expression (\ref{dstDrhoInside}), and the formula for
the Laplacian in logarithmic polar coordinates, we obtain%
\begin{align*}
\Delta s_{t}(\lambda)  &  =\frac{1}{\left\vert \lambda\right\vert ^{2}}\left(
\frac{\partial^{2}s_{t}}{\partial\rho^{2}}+\frac{\partial^{2}s_{t}}%
{\partial\theta^{2}}\right) \\
&  =\frac{1}{\left\vert \lambda\right\vert ^{2}}\left(  \frac{2}{t}%
+\frac{\partial}{\partial\theta}\frac{2r_{t}(\theta)\sin\theta}{r_{t}%
(\theta)^{2}+1-2r_{t}(\theta)\cos\theta}\right)
\end{align*}
for points inside $\Sigma_{t},$ accounting for the formula in Theorem
\ref{main.thm}.

We now briefly discuss what is needed to make the preceding arguments
rigorous. If $\lambda$ is outside $\overline{\Sigma}_{t}$ and $\varepsilon$ is
small and positive, we need to know that we can find a $\lambda_{0}$ close to
$\lambda$ and a small, positive $\varepsilon_{0}$ such that with these initial
conditions, $\varepsilon(t)=\varepsilon$ and $\lambda(t)=\lambda.$ To show
this, we apply the inverse function theorem to the map $U_{t}(\lambda
_{0},\varepsilon_{0}):=(\lambda(t),\varepsilon(t))$ in a neighborhood of the
point $(\lambda_{0},\varepsilon_{0})=(\lambda,0).$

For $\lambda$ inside $\Sigma_{t},$ we need to know first that $S(t,\lambda
,\varepsilon)$ is continuous---in all three variables---up to $\varepsilon=0.$
After all, $s_{t}(\lambda)$ is defined letting $\varepsilon$ tend to zero in
the expression $S(t,\lambda,\varepsilon),$ \textit{with }$t$\textit{ and
}$\lambda$\textit{ fixed}. But the Hamilton--Jacobi formula (\ref{Sformula})
gives a formula for $S(u,\lambda(u),\varepsilon(u))$, in which the first two
variables in $S$ are not remaining constant. Furthermore, we want to apply
also the Hamilton--Jacobi formula (\ref{Sderivatives}) for the derivatives of
$S,$ which means we need also continuity of the derivatives of $S$ with
respect to $\lambda$ up to $\varepsilon=0.$ Using another inverse function
theorem argument, we will show that after making the change of variable
$z=\sqrt{\varepsilon},$ the function $S$ will extend smoothly up to
$\varepsilon=z=0,$ from which the needed regularity will follow.

We use the following notation throughout the section.

\begin{notation}
We will let%
\begin{equation}
\varepsilon(t;\lambda_{0},\varepsilon_{0})\nonumber
\end{equation}
denote the $\varepsilon$-component of the solution to (\ref{theODEs}) with
$\lambda(0)=\lambda_{0}$ and $\varepsilon(0)=\varepsilon_{0}$ (and with
initial values of the momenta given by (\ref{initialConditions2})), and
similarly for the other components of the solution.
\end{notation}

\subsection{Outside $\overline{\Sigma}_{t}$\label{outside.sec}}

The goal of this subsection is to prove the following result.

\begin{theorem}
\label{outside.thm}Fix a pair $(t,\lambda)$ with $\lambda$ outside
$\overline{\Sigma}_{t}.$ Then
\begin{equation}
s_{t}(\lambda):=\lim_{\varepsilon\rightarrow0^{+}}S(t,\lambda,\varepsilon
)=\log(\left\vert \lambda-1\right\vert ^{2}). \label{outsideLimit}%
\end{equation}
Thus,%
\[
\Delta s_{t}(\lambda)=0
\]
whenever $\lambda$ is outside $\overline{\Sigma}_{t}.$
\end{theorem}

As we have discussed in Section \ref{outline.sec}, the idea is that for
$\lambda$ outside $\overline{\Sigma}_{t}$ and $\varepsilon$ small and
positive, we should try to find a $\lambda_{0}$ close to $\lambda$ and a
small, positive $\varepsilon_{0}$ such that $\varepsilon(u)$ and $\lambda(u)$
will approach $0$ and $\lambda,$ respectively, as $u$ approaches $t.$ To that
end, we define, for each $t>0,$ a map $U_{t}$ from an open subset of
$\mathbb{R}\times\mathbb{C}$ into $\mathbb{R}\times\mathbb{C}$ by%
\[
U_{t}(\lambda_{0},\varepsilon_{0})=(\lambda(t;\lambda_{0},\varepsilon
_{0}),\varepsilon(t;\lambda_{0},\varepsilon_{0})).
\]
We wish to evaluate the derivative of this map at the point $(\lambda
_{0},\varepsilon_{0})=(\lambda,0).$ For this idea to make sense,
$\lambda(t;\lambda_{0},\varepsilon_{0})$ and $\varepsilon(t;\lambda
_{0},\varepsilon_{0})$ must be defined in a neighborhood of $(\lambda,0)$; it
is for this reason that we have allowed $\varepsilon_{0}$ to be negative in
Section \ref{solving.sec}.

The domain of $U_{t}$ consists of pairs $(\lambda_{0},\varepsilon_{0})$ such
that (1) the assumptions (\ref{standingAssumptions}) are satisfied; (2) the
function $p_{\varepsilon}(\cdot)$ remains positive until it blows up, as in
Remark \ref{negativeX0.remark}; and (3) we have $t_{\ast}(\lambda
_{0},\varepsilon_{0})>t.$ We note that these conditions allow $\varepsilon
_{0}$ to be slightly negative and that all the results of Section
\ref{solving.sec} hold under these conditions. We also note that by
Proposition \ref{existence.prop}, if $\varepsilon_{0}=0$, then $\varepsilon
(t)\equiv0$ and $\lambda(t)\equiv\lambda_{0}$; thus,%
\begin{equation}
U_{t}(\lambda_{0},0)=(\lambda_{0},0). \label{UzeroLambda0}%
\end{equation}

We now fix a pair $(t,\lambda)$ with $\lambda$ outside of $\overline{\Sigma
}_{t}$ (so that $\lambda\neq1$). By Corollary \ref{smallx0.cor}, we then have
$t_{\ast}(\lambda,0)>t.$

\begin{lemma}
The Jacobian of $U_{t}$ at $(\lambda,0)$ has the form
\begin{equation}
U_{t}^{\prime}(\lambda,0)=\left(
\begin{array}
[c]{cc}%
I_{2\times2} & \frac{\partial\lambda}{\partial\varepsilon_{0}}(t;\lambda,0)\\
0 & \frac{\partial\varepsilon}{\partial\varepsilon_{0}}(t;\lambda,0)
\end{array}
\right)  \label{Uprime}%
\end{equation}
with $\partial\varepsilon/\partial\varepsilon_{0}(t;\lambda,0)>0.$ In
particular, the inverse function theorem applies at $(\lambda,0).$
\end{lemma}

\begin{proof}
The claimed form of the second column of $U_{t}^{\prime}(\lambda,0)$ follows
immediately from (\ref{UzeroLambda0}). We then compute from (\ref{xOfT}) that%
\begin{align}
\frac{\partial\varepsilon(t;\lambda_{0},\varepsilon_{0})}{\partial
\varepsilon_{0}}(0,\lambda_{0})  &  =\left.  \frac{1}{p_{\varepsilon}(t)^{2}%
}p_{0}^{2}e^{-Ct}+\varepsilon_{0}\frac{\partial}{\partial\varepsilon_{0}%
}\left[  \frac{1}{p_{\varepsilon}(t)^{2}}p_{0}^{2}e^{-Ct}\right]  \right\vert
_{\varepsilon_{0}=0}\nonumber\\
&  =\frac{1}{p_{\varepsilon}(t)^{2}}p_{0}^{2}e^{-Ct}, \label{dxdx0}%
\end{align}
which is positive.
\end{proof}

\begin{proof}
[Proof of Theorem \ref{outside.thm}]We note that the inverse of the matrix in
(\ref{Uprime}) will have a positive entry in the bottom right corner, meaning
that $U_{t}^{-1}$ has the property that $\partial\varepsilon_{0}%
/\partial\varepsilon>0.$ It follows that the $\varepsilon_{0}$-component of
$U_{t}^{-1}(\lambda,\varepsilon)$ will be positive for $\varepsilon$ small and
positive. In that case, the solution to the system (\ref{theODEs}) will have
$\varepsilon(u)>0$ up to the blow-up time. The blow-up time, in turn, exceeds
$t$ for all points in the domain of $U_{t}.$

We may, therefore, apply the Hamilton--Jacobi formula (\ref{Sformula}), which
we write as follows. We let $\mathrm{HJ}$ denote the right-hand side of the
Hamilton--Jacobi formula (\ref{Sformula}):%
\begin{align}
\mathrm{HJ}(t,\lambda_{0},\varepsilon_{0})  &  =\log(\left\vert \lambda
_{0}-1\right\vert ^{2}+\varepsilon_{0})-\frac{\varepsilon_{0}t}{(\left\vert
\lambda_{0}-1\right\vert ^{2}+\varepsilon_{0})^{2}}\nonumber\\
&  +\log\left\vert \lambda(t;\lambda_{0},\varepsilon_{0})\right\vert
-\log\left\vert \lambda_{0}\right\vert \label{scriptS}%
\end{align}
and we then have%
\begin{equation}
S(t,\lambda(t;\lambda_{0},\varepsilon_{0}),\varepsilon(t;\lambda
_{0},\varepsilon_{0}))=\mathrm{HJ}(t,\lambda_{0},\varepsilon_{0}).
\label{Sformula2}%
\end{equation}
If $\varepsilon$ is small and positive, we therefore obtain%
\[
S(t,\lambda,\varepsilon)=\mathrm{HJ}(t,U_{t}^{-1}(\lambda,\varepsilon)),
\]
where we note that by definition $\lambda(t;U_{t}^{-1}(\lambda,\varepsilon
))=\lambda.$

Now, in the limit $\varepsilon\rightarrow0^{+}$ with $\lambda$ fixed, the
inverse function theorem tells us that $U_{t}^{-1}(\varepsilon,\lambda
)\rightarrow(0,\lambda).$ Thus, the limit (\ref{outsideLimit}) may be computed
by putting $\lambda(t;\lambda_{0},\varepsilon_{0})=\lambda$ in
(\ref{Sformula2}) and letting $\varepsilon_{0}$ tend to zero and $\lambda_{0}$
tend to $\lambda.$ This process gives (\ref{outsideLimit}).

Finally, when $\lambda_{0}=0,$ we can use continuous dependence of the
solutions on the initial conditions. The formula for $p_{\varepsilon}(t)$ in
Proposition \ref{pxSolution.prop} has a limit as $\left\vert \lambda
_{0}\right\vert $ tends to zero, so that $\delta$ tends to $+\infty.$ From
(\ref{aOverY0}), we find that $a^{2}=y_{0}^{2},$ so that from
(\ref{quadraticSolution}), $y(t)\equiv y_{0}.$ We then obtain%
\[
p_{\varepsilon}(t)=e^{-Ct}p_{0},
\]
which remains nonsingular for all $t.$ We can then continue to use the formula
(\ref{xOfT}) for $\varepsilon(t).$ Furthermore, by exponentiating
(\ref{logLambdOfT}) and letting $\left\vert \lambda_{0}\right\vert $ tend to
zero, we find that $\lambda(t)\equiv0.$ We then continue to use the remaining
formulas in the proof of Proposition \ref{existence.prop} and find that the
solution to the system exists for all time.

When $\lambda_{0}=0,$ we apply the Hamilton--Jacobi formula in the form
(\ref{HJformula}), which is to say that we replace the last two terms in
(\ref{Sformula}) by $\int_{0}^{t}\varepsilon(s)p_{\varepsilon}(s)~ds.$ We then
compute as in (\ref{dxdx0}) that the derivative of $\varepsilon
(t;0,\varepsilon_{0})$ with respect to $\varepsilon_{0}$ is positive at
$\varepsilon_{0}=0.$ Thus, by the inverse function theorem, for small positive
$\varepsilon,$ we can find a small positive $\varepsilon_{0}$ that gives
$\varepsilon(t;0,\varepsilon_{0})=\varepsilon.$ We then apply (\ref{HJformula}%
) with $\lambda_{0}=0$ and $\lambda(t)=0,$ and let $\varepsilon$ tend to zero,
which means that $\varepsilon_{0}$ also tends to zero. As $\varepsilon_{0}$
tends to zero, the function
\[
\varepsilon(s)p_{\varepsilon}(s)=\frac{\varepsilon(s)p_{\varepsilon}(s)^{2}%
}{p_{\varepsilon}(s)}=\frac{\varepsilon_{0}p_{0}^{2}e^{-Cs}}{p_{\varepsilon
}(s)}%
\]
tends to zero uniformly and we obtain (\ref{outsideLimit}).
\end{proof}

\subsection{Inside $\Sigma_{t}$\label{inside.sec}}

In this subsection, we establish the needed regularity of $S(t,\lambda
,\varepsilon)$ as $\varepsilon$ tends to zero, for $\lambda$ in $\Sigma_{t}.$
This result, whose proof is on p. \pageref{stildeProof}, together with Theorem
\ref{atTimeTstar.thm}, will allow us to understand the structure of $s_{t}$
and its derivatives on $\Sigma_{t}.$

\begin{theorem}
\label{StildeExtension.thm}Define%
\[
\tilde{S}(t,\lambda,z)=S(t,\lambda,z^{2}),\quad z>0.
\]
Fix a pair $(\sigma,\mu)$ with $\mu$ in $\Sigma_{\sigma}.$ Then $\tilde
{S}(t,\lambda,z),$ initially defined for $z>0,$ extends to a real-analytic
function in a neighborhood of $(\sigma,\mu,0)$ inside $\mathbb{R}%
\times\mathbb{C}\times\mathbb{R}.$
\end{theorem}

We emphasize that the analytically extended $\tilde{S}$ \textit{does not}
satisfy the identity $\tilde{S}(t,\lambda,z)=S(t,\lambda,z^{2}).$ Indeed,
since $\sqrt{\varepsilon(t)}p_{\varepsilon}(t)$ is always bounded away from
zero (Proposition \ref{xp2.prop}), the second Hamilton--Jacobi formula
(\ref{Sderivatives}) tells us that $\partial\tilde{S}/\partial z(t,\lambda
,z)=2\sqrt{\varepsilon}\partial S/\partial\varepsilon(t,\lambda,z^{2})$ has a
nonzero limit as $z$ tends to zero, ruling out a smooth extension that is even
in $z.$

\begin{corollary}
\label{StildeExtension.cor}Fix a pair $(\sigma,\mu)$ with $\mu$ in
$\Sigma_{\sigma}.$ Then the functions%
\begin{equation}
S(t,\lambda,\varepsilon),\quad\frac{\partial S}{\partial a}(t,\lambda
,\varepsilon),\quad\frac{\partial S}{\partial b}(t,\lambda,\varepsilon
),\quad\sqrt{\varepsilon}\frac{\partial S}{\partial\varepsilon}(t,\lambda
,\varepsilon) \label{fourFunctions}%
\end{equation}
all have extensions that are continuous in all three variables to the set of
$(t,\lambda,\varepsilon)$ with $\lambda\in\Sigma_{t}$ and $\varepsilon\geq0.$
Furthermore, for each $t>0,$ the function $s_{t}$ is infinitely differentiable
on $\Sigma_{t},$ and its derivatives with respect to $a$ and $b$ agree with
the $\varepsilon\rightarrow0^{+}$ limit of $\partial S/\partial a$ and
$\partial S/\partial b.$ If we let $t_{\ast}$ be short for $t_{\ast}%
(\lambda_{0},\varepsilon_{0}),$ then for all $\lambda_{0}$ and $\varepsilon
_{0}>0,$ we have
\begin{align*}
s_{t}(t_{\ast},\lambda(t_{\ast};\lambda_{0},\varepsilon_{0}))  &
=\log(\left\vert \lambda_{0}-1\right\vert ^{2}+\varepsilon_{0})-\frac
{\varepsilon_{0}t_{\ast}}{(\left\vert \lambda_{0}-1\right\vert ^{2}%
+\varepsilon_{0})^{2}}\\
&  +\log\left\vert \lambda(t_{\ast};\lambda_{0},\varepsilon_{0})\right\vert
-\log\left\vert \lambda_{0}\right\vert
\end{align*}
and%
\begin{align}
\frac{\partial s_{t}}{\partial a}(t_{\ast},\lambda(t_{\ast};\lambda
_{0},\varepsilon_{0}))  &  =\lim_{t\rightarrow t_{\ast}}p_{a}(t)\nonumber\\
\frac{\partial s_{t}}{\partial b}(t_{\ast},\lambda(t_{\ast};\lambda
_{0},\varepsilon_{0}))  &  =\lim_{t\rightarrow t_{\ast}}p_{b}(t).
\label{dsdaLim}%
\end{align}

\end{corollary}

\begin{proof}
We note that the four functions in (\ref{fourFunctions}) may be computed as%
\[
\tilde{S}(t,\lambda,\sqrt{\varepsilon}),\quad\frac{\partial\tilde{S}}{\partial
a}(t,\lambda,\sqrt{\varepsilon}),\quad\frac{\partial\tilde{S}}{\partial
b}(t,\lambda,\sqrt{\varepsilon}),\quad\frac{1}{2}\frac{\partial\tilde{S}%
}{\partial z}(t,\lambda,\sqrt{\varepsilon}),
\]
respectively, and that $S(t,\lambda,0)=\tilde{S}(t,\lambda,0).$ The first
claim then follows from Theorem \ref{StildeExtension.thm}. Now that the
continuity of $S$ and its derivatives has been established, we may let $t$
approach $t_{\ast}(\lambda_{0},\varepsilon_{0})$ in the Hamilton--Jacobi
formulas (\ref{Sformula}) and (\ref{Sderivatives}) to obtain the second claim.
\end{proof}

\begin{corollary}
\label{sLaplacian.cor}Let us write $\lambda\in\Sigma_{t}$ in logarithmic polar
coordinates, with $\rho=\log\left\vert \lambda\right\vert $ and $\theta
=\arg\lambda.$ Then for each pair $(t,\lambda)$ with $\lambda\in\Sigma_{t},$
we have%
\begin{equation}
\frac{\partial s_{t}}{\partial\rho}(t,\lambda)=\frac{2\rho}{t}+1.
\label{dsdRho}%
\end{equation}
Furthermore, $\partial s_{t}/\partial\theta$ is independent of $\rho$; that
is,%
\[
\frac{\partial s_{t}}{\partial\theta}=m_{t}(\theta),
\]
for some smooth function $m_{t}.$ Thus,
\begin{equation}
\frac{\partial^{2}s_{t}}{\partial\rho^{2}}+\frac{\partial^{2}s_{t}}%
{\partial\theta^{2}}=\frac{2}{t}+\frac{\partial}{\partial\theta}m_{t}(\theta)
\label{sTwoDerivs}%
\end{equation}
for some smooth function $m_{t}$, and
\begin{equation}
\left(  \frac{\partial^{2}}{\partial a^{2}}+\frac{\partial^{2}}{\partial
b^{2}}\right)  s_{t}(\lambda)=\frac{1}{\left\vert \lambda\right\vert ^{2}%
}\left(  \frac{2}{t}+\frac{\partial}{\partial\theta}m_{t}(\theta)\right)  .
\label{sLaplacian}%
\end{equation}

\end{corollary}

In Section \ref{boundary.sec}, we will obtain a formula for the function
$m_{t}(\theta)$ appearing in Corollary \ref{sLaplacian.cor}.

\begin{proof}
The derivative $\partial/\partial\rho$ may be computed in ordinary polar
coordinates as $r\partial/\partial r$ or in rectangular coordinates as
$a\partial/\partial a+b\partial/\partial b.$ It then follows from the
Hamilton--Jacobi formula (\ref{Sderivatives}) that%
\[
\left(  a\frac{\partial S}{\partial a}+b\frac{\partial S}{\partial b}\right)
(t,\lambda(t),\varepsilon(t))=a(t)p_{a}(t)+b(t)p_{b}(t).
\]
Now, for each pair $(t,\lambda)$ with $\lambda\in\Sigma_{t},$ Theorem
\ref{surjectivity.thm} tells us that we can find $(\lambda_{0},\varepsilon
_{0})$ so that
\[
\lim_{u\rightarrow t}\varepsilon(u)=0;\quad\lim_{u\rightarrow t}%
\lambda(u)=\lambda.
\]
In light of (\ref{dsdaLim}), the formula (\ref{dsdRho}) then follows from the
formula (\ref{apa}) in Theorem \ref{atTimeTstar.thm}.

Now, $\partial s_{t}/\partial\rho$ is manifestly independent of $\theta.$
Since, by Corollary \ref{StildeExtension.cor}, $s_{t}$ is an analytic, hence
$C^{2},$ function on $\Sigma_{t},$ we conclude that%
\[
\frac{\partial}{\partial\rho}\frac{\partial s_{t}}{\partial\theta}%
=\frac{\partial}{\partial\theta}\frac{\partial s_{t}}{\partial\rho}=0,
\]
showing that $\partial s_{t}/\partial\theta$ is independent of $\rho.$ The
formula (\ref{sTwoDerivs}) then follows by differentiating (\ref{dsdRho}) with
respect to $\rho.$ Finally, if we use the standard formula for the Laplacian
in polar coordinates,%
\[
\Delta=\frac{1}{r^{2}}\left(  \left(  r\frac{\partial}{\partial r}\right)
^{2}+\frac{\partial^{2}}{\partial\theta^{2}}\right)  =\frac{1}{r^{2}}\left(
\frac{\partial^{2}}{\partial\rho^{2}}+\frac{\partial^{2}}{\partial\theta^{2}%
}\right)  .
\]
we obtain (\ref{sLaplacian}) from (\ref{sTwoDerivs}).
\end{proof}

We now begin preparations for the proof of Theorem \ref{StildeExtension.thm}.
Recall from Proposition \ref{xp2.prop} that $\varepsilon(t)p_{\varepsilon
}(t)^{2}=\varepsilon_{0}p_{0}^{2}e^{-Ct},$ where $C=2\Psi-1$ is a constant
computed from $\varepsilon_{0}$ and $\lambda_{0}$ as in (\ref{CfromX0}).
Recall also from Proposition \ref{pxSolution.prop} that for $\varepsilon
_{0}>0,$ the function $1/p_{\varepsilon}(t)$ extends to real analytic function
of $t$ defined for all $t\in\mathbb{R}.$ We then define, for $\varepsilon
_{0}>0,$%
\begin{equation}
z(t;\lambda_{0},\varepsilon_{0})=\sqrt{\varepsilon_{0}}p_{0}e^{-Ct/2}\frac
{1}{p_{\varepsilon}(t;\lambda_{0},\varepsilon_{0})} \label{zFormula}%
\end{equation}
for all $t\in\mathbb{R}.$ For $t<t_{\ast}(\lambda_{0},\varepsilon_{0}),$ the
function $z(t;\lambda_{0},\varepsilon_{0})$ is positive and satisfies
\[
z(t;\lambda_{0},\varepsilon_{0})^{2}=\varepsilon(t;\lambda_{0},\varepsilon
_{0}),
\]
while for $t=t_{\ast}(\lambda_{0},\varepsilon_{0}),$ we have $z(t;\lambda
_{0},\varepsilon_{0})=0$ and for $t>t_{\ast}(\lambda_{0},\varepsilon_{0}),$
the function $z(t;\lambda_{0},\varepsilon_{0})$ is negative.

Furthermore, using (\ref{logLambdOfT}) and Point \ref{argConserved.point} of
Proposition \ref{constantsOfMotion.prop}, we see that%
\[
\lambda(t;\lambda_{0},\varepsilon_{0})=\lambda_{0}e^{\int_{0}^{t}%
\varepsilon(s)p_{\varepsilon}(s)~ds},
\]
where by Proposition \ref{xp2.prop}, we have
\[
\varepsilon(s)p_{\varepsilon}(s)=\frac{\varepsilon(s)p_{\varepsilon}(s)^{2}%
}{p_{\varepsilon}(s)}=\frac{\varepsilon_{0}p_{0}^{2}e^{-Cs}}{p_{\varepsilon
}(s)}.
\]
Since $1/p_{\varepsilon}(s)$ extends to an analytic function of $s\in
\mathbb{R}$, we see that $\lambda(t)$ extends to an analytic function of
$t\in\mathbb{R}.$ We may therefore define a map
\[
V(t,\lambda_{0},\varepsilon_{0}):=(t,\lambda(t;\lambda_{0},\varepsilon
_{0}),z(t;\lambda_{0},\varepsilon_{0})),
\]
for all $t\in\mathbb{R},$ $\lambda_{0}\in\mathbb{C},$ and $\varepsilon_{0}>0.$

\begin{proposition}
Suppose $(t,\lambda_{0},\varepsilon_{0})$ has the property that $\lambda
_{0}\neq0$ and $t_{\ast}(\lambda_{0},\varepsilon_{0})=t,$ so that
$z(t;\lambda_{0},\varepsilon_{0})=0.$ Then the Jacobian matrix of $V$ at
$(t,\lambda_{0},\varepsilon_{0})$ is invertible.
\end{proposition}

\begin{proof}
We make some convenient changes of variables. First, we replace $(t,\lambda
_{0},\varepsilon_{0})$ with $(t,\lambda_{0},\delta),$ where $\delta$ is as in
(\ref{deltaNotation}). This change has a smooth inverse, since we can recover
$\varepsilon_{0}$ from $\delta$ as%
\[
\varepsilon_{0}=\left\vert \lambda_{0}\right\vert \delta-\left\vert
\lambda_{0}\right\vert ^{2}-1.
\]
Then we write $\lambda_{0}$ in terms of its polar coordinates, $(r_{0}%
,\theta_{0}).$ Finally, we write $\lambda(t;\lambda_{0},\varepsilon_{0})$ in
logarithmic polar coordinates,%
\[
\rho(t;\lambda_{0},\varepsilon_{0}):=\log\left\vert \lambda_{0}(t;\lambda
_{0},\varepsilon_{0})\right\vert ;\quad\theta(t;\lambda_{0},\varepsilon
_{0}):=\arg(\lambda(t;\varepsilon_{0},\lambda_{0})),
\]
where by Point \ref{argConserved.point} of Proposition
\ref{constantsOfMotion.prop}, $\theta(t;\lambda_{0},\varepsilon_{0}%
)=\theta_{0}.$

Thus, to prove the proposition, it suffices to verify that the Jacobian matrix
of the map%
\[
W(t,\theta_{0},r_{0},\delta):=(t,\theta_{0},\rho(t;\lambda_{0},\varepsilon
_{0}),z(t;\lambda_{0},\varepsilon_{0}))
\]
is invertible. We observe that this Jacobian has the form
\[
W^{\prime}=\left(
\begin{array}
[c]{cc}%
I_{2\times2} & 0\\
\ast & K
\end{array}
\right)  ,
\]
where
\[
K=\left(
\begin{array}
[c]{cc}%
\frac{\partial\rho}{\partial r_{0}} & \frac{\partial\rho}{\partial\delta}\\
\frac{\partial z}{\partial r_{0}} & \frac{\partial z}{\partial\delta}%
\end{array}
\right)  .
\]
Now, by Proposition \ref{pxSolution.prop}, the lifetime is independent of
$r_{0}$ with $\delta$ and $\theta_{0}$ fixed. Thus, if we start at a point
with $t_{\ast}(\lambda_{0},\varepsilon_{0})=t$ and vary $r_{0},$ the lifetime
will remain equal to $t$ and $z(t;\varepsilon_{0},\lambda_{0})$ will remain
equal to 0. Thus, at the point in question, $\partial z/\partial r_{0}=0.$
Meanwhile, Proposition \ref{tStar.prop} gives a formula for the value of
$\rho(t;\lambda_{0},\varepsilon_{0})$ at $t=t_{\ast}(\lambda_{0}%
,\varepsilon_{0}),$ from which we can easily see that $\partial\rho/\partial
r_{0}>0.$ It therefore remains only to verify that $\partial z/\partial\delta$
is nonzero.

Now, $z(t_{\ast}(\lambda_{0},\varepsilon_{0});\lambda_{0},\varepsilon_{0})=0.$
If we differentiate this relation with respect to $\varepsilon_{0}$ with
$\lambda_{0}$ fixed, we find that%
\begin{equation}
\frac{\partial z}{\partial\varepsilon_{0}}=-\frac{\partial z}{\partial t}%
\frac{\partial t_{\ast}(\lambda_{0},\varepsilon_{0})}{\partial\varepsilon_{0}%
}. \label{dZdX0}%
\end{equation}
The derivative $\partial t_{\ast}/\partial\varepsilon_{0}$ may be computed as
\[
\frac{\partial t_{\ast}(\lambda_{0},\varepsilon_{0})}{\partial\varepsilon_{0}%
}=\frac{\partial g_{\theta_{0}}(\delta)}{\partial\delta}\frac{\partial\delta
}{\partial\varepsilon_{0}},
\]
where $g_{\theta_{0}}$ is as in (\ref{gDelta}). But the proof of Proposition
\ref{monotoneTstar.prop} shows that $\partial g_{\theta}/\partial\delta>0$ for
all $\delta>2,$ while from the formula (\ref{deltaNotation}) for $\delta,$ we
see that $\partial\delta/\partial\varepsilon_{0}>0.$ (Note also that
$\delta>2$ whenever $\varepsilon_{0}>0.$) Thus, $\partial t_{\ast}(\lambda
_{0},\varepsilon_{0})/\partial\varepsilon_{0}>0.$

Meanwhile, from (\ref{zFormula}) and (\ref{pxFormula}), we have%
\[
z(t)=\frac{\sqrt{\varepsilon_{0}}e^{Ct/2}}{\cosh(at)+\frac{2\left\vert
\lambda_{0}\right\vert -\delta}{\sqrt{\delta^{2}-4}}\sinh(at)}\left(
\cosh(at)-\frac{\delta}{\sqrt{\delta^{2}-4}}\sinh(at)\right)  .
\]
If we differentiate with respect to $t$ and evaluate at the time $t_{\ast}$
when the last factor is zero, the product rule gives%
\begin{align*}
&  \left.  \frac{\partial z(t;\lambda_{0},\varepsilon_{0})}{\partial
t}\right\vert _{t=t_{\ast}(\lambda_{0},\varepsilon_{0})}\\
&  =0+\frac{\sqrt{\varepsilon_{0}}e^{Ct/2}}{\cosh(at)+\frac{2\left\vert
\lambda_{0}\right\vert -\delta}{\sqrt{\delta^{2}-4}}\sinh(at)}a\left(
\sinh(at)-\frac{\delta}{\sqrt{\delta^{2}-4}}\cosh(at)\right)  ,
\end{align*}
which is negative because $\delta/\sqrt{\delta^{2}-4}>1$ and the denominator
is positive (Proposition \ref{pxSolution.prop}). Thus, from (\ref{dZdX0}), we
conclude that $\partial z/\partial\varepsilon_{0}>0.$
\end{proof}

We are now ready for the proof of the main result of this section.

\begin{proof}
[Proof of Theorem \ref{StildeExtension.thm}]\label{stildeProof}By
(\ref{Sformula2}), we have%
\begin{align}
\tilde{S}(t;\lambda(t;\lambda_{0},\varepsilon_{0}),z(t;\lambda_{0}%
,\varepsilon_{0}))  &  =S(t;\lambda(t;\lambda_{0},\varepsilon_{0}%
),\varepsilon(t;\lambda_{0},\varepsilon_{0}))\nonumber\\
&  =\mathrm{HJ}(t,\lambda_{0},\varepsilon_{0}), \label{StildeHJ}%
\end{align}
where $\mathrm{HJ}$ is as in (\ref{scriptS}), whenever $t_{\ast}(\lambda
_{0},\varepsilon_{0})>\sigma.$ Fix a point $(\sigma,\mu)$ with $\mu\in
\Sigma_{\sigma}.$ Then by Theorem \ref{surjectivity.thm}, we can find a pair
$(\lambda_{0},\varepsilon_{0})$ with $t_{\ast}(\lambda_{0},\varepsilon_{0}%
)=t$---so that $z(t;\lambda_{0},\varepsilon_{0})=0$---and $\lambda
(t;\lambda_{0},\varepsilon_{0})=\lambda.$ We now construct a local inverse
$V^{-1}$ to $V$ around the point $V(t,\lambda_{0},\varepsilon_{0}%
)=(t,\lambda,0).$

For any triple $(t,\lambda,z)$ in the domain of $V^{-1},$ we write
$V^{-1}(t,\,\lambda,z)$ as $(t,\lambda_{0},\varepsilon_{0}).$ We note that if
$z>0$ then $t_{\ast}(\lambda_{0},\varepsilon_{0})$ must be greater than $t,$
because if we had $t_{\ast}(\lambda_{0},\varepsilon_{0})\leq t,$ then
$z(t;\lambda_{0},\varepsilon_{0})=z$ would be zero or negative. Thus, we may
apply (\ref{StildeHJ}) at $(t,\lambda_{0},\varepsilon_{0})=V^{-1}%
(t,\lambda,z)$ to obtain%
\begin{equation}
\tilde{S}(t,\lambda,z)=\mathrm{HJ}(V^{-1}(t,\lambda,z)), \label{StildeInverse}%
\end{equation}
whenever $(t,\lambda,z)$ is in the domain of $V^{-1}$ and $z>0.$

Recall now that $\lambda(t;\lambda_{0},\varepsilon_{0})$ extends to an
analytic function of $t\in\mathbb{R}.$ Thus, the function $\mathrm{HJ}$ in
(\ref{scriptS}) extends to a smooth function of $t\in\mathbb{R},$ $\lambda
_{0}\in\mathbb{C}\setminus\{0\},$ and $\varepsilon_{0}>0,$ defined even if
$t_{\ast}(\lambda_{0},\varepsilon_{0})<t.$ Therefore, the right-hand side of
(\ref{StildeInverse}) provides the claimed smooth extension of $\tilde{S}.$
\end{proof}

\subsection{Near the boundary of $\Sigma_{t}$\label{boundary.sec}}

We start by considering what is happening right on the boundary of $\Sigma
_{t}.$

\begin{remark}
Neither the method of Section \ref{outside.sec} nor the method of Section
\ref{inside.sec} allows us to compute the value of $s_{t}(\lambda)$ for
$\lambda$ in the boundary of $\Sigma_{t}.$ Although we expect that this value
will be $\log(\left\vert \lambda-1\right\vert ^{2}),$ the question is
irrelevant to the computation of the Brown measure. After all, we are supposed
to consider $\Delta s_{t}$ \emph{computed in the distribution sense}, that is,
the distribution whose value on a test function $\psi$ is
\begin{equation}
\int_{\mathbb{C}}s_{t}(\lambda)\Delta\psi(\lambda)~d^{2}\lambda. \label{LapS}%
\end{equation}
The value of (\ref{LapS}) is unaffected by the value of $s_{t}(\lambda)$ for
$\lambda$ in $\partial\Sigma_{t},$ which is a set of measure zero in
$\mathbb{C}.$
\end{remark}

It is nevertheless essential to understand the behavior of $s_{t}(\lambda)$ as
$\lambda$ \textit{approaches} the boundary of $\Sigma_{t}.$

\begin{definition}
We say that a function $f:\mathbb{C}\rightarrow\mathbb{R}$ is \textbf{analytic
up to the boundary from inside }$\Sigma_{t}$ if the following conditions hold.
First, $f$ is real analytic on $\Sigma_{t}$. Second, for each $\lambda
\in\partial\Sigma_{t},$ we can find an open set $U$ containing $\lambda$ and a
real analytic function $g$ on $U$ such that $g$ agrees with $f$ on
$U\cap\Sigma_{t}.$ We may similarly define what it means for $f$ to be
\textbf{analytic up to the boundary from outside }$\Sigma_{t}.$
\end{definition}

\begin{proposition}
\label{smoothFromInOut.prop}For each $t>0,$ the function $s_{t}$ is analytic
up to the boundary from inside $\Sigma_{t}$ and analytic up to the boundary
from outside $\Sigma_{t}.$
\end{proposition}

Note that the proposition is not claiming that $s_{t}$ is an analytic function
on all of $\mathbb{C}.$ Indeed, our main results tell us that $\frac{1}{4\pi
}\Delta s_{t}(\lambda)$ is identically zero for $\lambda$ outside $\Sigma_{t}$
but approaches a typically nonzero value as $\lambda$ approaches a boundary
point from the inside. As we approach from the inside a boundary point with
polar coordinates $(r,\theta),$ the limiting value of $\frac{1}{4\pi}\Delta
s_{t}(\lambda)$ is $w_{t}(\theta)/r^{2}.$ This quantity certainly cannot
always be zero, or the Brown measure of $b_{t}$ would be identically zero.
Actually, we will see in Section \ref{OmegaFormula.sec} that $w_{t}(\theta)$
is strictly positive except when $t=4$ and $\theta=\,\pi.$

\begin{proof}
We have shown that $s_{t}(\lambda)=\log(\left\vert \lambda-1\right\vert ^{2})$
for $\lambda$ in $(\overline{\Sigma}_{t})^{c}.$ Since $1\in\Sigma_{t},$ we see
that $s_{t}$ is analytic from the outside of $\Sigma_{t}.$

To address the analyticity from the inside, first note that by applying
(\ref{StildeInverse}) with $z=0,$ we have%
\[
s_{t}(\lambda)=S(t,\lambda,0)=\tilde{S}(t,\lambda,0)=\mathrm{HJ}%
(V^{-1}(t,\lambda,0)),
\]
where $\mathrm{HJ}$ is as in (\ref{scriptS}). But if $\varepsilon_{0}%
^{t}:\Sigma_{t}\rightarrow\mathbb{R}$ and $\lambda_{t}:\Sigma_{t}%
\rightarrow\mathbb{C}$ are as in Lemma \ref{lambda_t.lem}, then we can see
that
\[
V^{-1}(t,\lambda,0)=(t,\lambda_{t}^{-1}(\lambda),\varepsilon_{0}^{t}%
(\lambda_{t}^{-1}(\lambda)))
\]
and we conclude that%
\begin{equation}
s_{t}(\lambda)=\mathrm{HJ}(t,\lambda_{t}^{-1}(\lambda),\varepsilon_{0}%
^{t}(\lambda_{t}^{-1}(\lambda))). \label{stFromInside}%
\end{equation}

We now claim that the function $\varepsilon_{0}^{t}(\lambda_{0}),$ initially
defined for $\lambda_{0}\in\overline{\Sigma}_{t},$ extends to an analytic
function in a neighborhood of $\overline{\Sigma}_{t}.$ For $t\geq4,$ we can
simply use the formula (\ref{X0OfLambda0}) for all nonzero $\lambda_{0}.$ For
$t<4,$ however, the formula (\ref{X0OfLambda0}) becomes undefined in a
neighborhood of a point where $\partial\Sigma_{t}$ intersects the unit circle.

Nevertheless, we can make a general argument as follows. To compute
$\varepsilon_{0}^{t}(\lambda_{0}),$ we solve the equation $t_{\ast}%
(\lambda_{0},\varepsilon_{0})=t$ for $\varepsilon_{0}$ as a function of
$\lambda_{0}.$ To do this, we first solve the equation $g_{\theta_{0}}%
(\delta)=t$ for $\delta_{\theta_{0},t}$ and then solve for $\varepsilon_{0}$
in terms of $\delta$ as $\varepsilon_{0}=\left\vert \lambda_{0}\right\vert
\delta-\left\vert \lambda_{0}\right\vert ^{2}-1.$ Now, we know from the proof
of Proposition \ref{monotoneTstar.prop} that $g_{\theta_{0}}(\delta)=t$ has a
solution when $\left\vert \theta_{0}\right\vert \leq\theta_{\max}(t)=\cos
^{-1}(1-t/2),$ with the solution being $\delta=2$ when $\theta_{0}=\pm
\theta_{\max}(t).$ We can also verify that $\partial g_{\theta_{0}}%
/\partial\delta>0$ for all $\delta\geq2.$ This was verified for $\delta>2$ in
the proof of Proposition \ref{monotoneTstar.prop}. To see that the result
holds even when $\delta=2,$ it suffices to verify that the expressions in
(\ref{Derivative1}) and (\ref{Derivative2}) have positive limits as
$\delta\rightarrow2^{+}$. We omit this verification and simply note that the
limits have the values 1 and $1/3,$ respectively. It then follows from the
implicit function theorem that (1) the solution $\delta_{\theta_{0},t}$
continues to exist (with $\delta<2$) for $\left\vert \theta_{0}\right\vert $
slightly larger than $\theta_{\max}(t),$ and (2) the solution $\delta
_{\theta_{0},t}$ depends analytically on $\theta_{0}.$ Then, the expression
\[
\varepsilon_{0}^{t}(\lambda_{0})=\left\vert \lambda_{0}\right\vert
\delta_{\theta_{0},t}-\left\vert \lambda_{0}\right\vert ^{2}-1
\]
makes sense and is analytic for all nonzero $\lambda_{0}$ with $\left\vert
\arg\lambda_{0}\right\vert <\theta_{\max}(t)+\alpha_{t},$ for some positive
quantity $\alpha_{t}.$ We note that in this expression, $\varepsilon_{0}%
^{t}(\lambda_{0})$ can be negative---for example if $\left\vert \lambda
_{0}\right\vert =1$ and $\arg\lambda_{0}>\theta_{\max}(t).$

We now consider the function $\lambda_{t},$ defined as%
\[
\lambda_{t}(\lambda_{0})=\lambda(t;\lambda_{0},\varepsilon_{0}^{t}(\lambda
_{0})),
\]
and we recall that $\lambda_{t}(\lambda_{0})=\lambda_{0}$ for $\lambda_{0}%
\in\partial\Sigma_{t}.$ Although $\lambda_{t}$ was initially defined for
$\lambda_{0}$ in $\overline{\Sigma}_{t},$ it has an analytic extension to a
neighborhood of $\overline{\Sigma}_{t}$, namely the set of $\lambda_{0}$ in
the domain of the extended function $\varepsilon_{0}^{t}$ for which the pair
$(\lambda_{0},\varepsilon_{0})$ satisfy the assumptions in
(\ref{standingAssumptions}). We now claim that the derivative of $\lambda
_{t}(\lambda_{0})$ is invertible at each point in its domain. We use polar
coordinates in both domain and range. Since $\arg(\lambda_{t}(\lambda
_{0}))=\arg\lambda_{0},$ the derivative will have the form%
\[
\lambda_{t}^{\prime}(\lambda_{0})=\left(
\begin{array}
[c]{cc}%
\frac{\partial\left\vert \lambda_{t}\right\vert }{\partial r} & \frac
{\partial\arg\lambda_{t}}{\partial\theta}\\
0 & 1
\end{array}
\right)  ,
\]
and it therefore suffices to check that $\partial\left\vert \lambda
_{t}\right\vert /\partial r$ is nonzero. To see this, we use the formula
(\ref{lambdaTformula2}), where $\delta=\delta_{\theta_{0},t}$ as in the
previous paragraph. Since $\delta$ is independent of $\left\vert \lambda
_{0}\right\vert $ with $t$ and $\arg\lambda_{0}$ fixed, we can easily verify
from (\ref{lambdaTformula2}) that $\partial\left\vert \lambda_{t}\right\vert
/\partial r>0.$

Now, we have already established that $s_{t}$ is analytic in the interior of
$\Sigma_{t}.$ Consider, then, a point $\lambda$ in $\partial\Sigma_{t},$ so
that $\lambda_{t}(\lambda)=\lambda.$ Since $\lambda_{t}^{\prime}(\lambda)$ is
invertible, it has a analytic local inverse $\lambda_{t}^{-1}$ defined near
$\lambda.$ Then the formula (\ref{stFromInside}) gives an analytic extension
of $s_{t}$ to a neighborhood of $\lambda.$
\end{proof}

\begin{proposition}
\label{InEqualsOut.prop}Fix a point $\mu$ on the boundary of $\Sigma_{t}.$
Then the functions
\[
s_{t}(\lambda),\quad\frac{\partial s_{t}}{\partial a}(\lambda),\quad
\frac{\partial s_{t}}{\partial b}(\lambda)
\]
all approach the same value when $\lambda$ approaches $\mu$ from inside
$\Sigma_{t}$ as when $\lambda$ approaches $\mu$ from outside $\overline
{\Sigma}_{t}.$
\end{proposition}

\begin{proof}
We begin by considering $s_{t}$ itself. The limit as $\lambda$ approaches
$\mu$ from the inside may be computed by using (\ref{stFromInside}). By Lemma
\ref{lambda_t.lem}, as $\lambda$ approaches $\mu$ from the inside,
$\lambda_{t}^{-1}(\lambda)$ approaches $\lambda_{t}^{-1}(\mu)=\mu,$ and
$\varepsilon_{0}^{t}(\lambda_{t}^{-1}(\lambda))$ approaches 0. Thus, the
limiting value of $s_{t}$ from the inside is%
\[
\mathrm{HJ}(t,\mu,0)=\log(\left\vert \mu-1\right\vert ^{2}),
\]
where $\mathrm{HJ}$ is given by (\ref{scriptS}) and were we have used that
$\lambda(t;\mu,0)=\mu.$ (See the last part of Proposition \ref{existence.prop}%
.) Since $s_{t}(\lambda)=\log(\left\vert \lambda-1\right\vert ^{2})$ outside
$\overline{\Sigma}_{t},$ the limit of $s_{t}$ from the outside agrees with the
limit from the inside.

Next we consider the derivatives, which we compute in logarithmic polar
coordinates $\rho=\log\left\vert \lambda\right\vert $ and $\theta=\arg\lambda
$. By (\ref{dsdRho}), we have%
\[
\frac{\partial s_{t}}{\partial\rho}(\lambda)=\left(  a\frac{\partial s_{t}%
}{\partial a}+b\frac{\partial s_{t}}{\partial b}\right)  (\lambda)=\frac
{\log(\left\vert \lambda\right\vert ^{2})}{t}+1
\]
for $\lambda\in\Sigma_{t}.$ Letting $\lambda$ approach $\mu$ from the inside
gives the value $\log(\left\vert \mu\right\vert ^{2})/t+1.$ Since $\mu$ is on
the boundary of $\Sigma_{t},$ Theorem \ref{domainGobbles.thm} says that
$T(\mu)=t,$ so that%
\begin{align*}
\frac{\log\left\vert \mu\right\vert ^{2}}{t}+1  &  =\frac{\left\vert
\mu\right\vert ^{2}-1}{\left\vert \mu-1\right\vert ^{2}}+1\\
&  =\frac{2(\left\vert \mu\right\vert ^{2}-\operatorname{Re}\mu)}{\left\vert
\mu-1\right\vert ^{2}}.
\end{align*}
Taking the corresponding derivative of the \textquotedblleft
outside\textquotedblright\ function \thinspace$\log(\left\vert \lambda
-1\right\vert ^{2})$ and letting $\lambda$ tend to $\mu$ from the outside
gives the same result.

Finally, we recall from Proposition \ref{constantsOfMotion.prop} that
$ap_{b}-bp_{a}$ is a constant of motion. Thus, by the second Hamilton--Jacobi
formula (\ref{Sderivatives}) and the initial conditions
(\ref{initialConditions2}), we have%
\begin{align*}
a\frac{\partial s_{t}}{\partial b}(u,\lambda(u),\varepsilon(u))-b\frac
{\partial s_{t}}{\partial a}(u,\lambda(u),\varepsilon(u))  &  =a_{0}%
p_{b,0}-b_{0}p_{a,0}\\
&  =(2a_{0}b_{0}-2b_{0}(a_{0}-1))p_{0}\\
&  =\frac{2b_{0}}{\left\vert \lambda_{0}-1\right\vert ^{2}+\varepsilon_{0}}.
\end{align*}
If we choose $\varepsilon_{0}$ and $\lambda_{0}$ so that $t_{\ast}(\lambda
_{0},\varepsilon_{0})=t$ we can use the regularity result in Corollary
\ref{StildeExtension.cor} to let $u$ tend to $t.$ This gives%
\[
a\frac{\partial s_{t}}{\partial b}(\lambda)-b\frac{\partial s_{t}}{\partial
a}(\lambda)=\frac{2b_{0}}{\left\vert \lambda_{0}-1\right\vert ^{2}%
+\varepsilon_{0}},
\]
where now $\lambda_{0}=\lambda_{t}^{-1}(\lambda)$ and $\varepsilon
_{0}=\varepsilon_{0}^{t}(\lambda_{t}^{-1}(\lambda)).$ As $\lambda$ approaches
$\mu$, Theorem \ref{surjectivity.thm} says that the value of $\lambda_{0}$
approaches $\mu$ and $\varepsilon_{0}$ approaches 0, so we get%
\[
\lim_{\lambda\rightarrow\mu^{\mathrm{inside}}}\left(  a\frac{\partial s_{t}%
}{\partial b}(\lambda)-b\frac{\partial s_{t}}{\partial a}(\lambda)\right)
=\frac{2\operatorname{Im}\mu}{\left\vert \mu-1\right\vert ^{2}}.
\]
Taking the corresponding derivative of \thinspace$\log(\left\vert
\lambda-1\right\vert ^{2})$ and letting $\lambda$ tend to $\mu$ from the
outside gives the same result.
\end{proof}

\subsection{Proof of the main result}

In this subsection, we prove our first main result, Theorem \ref{main.thm}.
Proposition \ref{omega.prop} will then be proved in Section
\ref{OmegaFormula.sec}, while Propositions \ref{connectToBiane.prop} and
\ref{connectToBiane2.prop} will be proved in Section \ref{bianeConnection.sec}.

\begin{proposition}
\label{dsdTheta.prop}For each fixed $t,$ the restriction to $\Sigma_{t}$ of
the function
\[
\frac{\partial s_{t}}{\partial\theta}(t,\lambda)
\]
is the unique function that on $\Sigma_{t}$ that (1) extends continuously to
the boundary, (2) agrees with the $\theta$-derivative of $\log(\left\vert
\lambda-1\right\vert ^{2})$ on the boundary, and (3) is independent of
$r=\left\vert \lambda\right\vert .$ Thus, the function $m_{t}$ in Corollary
\ref{sLaplacian.cor} is given by
\begin{equation}
m_{t}(\theta)=\frac{2r_{t}(\theta)\sin\theta}{r_{t}(\theta)^{2}+1-2r_{t}%
(\theta)\cos\theta}, \label{dsdthetaProved}%
\end{equation}
where $r_{t}(\theta)$ is the outer radius of the domain $\Sigma_{t}$ (Figure
\ref{r1r2.fig}).
\end{proposition}

\begin{proof}
We have already established in Corollary \ref{sLaplacian.cor} that $\partial
s_{t}/\partial\theta$ is independent of $\rho$ (or equivalently, of $r$) in
$\Sigma_{t}.$ Then Propositions \ref{smoothFromInOut.prop} and
\ref{InEqualsOut.prop} tell us that $\partial s_{t}/\partial\theta$ is
continuous up to the boundary and agrees there with the angular derivative of
$\log(\left\vert \lambda-1\right\vert ^{2}).$ Thus, to compute $\partial
s_{t}/\partial\theta$ at a point in $\Sigma_{t},$ we travel along a radial
segment (in either direction) until we hit the boundary at radius
$r_{t}(\theta)$ or $1/r_{t}(\theta)$. We then evaluate the angular derivative
of $\log(\left\vert \lambda-1\right\vert ^{2}),$ as in (\ref{angularDeriv}),
giving the claimed expression for $\partial s_{t}/\partial\theta=m_{t}%
(\theta).$
\end{proof}

\begin{proposition}
\label{distributionalLap.prop}For each $t>0,$ the distributional Laplacian of
$s_{t}(\lambda)$ with respect to $\lambda$ may be computed as follows. Take
the pointwise Laplacian of $s_{t}$ outside $\overline{\Sigma}_{t}$ (giving
zero), take the pointwise Laplacian of $s_{t}$ inside $\Sigma_{t}$ (giving the
expression (\ref{sLaplacian}) in Corollary \ref{sLaplacian.cor}) and ignore
the boundary of $\Sigma_{t}.$
\end{proposition}

\begin{proof}
Since, by Proposition \ref{smoothFromInOut.prop}, $s_{t}$ is analytic up to
the boundary of $\Sigma_{t}$ from the inside, Green's second identity says
that%
\begin{align*}
\int_{\Sigma_{t}}s_{t}(\lambda)\Delta\psi(\lambda)~d^{2}\lambda &
=\int_{\Sigma_{t}}(\Delta s_{t}(\lambda))\psi(\lambda)~d^{2}\lambda\\
&  +\int_{\partial\Sigma_{t}}(s_{t}(\lambda)\nabla\psi(\lambda)-\psi
(\lambda)\nabla s_{t}(\lambda))\cdot\hat{n}~dS,
\end{align*}
for any test function $\psi,$ where in the last integral, the limiting value
of $\nabla s_{t}$ from the inside should be used. This identity holds because
the boundary of $\Sigma_{t}$ is smooth for $t\neq4$ and piecewise smooth when
$t=4$ (Point \ref{domainSmooth.point} of Theorem \ref{regionProperties.thm}).
We also have similar formula for the integral over the complement of
$\overline{\Sigma}_{t},$ provided that $\psi$ is compactly supported, but with
the direction of the unit normal reversed. Proposition \ref{InEqualsOut.prop}
then tells us that the boundary terms in the two integrals cancel, giving
\begin{equation}
\int_{\mathbb{C}}s_{t}(\lambda)\Delta\chi(\lambda)~d^{2}\lambda=\int%
_{(\overline{\Sigma}_{t})^{c}}(\Delta s_{t}(\lambda))\chi(\lambda
)~d^{2}\lambda+\int_{\Sigma_{t}}(\Delta s_{t}(\lambda))\chi(\lambda
)~d^{2}\lambda, \label{distributionalLap}%
\end{equation}
where the integral over $(\overline{\Sigma}_{t})^{c}$ is actually zero, since
$\Delta s_{t}(\lambda)=0$ there. The formula (\ref{distributionalLap}) says
that the distributional Laplacian of $s_{t}$ may be computed by taking the
ordinary, pointwise Laplacian in $\overline{\Sigma}_{t}$ and in $\Sigma_{t}$
and ignoring the boundary of $\Sigma_{t}.$
\end{proof}

We now have all the ingredients for a proof of Theorem \ref{main.thm}.

\begin{proof}
[Proof of Theorem \ref{main.thm}]Proposition \ref{distributionalLap.prop}
tells us that we can compute the distributional Laplacian of $s_{t}$
separately inside $\Sigma_{t}$ and outside $\overline{\Sigma}_{t},$ ignoring
the boundary. Theorem \ref{outside.thm} tells us that the Laplacian outside
$\overline{\Sigma}_{t}$ is zero. Corollary \ref{sLaplacian.cor} gives us the
form of $\Delta s_{t}$ inside $\Sigma_{t},$ while Proposition \ref{dsdtheta}
identifies the function $m_{t}$ appearing in Corollary \ref{sLaplacian.cor}.
The claimed formula for the Brown measure therefore holds.
\end{proof}

\section{Further properties of the Brown measure\label{omega.sec}}

\subsection{The formula for $\omega$\label{OmegaFormula.sec}}

In this subsection, we derive the formula for $w_{t}$ given in Proposition
\ref{omega.prop} in terms of the density $\omega.$ Throughout, we will write
the function $T$ in (\ref{Tlambda}) in polar coordinates as%
\begin{equation}
T(r,\theta)=(r^{2}+1-2r\cos\theta)\frac{\log(r^{2})}{r^{2}-1}. \label{Tpolar2}%
\end{equation}
We start with a simple rewriting of the expression for $w_{t}$ in Theorem
\ref{main.thm}.

\begin{lemma}
The density $w_{t}(\theta)$ in Theorem \ref{main.thm} may also be written as%
\begin{equation}
w_{t}(\theta)=\frac{1}{2\pi t}\left(  1+\frac{\partial}{\partial\theta
}[h(r_{t}(\theta))\sin\theta]\right)  , \label{wtWithHr}%
\end{equation}
where%
\[
h(r)=r\frac{\log(r^{2})}{r^{2}-1}.
\]

\end{lemma}

\begin{proof}
We start by noting that the point with polar coordinates $(r_{t}%
(\theta),\theta)$ is on the boundary of $\Sigma_{t}.$ Thus, by Theorem
\ref{domainGobbles.thm}, we have $T(r_{t}(\theta),\theta)=t,$ from which we
obtain%
\[
\frac{1}{r_{t}(\theta)^{2}+1-2r_{t}(\theta)\cos\theta}=\frac{1}{t}\frac
{\log(r_{t}(\theta))}{r_{t}(\theta)^{2}-1}.
\]
Thus, we may write%
\[
\frac{2r_{t}(\theta)\sin\theta}{r_{t}(\theta)^{2}+1-2r_{t}(\theta)\cos\theta
}=\frac{2}{t}h(r_{t}(\theta))\sin\theta,
\]
from which the claimed formula follows easily from the expression in Theorem
\ref{main.thm}.
\end{proof}

We now formulate the main result of this subsection, whose proof is on p.
\pageref{densityWithoutProof}.

\begin{theorem}
\label{densityWithoutDeriv.thm}Consider the function $\omega(r,\theta)$
defined in (\ref{omegaFormula}). Although the right-hand side of
(\ref{omegaFormula}) is indeterminate at $r=1,$ the function $\omega$ has a
smooth extension to all $r>0$ and all $\theta.$ The function $w_{t}(\theta)$
in Theorem \ref{main.thm} can then be expressed as%
\[
w_{t}(\theta)=\frac{1}{2\pi t}\omega(r_{t}(\theta),\theta).
\]

The function $\omega$ has the following properties.

\begin{enumerate}
\item \label{omegaInversion.point}We have $\omega(1/r,\theta)=\omega
(r,\theta)$ for all $r>0$ and all $\theta.$

\item \label{omegaOnCircle.point}When $r=1,$ we have%
\[
\omega(1,\theta)=3\frac{1+\cos\theta}{2+\cos\theta}.
\]
In particular, $\omega(1,0)=2$ and $\omega(1,\pi)=0.$

\item \label{omegaPos.point}The density $\omega(r,\theta)$ is strictly
positive except when $r=1$ and $\theta=\pm\pi.$ Furthermore, $\omega
(r,\theta)\leq2$ with equality precisely when $r=1$ and $\theta=0.$

\item \label{omegaAtZero.point}We have%
\[
\lim_{r\rightarrow0}\omega(r,\theta)=1,
\]
where the limit is uniform in $\theta.$
\end{enumerate}
\end{theorem}

We now derive consequences for $w_{t}.$ For $t\leq4,$ the density
$w_{t}(\theta)$ is only defined for $-\theta_{\max}(t)<\theta<\theta_{\max
}(t),$ where $\theta_{\max}(t)=\cos^{-1}(1-t/2),$ while for $t>4,$ the density
$w_{t}(\theta)$ is defined for all $\theta.$ (Recall Theorem
\ref{regionProperties.thm}.)

\begin{corollary}
[Positivity]If $t>4,$ then $w_{t}(\theta)$ is strictly positive for all
$\theta.$ If $t<4,$ then $w_{t}(\theta)$ is strictly positive for
$-\theta_{\max}(t)<\theta<\theta_{\max}(t)$ and the limit as $\theta$
approaches $\pm\theta_{\max}(t)$ of $w_{t}(\theta)$ is strictly positive.
Finally, if $t=4,$ then $w_{t}(\theta)$ is strictly positive for $-\pi
<\theta<\pi,$ but $\lim_{\theta\rightarrow\pm\pi}w_{t}(\theta)=0.$
\end{corollary}

\begin{proof}
The only time $\omega(r,\theta)$ equals zero is when $r=1$ and $\theta=\pm
\pi.$ When~$t>4,$ the function $r_{t}(\theta)$ is continuous and and greater
than $1$ for all $\theta,$ so that $w_{t}(\theta)$ is strictly positive in
this case. When $t\leq4,$ we know from Proposition \ref{setFt.prop} that
$r_{t}(\theta)$ is greater than 1 for $\left\vert \theta\right\vert
<\theta_{\max}(t)$ and approaches 1 when $\theta$ approaches $\pm\theta_{\max
}(t).$ Thus, $w_{t}(\theta)=\omega(r_{t}(\theta),\theta)$ is strictly positive
for $\left\vert \theta\right\vert <\theta_{\max}(t).$ When $t<4,$ we have
$\theta_{\max}(t)=\cos^{-1}(1-t/2)<\pi$ and the limiting value of
$w_{t}(\theta)$---namely $\omega(1,\theta_{\max})$---will be positive.
Finally, when $t=4,$ we have $\theta_{\max}(t)=\pi$ and the limiting value of
$w_{t}(\theta)$ is $\omega(1,\pi)=0.$
\end{proof}

%

\begin{figure}[ptb]%
\centering
\includegraphics[
height=1.6034in,
width=4.8282in
]%
{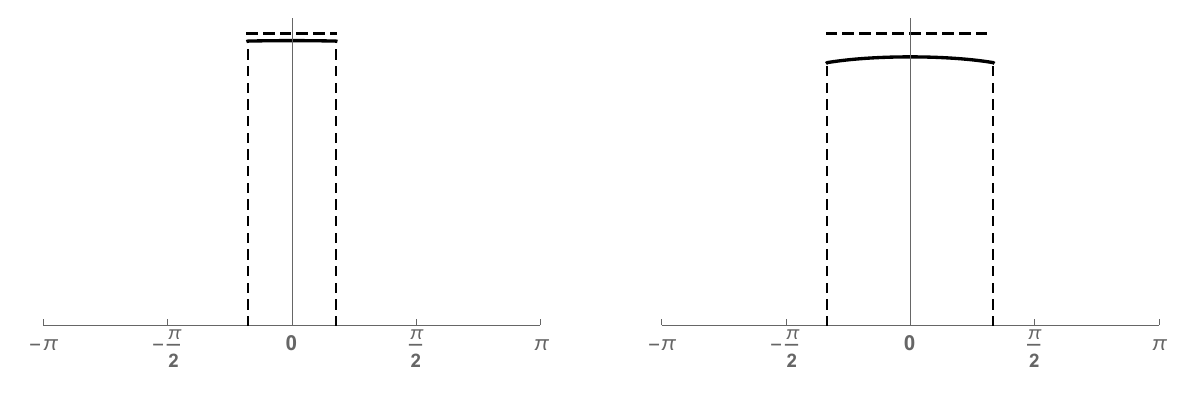}%
\caption{Plots of $w_{t}(\theta)$ (black) and $1/(\pi t)$ (dashed) for $t=0.3$
and 1}%
\label{asymptotics1.fig}%
\end{figure}

\begin{corollary}
[Asymptotics]\label{asymptotics.cor}The density $w_{t}(\theta)$ has the
property that%
\[
w_{t}(\theta)\sim\frac{1}{\pi t}%
\]
for small $t$. More precisely, for all sufficiently small $t$ and all
$\theta\in(-\theta_{\max}(t),\theta_{\max}(t)),$ the quantity $\pi
tw_{t}(\theta)$ is close to 1. Furthermore,%
\[
w_{t}(\theta)\sim\frac{1}{2\pi t}%
\]
for large $t$. More precisely, for all sufficiently large $t$ and all
$\theta,$ the quantity $2\pi tw_{t}(\theta)$ is close to 1.
\end{corollary}

See Figures \ref{asymptotics1.fig} and \ref{asymptotics2.fig}. The small- and
large-$t$ behavior of the region $\Sigma_{t}$ can also be determined using the
behavior of the function $T(\lambda)$ near $\lambda=1$ (small $t$) and near
$\lambda=0$ (large $t$), together with the invariance of the region under
$\lambda\mapsto1/\lambda.$ For small $t,$ the region resembles a disk of
radius $\sqrt{t}$ around 1, while for large $t,$ the region resembles an
annulus with inner radius $e^{-t/2}$ and outer radius $e^{t/2}.$ In
particular, the expected behavior of the Brown measure for small $t$ can be
observed: it resembles the uniform probability measure on a disk of radius
$\sqrt{t}$ centered at 1.%

\begin{figure}[ptb]%
\centering
\includegraphics[
height=1.6034in,
width=4.8282in
]%
{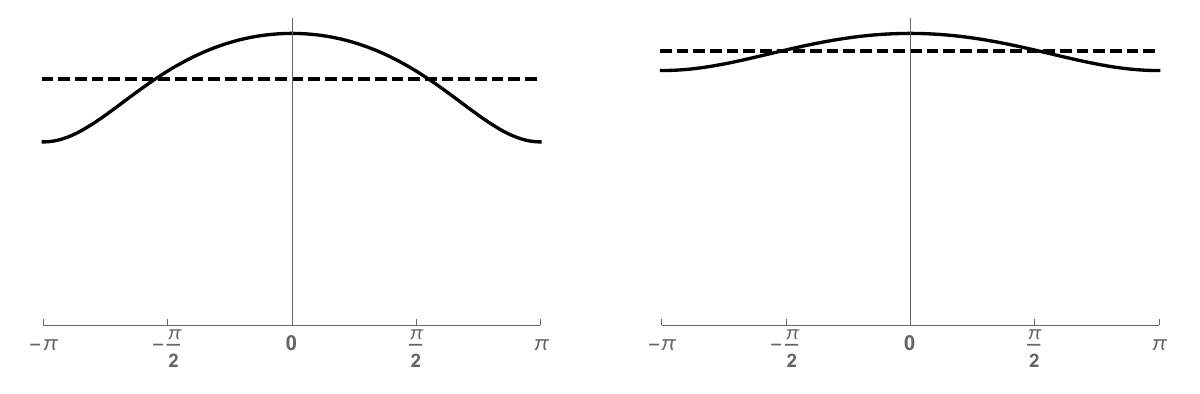}%
\caption{Plots of $w_{t}(\theta)$ (black) and $1/(2\pi t)$ (dashed) for $t=7$
and 10}%
\label{asymptotics2.fig}%
\end{figure}

\begin{proof}
When $t$ is small, the entire boundary of $\Sigma_{t}$ will be close to
$\lambda=1,$ since this is the only point where $T(\lambda)=0.$ Furthermore,
when $t$ is small, $\theta_{\max}(t)=\cos^{-1}(1-t/2)$ is close to zero. When
$t$ is small, therefore, the quantity
\[
\pi tw_{t}(\theta)=\frac{1}{2}\omega(r_{t}(\theta),\theta)
\]
will be close to $\omega(1,0)/2=1$ for all $\theta\in(-\theta_{\max}%
(t),\theta_{\max}(t)),$ by Point \ref{omegaOnCircle.point} of Theorem
\ref{densityWithoutDeriv.thm}.

When $t$ is large (in particular, greater than 4), the inner boundary of the
domain will be close to $\lambda=0,$ since this is the only point in the unit
disk where $T(\lambda)$ is large. Thus, for large $t,$ the inner radius
$1/r_{t}(\theta)$ of the domain will be uniformly small, and therefore%
\[
2\pi tw_{t}(\theta)=\omega(r_{t}(\theta),\theta)=\omega(1/r_{t}(\theta
),\theta)
\]
will be uniformly close to 1, by Point \ref{omegaAtZero.point} of Theorem
\ref{densityWithoutDeriv.thm}.
\end{proof}

\begin{proof}
[Proof of Theorem \ref{densityWithoutDeriv.thm}]\label{densityWithoutProof}We
note that the function $T$ in (\ref{Tpolar2}) can be written as%
\[
T(r,\theta)=\left(  r+\frac{1}{r}-2\cos\theta\right)  h(r),
\]
so that%
\begin{align*}
\frac{\partial T}{\partial r}  &  =\left(  1-\frac{1}{r^{2}}\right)
h(r)+\left(  r+\frac{1}{r}-2\cos\theta\right)  h^{\prime}(r);\\
\frac{\partial T}{\partial\theta}  &  =2\sin\theta~h(r).
\end{align*}
Applying implicit differentiation to the identity $T(r_{t}(\theta),\theta)=t$
then gives%
\begin{equation}
\frac{dr_{t}(\theta)}{d\theta}=-\frac{\partial T/\partial\theta}{\partial
T/\partial r}. \label{implicit}%
\end{equation}

By the chain rule and (\ref{implicit}), $\frac{d}{d\theta}[h(r_{t}%
(\theta))\sin\theta]=q(r_{t}(\theta),\theta),$ where
\begin{align}
q(r,\theta)  &  =h(r)\cos\theta-h^{\prime}(r)\sin\theta\frac{\partial
T/\partial\theta}{\partial T/\partial r}\nonumber\\
&  =h(r)\cos\theta-~\frac{2h^{\prime}(r)\sin^{2}\theta~h(r)}{\left(
1-\frac{1}{r^{2}}\right)  h(r)+\left(  r+\frac{1}{r}-2\cos\theta\right)
h^{\prime}(r)}. \label{uForm}%
\end{align}
After computing that%
\[
h^{\prime}(r)=\frac{2}{r^{2}-1}-\frac{r^{2}+1}{r(r^{2}-1)}h(r),
\]
it is a straightforward but tedious exercise to simplify (\ref{uForm}) and
obtain the claimed formula (\ref{omegaFormula}).

Since $h(1/r)=h(r),$ we may readily verify Point (\ref{omegaInversion.point});
both numerator and denominator in the fraction on the right-hand side of
(\ref{omegaFormula}) change by a factor of $1/r^{2}$ when $r$ is replaced by
$1/r$.

To understand the behavior of $\omega$ at $r=1,$ we need to understand the
function $h$ better. We may easily calculate that $h$ has a removable
singularity at $r=1$ with $h(1)=1$, $h^{\prime}(1)=0,$ and $h^{\prime\prime
}(1)=-1/3.$ We also claim that $h$ satisfies $0<h(r)\leq1,$ with $h(r)=1$ only
at $r=1$. To verify the claim, we first compute that $\lim_{r\rightarrow
0}h(r)=0$ and that%
\[
h^{\prime}(r)=\frac{2(r^{2}-1)+(r^{2}+1)\log(1/r^{2})}{(r^{2}-1)^{2}}.
\]
Using the Taylor expansion of logarithm, we may then compute that
\[
h^{\prime}(r)=\frac{1}{(r^{2}-1)^{2}}\sum_{k=3}^{\infty}\left(  \frac{2}%
{k}-\frac{1}{k+1}\right)  (1-r^{2})^{k}>0
\]
for $0<r<1.$ Thus, $h(r)$ increases from $0$ to $1$ on $[0,1].$

We now write $h$ in the form%
\begin{equation}
h(r)=1-c(r)(r-1)^{2} \label{hAndC}%
\end{equation}
for some analytic function $c(r),$ with $c(1)=1/6.$ The minus sign in
(\ref{hAndC}) is convenient because $h$ has a strict global maximum at $1,$
which means $c(r)$ is strictly positive everywhere.

Now, since $h(1)=1,$ the fraction on the right-hand side of
(\ref{omegaFormula}) is of $0/0$ form when $r=1.$ To rectify this situation,
we observe that $\alpha$ and $\beta$ may be written as%
\[
\alpha(r)=(r-1)^{2}[1+2rc(r)];\quad\beta(r)=(r-1)^{2}[1-(r^{2}+1)c(r)].
\]
Thus, we can take a factor of $(r-1)^{2}$ out of numerator and denominator to
obtain
\begin{equation}
\omega(r,\theta)=1+h(r)\frac{\tilde{\alpha}(r)\cos\theta+\tilde{\beta}%
(r)}{\tilde{\beta}(r)\cos\theta+\tilde{\alpha}(r)}, \label{omegaFormula2}%
\end{equation}
where $\tilde{\alpha}(r)=1+2rc(r)$ and $\tilde{\beta}(r)=1-(r^{2}+1)c(r).$
This expression is no longer of $0/0$ form at $r=1.$ Indeed, since $h(1)=1$
and $c(1)=1/6,$ we may easily verify the claimed formula for $\omega
(1,\theta)$ in Point \ref{omegaOnCircle.point} of the theorem. We will shortly
verify that the denominator in the fraction on the right-hand side of
(\ref{omegaFormula2}) is positive for all $r>0$ and all $\theta,$ from which
the claimed smooth extension of $\omega$ follows.

To verify the claimed positivity of $\omega,$ we first observe that
$\tilde{\beta}(r)z+\tilde{\alpha}(r)$ is positive when $z=1$ (with a value of
$2-(r-1)^{2}c(r)=1+h(r)$) and also positive when $z=-1$ (with a value of
$(r+1)^{2}c(r)$), and hence positive for all $-1\leq z\leq1.$ Thus, the
denominator in the fraction on the right-hand side of (\ref{omegaFormula2}) is
never zero. We then compute that%
\[
\frac{d}{dz}\frac{\tilde{\alpha}(r)z+\tilde{\beta}(r)}{\tilde{\beta
}(r)z+\tilde{\alpha}(r)}=\frac{\tilde{\alpha}(r)^{2}-\tilde{\beta}(r)^{2}%
}{(\tilde{\beta}(r)z+\tilde{\alpha}(r))^{2}}=\frac{(r+1)^{2}h(r)}%
{(\tilde{\beta}(r)z+\tilde{\alpha}(r))^{2}}>0
\]
for all $r$ and $\theta.$ Thus, $(\tilde{\alpha}(r)z+\tilde{\beta}%
(r))/(\tilde{\beta}(r)z+\tilde{\alpha}(r))$ increases from $-1$ to $1$ as $z$
increases from~$-1$ to $1.$ Since $h(r)$ is positive, we conclude that%
\[
1-h(r)\leq\omega(r,\theta)\leq1+h(r)
\]
for all $r$ and $\theta,$ with equality when $\cos\theta=-1$ in the first case
and when $\cos\theta=1$ in the second case. Since $h(r)\leq1$ with equality
only at $r=1,$ we see that $\omega(r,\theta)$ is positive except when $r=1$
and $\cos\theta=-1.$ Similarly, $\omega(r,\theta)\leq2$ with equality only if
$r=1$ and $\cos\theta=1.$

Finally, from the definition (\ref{hAndC}) and the fact that $\lim
_{r\rightarrow0}h(r)=0,$ we find that $\lim_{r\rightarrow0}c(r)=1.$ Thus, as
$r\rightarrow0,$ we have $\tilde{\alpha}(r)\rightarrow1$ and $\tilde{\beta
}(r)\rightarrow0.$ In this limit, the fraction on the right-hand side of
(\ref{omegaFormula2}) converges uniformly to $\cos\theta,$ while $h(r)$ tends
to zero, giving Point \ref{omegaAtZero.point}.
\end{proof}

\subsection{The connection to free unitary Brownian
motion\label{bianeConnection.sec}}

Recall from Theorem \ref{t.Biane.ut} that the spectral measure $\nu_{t}$ of
the free unitary Brownian motion $u_{t}$ was computed by Biane. In this
subsection, we prove Proposition \ref{connectToBiane.prop}, which connects the
Brown measure of $b_{t}$ to Biane's measure $\nu_{t}$. The support of $\nu
_{t}$ is a proper subset of the unit circle for $t<4$ and the entire unit
circle for $t\geq4.$ For $t<4,$ the support of $\nu_{t}$ consists of points
with angles $\phi$ satisfying $\left\vert \phi\right\vert \leq\phi_{\max}(t),$
where%
\[
\phi_{\max}(t)=\frac{1}{2}\sqrt{t(4-t)}+\cos^{-1}(1-t/2).
\]

Recall the definition in (\ref{ftDef}) of the function $f_{t}$. Then $f_{t}$
maps the boundary of $\Sigma_{t}$ into the unit circle. (This is true by the
definition (\ref{FtDef}) for points in $\partial\Sigma_{t}$ outside the unit
circle and follows by continuity for points in $\partial\Sigma_{t}$ in the
unit circle.) Indeed, let the \textit{outer boundary} of $\Sigma_{t},$ denoted
$\partial\Sigma_{t}^{\mathrm{out}},$ be the portion of $\partial\Sigma_{t}$
outside the open unit disk. Then $f_{t}$ is a homeomorphism of $\partial
\Sigma_{t}^{\mathrm{out}}$ to the support of $\nu_{t}$:%
\begin{equation}
f_{t}:\partial\Sigma_{t}^{\mathrm{out}}\leftrightarrow\mathrm{supp}(\nu_{t}).
\label{ftBoundary}%
\end{equation}
In particular, for $t<4,$ let us define
\[
\theta_{\max}(t)=\cos^{-1}(1-t/2),
\]
so that the two points in $\partial\Sigma_{t}\cap S^{1}$ have angles
$\pm\theta_{\max}(t)$ (Theorem \ref{regionProperties.thm}). Then%
\[
f_{t}(e^{i\theta_{\max}(t)})=e^{i\phi_{\max}(t)},
\]
as may be verified by direct computation from the definition of $f_{t}.$ (Use
the formula (\ref{argFt}) below with $r=1$ and $\cos\theta=1-t/2.$)

We now describe the map (\ref{ftBoundary}) more concretely. We denote by
$\lambda_{t}(\theta)$ the point at angle $\theta$ in $\partial\Sigma
_{t}^{\mathrm{out}}$:%
\[
\lambda_{t}(\theta)=r_{t}(\theta)e^{i\theta},
\]
where for $t<4,$ we require $\left\vert \theta\right\vert \leq\theta_{\max
}(t).$ Then the map in (\ref{ftBoundary}) can be thought of as a map of
$\theta$ to $\phi$ determined by the relation%
\begin{equation}
f_{t}\left(  \lambda_{t}(\theta)\right)  =e^{i\phi}. \label{thetaPhi}%
\end{equation}

We now observe a close relationship between the density $w_{t}(\theta)$ in
Theorem \ref{main.thm} and the map in (\ref{thetaPhi}).

\begin{proposition}
\label{dPhiDtheta.prop}Let $\phi$ and $\theta$ be related as in
(\ref{thetaPhi}), where if $t<4,$ we require $\left\vert \phi\right\vert
\leq\phi_{\max}(t)$ and $\left\vert \theta\right\vert \leq\theta_{\max}(t).$
Then the density $w_{t}$ in Theorem \ref{main.thm} may be computed as%
\begin{equation}
w_{t}(\theta)=\frac{1}{2\pi t}\frac{d\phi}{d\theta}. \label{wtFromPhi}%
\end{equation}
We may also write this formula as a logarithmic derivative of $f_{t}$ along
the outer boundary of $\Sigma_{t}$:%
\begin{equation}
w_{t}(\theta)=\frac{1}{2\pi t}\frac{1}{i}\frac{\frac{d}{d\theta}f_{t}%
(\lambda_{t}(\theta))}{f_{t}(\lambda_{t}(\theta))}. \label{wtFromFt}%
\end{equation}

\end{proposition}

\begin{proof}
We compute that%
\[
\operatorname{Im}\left(  \frac{1+\lambda}{1-\lambda}\right)  =\frac
{2\operatorname{Im}\lambda}{\left\vert \lambda-1\right\vert ^{2}}=\frac
{2r\sin\theta}{r^{2}+1-2r\cos\theta}.
\]
Thus, using the definition (\ref{ftDef}) of $f_{t},$ we find that%
\begin{equation}
\arg(f_{t}(\lambda))=\arg\lambda+\arg e^{\frac{t}{2}\frac{1+\lambda}%
{1-\lambda}}=\theta+t\frac{r\sin\theta}{r^{2}+1-2r\cos\theta}. \label{argFt}%
\end{equation}
Evaluating this expression at the point $\lambda_{t}(\theta)$ gives%
\begin{align}
\phi &  =\arg(f_{t}(\lambda_{t}(\theta)))\nonumber\\
&  =\theta+t\frac{r_{t}(\theta)\sin\theta}{r_{t}(\theta)^{2}+1-2r_{t}%
(\theta)\cos\theta}. \label{phiOfTheta}%
\end{align}
(Strictly speaking, $\phi$ and $\theta$ are only defined \textquotedblleft mod
$2\pi,$\textquotedblright\ but for any local continuous version of $\theta,$
the last expression in (\ref{phiOfTheta}) gives a local continuous version of
$\phi.$) Thus,%
\[
d\phi=\left(  1+t\frac{d}{d\theta}\frac{r_{t}(\theta)\sin\theta}{r_{t}%
(\theta)^{2}+1-2r_{t}(\theta)\cos\theta}\right)  ~d\theta
\]
and the formula (\ref{wtFromPhi}) follows easily by recalling the definition
(\ref{wtAnswer}) of $w_{t}.$ The expression (\ref{wtFromFt}) is then obtained
by noting that $\phi=\frac{1}{i}\log f_{t}(\lambda_{t}(\theta)).$
\end{proof}

\begin{proposition}
\label{bianeMeasure.prop}Biane's measure $\nu_{t}$ may be computed as%
\begin{equation}
d\nu_{t}(\phi)=\frac{r_{t}(\theta)^{2}-1}{r_{t}(\theta)^{2}+1-2r_{t}%
(\theta)\cos\theta}\frac{d\phi}{2\pi} \label{nutRho1}%
\end{equation}
or as%
\begin{equation}
d\nu_{t}(\phi)=\frac{\log(r_{t}(\theta))}{\pi t}d\phi. \label{nutRho2}%
\end{equation}
Here, as usual, $r_{t}(\theta)$ is the outer radius of the domain $\Sigma_{t}$
and $\theta$ is viewed as a function of $\phi$ by inverting the relationship
(\ref{thetaPhi}). When $t<4,$ the formula should be used only for $\left\vert
\phi\right\vert \leq\phi_{\max}(t).$
\end{proposition}

\begin{proof}
We make use of the expression for $\nu_{t}$ in Theorem \ref{t.Biane.ut}. If
$\phi$ is in the interior of the support of $\nu_{t},$ then $\chi_{t}%
(e^{i\phi})$ is in the open unit disk, so that the density of $\nu_{t}$ is
nonzero at this point. Now, since $\chi_{t}$ is an inverse function to $f_{t}$
we see that $\chi_{t}(e^{i\phi})$ is (for $\phi$ in the interior of the
support of $\nu_{t}$) the unique point $\lambda$ with $\left\vert
\lambda\right\vert <1$ for which $f_{t}(\lambda)=e^{i\phi}.$ Thus,%
\[
d\nu_{t}(\phi)=\frac{1-1/r_{t}(\theta)^{2}}{1+1/r_{t}(\theta)^{2}-2\cos
\theta/r_{t}(\theta)}\frac{d\phi}{2\pi},
\]
which reduces to (\ref{nutRho1}). Finally, since $T(\lambda)=t$ on
$\partial\Sigma_{t}$ (Theorem \ref{domainGobbles.thm}), we have%
\begin{equation}
(r_{t}(\theta)^{2}+1-2r_{t}(\theta)\cos\theta)\frac{\log(r_{t}(\theta)^{2}%
)}{r_{t}(\theta)^{2}-1}=t, \label{TonBoundary}%
\end{equation}
which allows us to obtain (\ref{nutRho2}) from (\ref{nutRho1}).
\end{proof}

We are now ready for the proof of Proposition \ref{connectToBiane.prop}.

\begin{proof}
[Proof of Proposition \ref{connectToBiane.prop}]The distribution of
$\arg\lambda$ with respect to the Brown measure of $b_{t}$ is given in
(\ref{aTtheta}) as $2\log(r_{t}(\theta))w_{t}(\theta)~d\theta,$ which we write
using Proposition \ref{dPhiDtheta.prop} and Proposition
\ref{bianeMeasure.prop} as%
\begin{align*}
2\log(r_{t}(\theta))w_{t}(\theta)~d\theta &  =2\log(r_{t}(\theta))\frac
{1}{2\pi t}\frac{d\phi}{d\theta}~d\theta\\
&  =\log(r_{t}(\theta))\frac{d\phi}{\pi t}\\
&  =d\nu_{t}(\phi),
\end{align*}
as claimed.
\end{proof}

\begin{proof}
[Proof of Proposition \ref{connectToBiane2.prop}]The value $\Phi_{t}(\lambda)$
is computed by first taking the argument of $\lambda$ to obtain $\theta$ and
then applying the map in (\ref{thetaPhi}) to obtain $\phi.$ Thus, the first
result is just a restatement of Proposition \ref{connectToBiane.prop}. For the
uniqueness claim, suppose a measure $\mu$ on $\Sigma_{t}$ has the form
\[
d\mu(\lambda)=\frac{1}{r^{2}}g(\theta)~r~dr~d\theta.
\]
Then the distribution of the argument $\theta$ of $\lambda$ will be, by
integrating out the radial variable, $2\log(r_{t}(\theta))g(\theta)~d\theta.$
The distribution of $\phi$ will then be%
\[
2\log(r_{t}(\theta)g(\theta)\frac{d\theta}{d\phi}~d\phi=2\log(r_{t}%
(\theta)g(\theta)\frac{1}{2\pi tw_{t}(\theta)}~d\phi.
\]
The only way this can reduce to Biane's measure as computed in (\ref{nutRho2})
is if $g$ coincides with $w_{t}.$
\end{proof}

\subsection*{Acknowledgments}

The authors thank Ching-Wei Ho and Maciej Nowak for useful discussions.


\begin{thebibliography}{99}                                                                                               %


\bibitem {AS}M. Anshelevich and A. N. Sengupta, Quantum free Yang-Mills on the
plane, \textit{J. Geom. Phys.} \textbf{62} (2012), 330--343.

\bibitem {Bai}Z. D. Bai, Circular law, \textit{Ann. Probab.} \textbf{25}
(1997), 494--529.

\bibitem {BianeConvolution}P. Biane, On the free convolution with a
semi-circular distribution, \textit{Indiana Univ. Math. J.} \textbf{46}
(1997), 705--718.

\bibitem {BianeFields}P. Biane, \textquotedblleft Free Brownian motion, free
stochastic calculus and random matrices.\textquotedblright\ \textit{In}\ Free
Probability Theory (Waterloo, ON, 1995), 1--19. Fields Institute
Communications 12. Providence, RI: American Mathematical Society, 1997.

\bibitem {BianeJFA}P. Biane, Segal--Bargmann transform, functional calculus on
matrix spaces and the theory of semi-circular and circular systems, \textit{J.
Funct. Anal.}, \textbf{144} (1997), 232--286.

\bibitem {BS1}P. Biane and R. Speicher, Stochastic calculus with respect to
free Brownian motion and analysis on Wigner space, \textit{Probab. Theory
Related Fields} \textbf{112} (1998), 373--409.

\bibitem {BS2}P. Biane and R. Speicher, Free diffusions, free entropy and free
Fisher information, \textit{Ann. Inst. H. Poincar\'{e} Probab. Statist.}
\textbf{37} (2001), 581--606.

\bibitem {Br}Brown, L. G. Lidski\u{\i}'s theorem in the type II case.
\textit{In} Geometric methods in operator algebras (Kyoto, 1983), 1--35,
Pitman Res. Notes Math. Ser., 123, Longman Sci. Tech., Harlow, 1986.

\bibitem {BGNTW}Z. Burda, J. Grela, M. A. Nowak, W. Tarnowski, and P.
Warcho\l , Dysonian dynamics of the Ginibre ensemble, \textit{Phys. Rev.
Letters} \textbf{113} (2014), article 104102.

\bibitem {Ceb}G. C\'{e}bron, Free convolution operators and free Hall
transform, \textit{J. Funct. Anal.} \textbf{265} (2013), 2645--2708.

\bibitem {DN}A. Dahlqvist and J. Norris, Yang-Mills measure and the master
field on the sphere, \textit{Comm. Math. Phys.} (2020). https://doi.org/10.1007/s00220-020-03773-6

\bibitem {DH}N. Demni and T. Hamdi, Support of the Brown measure of the
product of a free unitary Brownian motion by a free self-adjoint projection,
preprint arXiv:2002.04585 [math.SP]

\bibitem {Driver89}B. K. Driver, YM$_{2}$: continuum expectations, lattice
convergence, and lassos, \textit{Comm. Math. Phys.} \textbf{123} (1989), 575--616.

\bibitem {DHKLargeN}B. K. Driver, B. C. Hall, and T. Kemp, The large-$N$ limit
of the Segal--Bargmann transform on $\mathbb{U}_{N}$, \textit{J. Funct. Anal.}
\textbf{265} (2013), 2585--2644.

\bibitem {DHK-MM1}B. K. Driver, B. C. Hall, and T. Kemp, Three proofs of the
Makeenko--Migdal equation for Yang--Mills theory on the plane, \textit{Commun.
Math. Phys.} 351 (2017), \textbf{2} (2017), 741--74.

\bibitem {DGHK-MM2}B. K. Driver, F. Gabriel, B. C. Hall, and T. Kemp, The
Makeenko--Migdal equation for Yang--Mills theory on compact surfaces,
\textit{Commun. Math. Phys.} 352 (2017), \textbf{3} (2017), 967--978.

\bibitem {Evans}L. C. Evans, Partial differential equations. Second edition.
Graduate Studies in Mathematics, 19. American Mathematical Society,
Providence, RI, 2010. xxii+749 pp.

\bibitem {FK1}B. Fuglede and R. V. Kadison, On determinants and a property of
the trace in finite factors, \textit{Proc. Nat. Acad. Sci. U. S. A.}
\textbf{37} (1951), 425--431.

\bibitem {FK2}B. Fuglede and R. V. Kadison, Determinant theory in finite
factors, \textit{Ann. of Math. (2)} \textbf{55} (1952), 520--530.

\bibitem {Gin}J. Ginibre, Statistical ensembles of complex, quaternion, and
real matrices, \textit{J. Math. Phys.} \textbf{6} (1965), 440--449.

\bibitem {Girko}V. L. Girko, The circular law. (Russian) \textit{Teor.
Veroyatnost. i Primenen.} \textbf{29} (1984), 669--679.

\bibitem {GT}F. G\"{o}tze and A. Tikhomirov, The circular law for random
matrices, \textit{Ann. Probab.} \textbf{38} (2010), 1444--1491.

\bibitem {DGross}D. Gross, Two dimensional QCD as a string theory,
\textit{Nucl. Phys. B} \textbf{400} (1993), 181--208.

\bibitem {GrossTaylor}D. Gross and W. Taylor, Two-dimensional QCD is a string
theory, \textit{Nucl. Phys. B} \textbf{400} (1993), 161--180.

\bibitem {GKS}L. Gross, C. King, and A. N. Sengupta, Two-dimensional
Yang-Mills theory via stochastic differential equations, \textit{Ann. Physics}
\textbf{194} (1989), 65--112.

\bibitem {GM}L. Gross and P. Malliavin, Hall's transform and the
Segal--Bargmann map. \textit{In} It\^{o}'s stochastic calculus and probability
theory (N. Ikeda, S. Watanabe, M. Fukushima and H. Kunita, Eds.), 73--116,
Springer, 1996.

\bibitem {Nowak}E. Gudowska-Nowak, R. A. Janik, J. Jurkiewicz, and M. A.
Nowak, Infinite products of large random matrices and matrix-valued diffusion,
\textit{Nuclear Phys. B} \textbf{670} (2003), 479--507.

\bibitem {FPZ}O. Feldheim, E. Paquette, and O. Zeitouni, Regularization of
non-normal matrices by Gaussian noise, \textit{Int. Math. Res. Not. IMRN}
\textbf{18} (2015), 8724--8751.

\bibitem {Ha1994}B. C. Hall, The Segal--Bargmann \textquotedblleft coherent
state\textquotedblright\ transform for compact Lie groups. \textit{J. Funct.
Anal.} \textbf{122} (1994), 103--151.

\bibitem {Hall2001}B. C. Hall, Harmonic analysis with respect to heat kernel
measure, \textit{Bull. Amer. Math. Soc. (N.S.)} \textbf{38} (2001), 43--78.

\bibitem {QMbook}B. C. Hall, Quantum theory for mathematicians. Graduate Texts
in Mathematics, \textbf{267}. Springer, New York, 2013. xvi+554 pp.

\bibitem {HallS2}B. C. Hall, The large-$N$ limit for two-dimensional
Yang-Mills theory, \textit{Comm. Math. Phys.} \textbf{363} (2018), 789--828.

\bibitem {PDEmethods}B. C. Hall, PDE methods in random matrix theory, preprint
arXiv:1910.09274 [math.PR].

\bibitem {HH}B. C. Hall and C.-W. Ho, The Brown measure of the sum of a
self-adjoint element and an imaginary multiple of a semicircular element,
preprint arXiv:2006.07168 [math.PR]

\bibitem {HK}B. C. Hall and T. Kemp, Brown measure support and the free
multiplicative Brownian motion, \textit{Adv. Math.} \textbf{355} (2019),
106771, 36 pp.

\bibitem {Ho}C.-W. Ho, The two-parameter free unitary Segal-Bargmann transform
and its Biane-Gross-Malliavin identification, \textit{J. Funct. Anal.}
\textbf{271} (2016), 3765--3817.

\bibitem {HZ}C.-W. Ho and P. Zhong, Brown measures of free circular and
multiplicative Brownian motions with self-adjoint and unitary initial
conditions, preprint arXiv:1908.08150 [math.OA]

\bibitem {KempLargeN}T. Kemp, The large-$N$ limits of Brownian motions on
$\mathbb{GL}_{N}$, \textit{Int. Math. Res. Not.}, (2016), 4012--4057.

\bibitem {KempHeatKernel}T. Kemp, Heat kernel empirical laws on $\mathbb{U}%
_{N}$ and $\mathbb{GL}_{N}$. \textit{J. Theoret. Probab.} \textbf{2} (2017), 397--451.

\bibitem {KNPS}T. Kemp, I. Nourdin, G. Peccati, and R. Speicher, Wigner chaos
and the fourth moment. \textit{Ann. Probab.} \textbf{40} (2012), 1577--1635.

\bibitem {Levy}T. L\'{e}vy, The master field on the plane,
\textit{Ast\'{e}risque} No. 388 (2017), ix+201 pp.

\bibitem {Lohmayer}R. Lohmayer, H. Neuberger, and T. Wettig, Possible
large-\textit{N} transitions for complex Wilson loop matrices, \textit{J. High
Energy Phys.} 2008, no. 11, 053, 44 pp.

\bibitem {MS}J. A. Mingo and R. Speicher, Free probability and random
matrices. Fields Institute Monographs, 35. Springer, New York; Fields
Institute for Research in Mathematical Sciences, Toronto, ON, 2017.

\bibitem {blueRudin}W. Rudin, Principles of mathematical analysis. Third
edition. International Series in Pure and Applied Mathematics. McGraw-Hill
Book Co., New York-Auckland-D\"{u}sseldorf, 1976.

\bibitem {Singer}I. M. Singer, On the master field in two dimensions, In:
Functional analysis on the eve of the 21st century, Vol. 1 (New Brunswick, NJ,
1993), 263--281, Progr. Math., 131, Birkhauser Boston, Boston, MA, 1995.

\bibitem {Sniady}P. \'{S}niady, Random regularization of Brown spectral
measure, \textit{J. Funct. Anal.} \textbf{193} (2002), 291--313.

\bibitem {TV}T. Tao and V. Vu, Random matrices: universality of ESDs and the
circular law. With an appendix by Manjunath Krishnapur. \textit{Ann. Probab.}
\textbf{38} (2010), 2023--2065.

\bibitem {Voiculescu}D. V. Voiculescu, Limit laws for random matrices and free
products. \textit{Invent. Math.} \textbf{104} (1991), 201--220.

\bibitem {Wigner}E. Wigner, Characteristic vectors of bordered matrices with
infinite dimensions. \textit{Ann. of Math. (2)} \textbf{62} (1955), 548--564.
\end{thebibliography}
\end{document}